\documentclass[10pt, english]{smfart}
\usepackage{amssymb}
\usepackage{amsmath}
\usepackage{amsthm}
\usepackage{mathrsfs}
\usepackage{frcursive}
\usepackage{graphicx}
\usepackage{color}
\usepackage{manfnt}
\usepackage[ansinew]{inputenc}
\usepackage[all]{xy}
\usepackage{hyperref}
\usepackage{vmargin}
\setmarginsrb{1cm}{1cm}{1cm}{1cm}{1cm}{1cm}{1cm}{1cm}
\usepackage{pdfsync}
\newtheorem{axiom}{Axiom}
\newtheorem{thm}{Theorem}[section]
\newtheorem{cor}[thm]{Corollary}

\newtheorem{lem}[thm]{Lemma}
\newtheorem{prop}[thm]{Proposition}

\theoremstyle{definition}
\newtheorem{rem}[thm]{Remark}

\newtheorem{defn}[thm]{Definition}

\newtheorem{example}[thm]{Example}

\newtheorem{notation}[thm]{Notation}
\newtheorem{notations}[thm]{Notations}

\theoremstyle{remark}
\numberwithin{equation}{section}
\newcommand{\abs}[1]{\left\vert#1\right\vert}

\newcommand{\Qp}{\mathbb Q_{p}}
\newcommand{\ord}{\mathrm{ord}\:}
\newcommand{\ac}{\overline{\mathrm{ac}}\:}

\newcommand{\Def}{\mathrm{Def}}

\newcommand{\A}{\mathbb A}
\newcommand{\Kdim}{\mathrm{Kdim}\:}

\newcommand{\Z}{\mathbb Z}

\renewcommand{\L}{\mathbb L}
\newcommand{\Bn}{ \mathscr B_{\mathfrak n}}
\newcommand{\n}{\mathfrak n}

\newcommand{\ra}{\rightarrow}

\newcommand{\RDef}{\mathrm{RDef}}
\newcommand{\I}{\mathcal I}
\newcommand{\E}{\mathcal E}
\newcommand{\N}{\mathcal N}

\newcommand{\mA}{\mathcal A}
\newcommand{\D}{\mathcal D}

\definecolor{vert}{rgb}{0,0.6,0.2}

\def\llp{\mathopen{(\!(}}

\def\rrp{\mathopen{)\!)}}

\def\SS{\mathrm{SS}}
\def\WF{\mathrm{WF}}
\def\Id{\mathrm{Id}}
\def\ccF{{\mathscr F}}

\title{Motivic wave front sets}
\date{\today}
\begin{document}
\author{Michel Raibaut}
\address{Laboratoire de Math\'ematiques\\
 Univ. Grenoble Alpes, Univ. Savoie Mont Blanc, CNRS, LAMA, 73000 Chamb\'ery, France. B\^atiment Chablais, Campus Scientifique, Le Bourget du Lac, 73376 Cedex, France}
\email{Michel.Raibaut@univ-smb.fr}
\urladdr{http://raibautm.perso.math.cnrs.fr/site/MichelRaibaut.html}

\begin{abstract}
The concept of wave front set was introduced in 1969-1970 by M. Sato in the hyperfunctions context (\cite{Sato69} and \cite{SaKaKa73}) and by L. H\"ormander (\cite{Hormander71}) in the $\mathcal C^{\infty}$ context. Howe in \cite{Howe79} used the theory of wave front sets in the study of Lie groups representations.
Heifetz in \cite{Hei85a} defined a notion of wave front set for distributions in the $p$-adic setting and used it to study some representations of $p$-adic Lie groups.

In this article, we work in the $k\llp t \rrp$-setting with $k$ a characteristic zero field.
In that setting, balls are no longer compact but working in a definable context  provides good substitutes for finiteness and compactness properties. We develop a notion of definable distributions in the framework of \cite{CluLoe08a} and \cite{CluLoe10a} for which we define notions of singular support and $\Lambda$-wave front sets (relative to some multiplicative subgroups $\Lambda$ of the valued field) and we investigate their behaviour under natural operations like pull-back, tensor product, and products of distributions.
\end{abstract}

\maketitle
\tableofcontents
%%%%%%%%%%%%%%%%%%%%%%%%%%%%%%%%%%%%%%%%%%%%%%%%%%%%%%%%%%%%%%%%%%%%%%%%%%%%

\newpage
\section{Introduction}
The concept of wave front set was introduced in 1969-1970 by M. Sato in the hyperfunctions context (\cite{Sato69} and \cite{SaKaKa73}) and by L. H\"ormander (\cite{Hormander71}) in the $\mathcal C^{\infty}$ context. For any distribution $u$ in $\mathbb R^n$, a point $x$ is said to be a smooth point of $u$ if $u$ can be represented on a neighbourhood of $x$ by a $\mathcal C^{\infty}$ function. The complement of the smooth locus is called \emph{singular support} of $u$.
The \emph{wave front set} of $u$ is a set denoted by $\WF(u)$ and contained in $T^{*}(\mathbb R^n)$, the image of its projection on $\mathbb R^n$ is the singular support of $u$ and it is conic with respect to multiplication by positive scalars in the fibers of $T^{*}(\mathbb R^n)$.
For instance, if $M$ is a submanifold of $\mathbb R^n$ and $u$ is the integration of test functions along $M$, then the singular support is $M$ and its wave front set is the conormal space of $M$ minus the zero section.
The main idea of the definition of the wave front set is the characterisation of smoothness using the Fourier transform. Indeed, the Fourier transform of a distribution with compact support is representable by a function and this distribution is smooth, namely globally representable by a $\mathcal C^{\infty}$ function or equivalently its singular support is empty, if and only if its Fourier transform is rapidly decreasing. If the distribution is not smooth then we can consider the set of critical directions where $\mathcal F u$ is not rapidly decreasing. Using this idea, a point $(x_0,\xi_0)$ with $\xi_0$ non-zero is called \emph{microlocally smooth}, if there is a neighbourhood $U_{x_0}$ of $x_0$, if there is a conical neighbourhood $\Gamma$ of $\xi_0$, such that for any test function in $\mathcal D(U_{x_0})$ the restriction of the Fourier transform of the compact support distribution $\varphi u$ is rapidly decreasing in $\Gamma$.
The wave front set of $u$ is the complement in $T^{*}(\mathbb R^n)\setminus \{0\}$ of the microlocally smooth locus.
This analysis which takes into account space variables $x$ in $\mathbb R^n$ and frequency variables $\xi$ in $(\mathbb R^{n})^*\setminus \{0\}$ is called \emph{microlocal analysis}. For instance, Kashiwara and Schapira in their treatise \cite{Kashiwara-Schapira} assign a microsupport to any sheaf on a real manifold. The concept of wave front set allowed a better understanding of operations on distributions such as product, restriction, pull-back or push-forward
(see for instance \cite{Hormander71}, \cite{Hormander83},\cite{Gabor} and \cite{Duistermaat}) and it is useful in the study of propagation of singularities by pseudo-differential operators (see \cite{Hormander83}).

Howe in \cite{Howe79} used the theory of wave front sets in the study of Lie groups representations.
Heifetz in \cite{Hei85a} studied some representations of $p$-adic Lie groups.  To do this, he defined a notion of wave front set for distributions in the $p$-adic setting. He proved the $p$-adic analogues of the usual real results, for instance the projection of the wave front set on the singular support and the construction of the pull-back of a distribution. Recently, Aizenbud and Drinfeld in \cite{AizDri} used this work to study the wave front set of the Fourier transform of algebraic measures.
The construction by Heifetz is done by analogy with the real construction.
For instance, the Lebesgue measure on $\mathbb R^n$ is replaced by the $p$-adic Haar measure on $\mathbb Q_p^n$ which is locally compact. As $\mathbb Q_p$ is totally disconnected, the test functions of $\mathcal D(\mathbb R^n)$ are replaced by Schwartz-Bruhat functions of $S(\mathbb Q_p^n)$, namely locally constant functions with compact support. A distribution is an element of the dual of $S(\mathbb Q_p^n)$.
The real exponential is replaced by an additive character on $\mathbb Q_p$ with conductor $\mathbb Z_p$. The multiplicative group $(\mathbb R^*_+,\times)$ is replaced by a finite index subgroup of $\mathbb Q_p^\times$ denoted by $\Lambda$. In particular real cones are replaced by $\Lambda$-cones. For instance, this point of view was recently used by Cluckers, Comte and Loeser in \cite{CluComLoe12} and Forey in \cite{Forey} to define a notion of tangent cones in $p$-adic and $t$-adic contexts. Finally, the notion of rapidly decreasing is replaced by the notion of bounded support. 
In \cite{CHLR}, the author with Cluckers, Halupczok and Loeser revise and generalize some of the results of Heifetz. Using $\mathcal C^{\text{exp}}$-class functions introduced in \cite{CH}, we introduce a class of distributions called \emph{distributions of $\mathcal C^{\text{exp}}$-class} which is stable under Fourier transformation and has various forms of uniform behaviour across non-archimedean local fields. Their wave front set is the complement of the zero locus of a $\mathcal C^{\text{exp}}$-class function.

 In this article we present a notion of wave front set in the motivic context, as suggested by Loeser in \cite{Loe11a}.
 Motivic integration, introduced by Kontsevich in 1995 in a Lecture in Orsay \cite{Kon95a}, is an integration theory over $k\llp t \rrp$, where $k$ is a  characteristic zero field. The field $k\llp t \rrp$ endowed with the $t$-adic topology is totally disconnected and it is not locally compact. There is no possible way to define a real Haar measure invariant by translation. The values of the motivic measure are not reals but elements, sometimes called in this context \emph{motives}, of a Grothendieck ring of varieties. This theory was developed by Denef and Loeser in a geometric way in \cite{DenLoe99a} and in an arithmetic way in \cite{DenLoe01a}.
 In this last paper some specialization results of motivic integrals on $p$-adic integrals are proved. Later, by analogy with integration of constructible functions against the Euler characteristic in real geometry, Cluckers -- Loeser generalized these works, defining in \cite{CluLoe08a} and \cite{CluLoe10a} (announced in \cite{CluLoe04a}, \cite{CluLoe04b}, \cite{CluLoe05b}) a motivic integration of constructible exponential functions with specialization results to $p$-adic integrals.
 Hrushovski and Kazhdan defined also a motivic integration theory with additive characters and definable distributions for the theory of algebraically closed valued fields of equicharacteristic zero (\cite{HruKaz06} and \cite{HruKaz08}).

 In this article we use the point of view of Cluckers -- Loeser. The measurable sets are definable parts for the Denef--Pas language of models of the theory $H_{\ac,0}$ of Henselian valued fields with residue characteristic zero and discrete value group. We recall in section \ref{rappelCL} some of the constructions of \cite{CluLoe08a} and \cite{CluLoe10a} used all along the paper. In particular the notions of constructible exponential functions, motivic integrals and motivic Fourier transforms. 
 In section \ref{definable distributions} we define a notion of \emph{definable distributions} and the Fourier transform of a definable distribution. Even if the motivic Schwartz-Bruhat functions are not finite linear combinations of characteristic functions of balls, a definable distribution is determined by just its values on characteristic functions of balls (Theorem \ref{extension-theorem}).
 In section \ref{section-microlocal} we define the notions of singular support and \emph{$\Lambda$-motivic wave front set} of a definable distribution where $\Lambda$ is a definable multiplicative subgroup of the valued field. In Example \ref{wf-variete} we describe the wave front set of the definable distribution induced by a definable set which is locally a graph of a definable function. We prove results on the projection of the wave front set (Theorem \ref{projectionWF}),
  pull-back of a definable distribution (Theorem \ref{thmf^{*}u}), tensor product and product (Definition \ref{produit-tensoriel} and \ref{produit}), analogous to the classical real and $p$-adic results. These proofs use in the real and $p$-adic settings a stationary phase formula. We give its motivic version in section \ref{integrales-oscillantes} (Proposition \ref{oscillante}). The classical settings use also in a crucial way the compactness of the sphere (real or $p$-adic).
 In our context, where the $t$-adic spheres are not compact, finiteness results come from
 the definability of our objects. In particular we prove in the section \ref{dc} (Proposition \ref{definablecompactness}), that any definable and continuous function defined on a bounded and closed subset of $k\llp t \rrp^n$ with integer values takes finitely many values.

\section{Motivic integration and constructible motivic functions}  \label{rappelCL}
For the reader's convenience
we shall start by recalling briefly some definitions, notations and constructions from \cite{CluLoe08a} and
\cite{CluLoe10a} that will be used in this article.  For an introduction to this circle of ideas  we refer to the surveys \cite{CluLoe05a}, \cite{CluHalLoe11} and \cite{GorYaf09} and the notes \cite{CluLoe04a}, \cite{CluLoe04b} and \cite{CluLoe05b}. 

\subsection{Denef-Pas, Presburger language}
We fix a field $k$ of characteristic zero and we denote by $\mathrm{Field}_{k}$ the category of fields containing $k$.
For any field $K$ in this category we consider the field of Laurent series
$K(\!(t)\!)$ endowed with its natural \emph{valuation}
$$\ord : K(\!(t)\!)\setminus\{0\} \longrightarrow \mathbb Z$$
extended by $\ord 0=+\infty$,
and with the \emph{angular component} mapping
$$\ac:K(\!(t)\!)\ra K$$
defined by $\ac(x)=xt^{-\ord x}\mod t$ if $x\neq 0$ and $\ac (0) =0$.\\

We shall use the three sorted language introduced by Denef and Pas in \cite{Pas89}
$$\mathscr L_{DP,P} = (\mathbf{L}_{\mathrm{Val}},\mathbf{L}_{\mathrm{Res}},\mathbf{L}_{\mathrm{Ord}},\ord, \ac)$$
with sorts corresponding respectively to \emph{valued field}, \emph{residue field} and
\emph{value group} variables.
The languages $\mathbf{L}_{\mathrm{Val}}$ and $\mathbf{L}_{\mathrm{Res}}$ are the ring language
$\mathbf{L}_{\mathrm{Rings}}=(+,-, \cdot ,0,1)$ and the language $\mathbf{L}_{\mathrm{Ord}}$ is the Presburger language
$$\mathbf{L}_{\mathrm{PR}}=\{+,-,0,1,\leq\}\cup\{\equiv_{n}\mid n\in \mathbb N, n>1\},$$
with $\equiv_{n}$  symbols interpreted as  equivalence relation modulo $n$. Symbols $\ord$and $\ac$will be interpreted respectively as valuation and angular component, so that for any $K$ in $\mathrm{Field}_{k}$ the triple $(K(\!(t)\!),K,\mathbb Z)$ is a structure for $\mathscr L_{DP,P}$. We shall also add
constant symbols in the $\mathrm{Val}$-sort and in the $\mathrm{Res}$-sort for elements of $k(\!(t)\!)$, resp.~of $k$.\\

We will work with the $\mathcal{L}_{DP,P}$-theory $H_{\ac,0}$ of  structures whose valued field is Henselian, with residue field characteristic zero, and with value group $\Z$.
Denef and Pas proved in \cite{Pas89} the following theorem on elimination of valued field quantifiers.

\begin{thm}[Denef-Pas \cite{Pas89}, Presburger \cite{Presb29}]  \label{QE}
Every formula $\phi(x,\xi,\alpha)$ without parameters in the $\mathscr L_{DP,P}$-language, with $x$ variables in the $\mathrm{Val}$-sort, $\xi$ variables in the $\mathrm{Res}$-sort and $\alpha$ variables in the
$\mathrm{Ord}$-sort is $H_{\ac,0}$-equivalent to a finite disjunction of formulas of the form
$$\psi(\ac f_{1}(x),...,\ac f_{k}(x),\xi)\land \eta(\ord f_1(x),...,\ord f_k(x),\alpha),$$
with $\psi$ a $\mathbf{L}_{\mathrm{Res}}$-formula, $\eta$ a $\mathbf{L}_{\mathrm{Ord}}$-formula without quantifiers and $f_1,...,f_k$ polynomials in $\Z[x]$. The theory $H_{\ac,0}$ admits elimination of quantifiers in the valued field sort.
\end{thm}

\subsection{Definable subassignments}

From now on we will work with the Denef-Pas language enriched with constant symbols in the $\mathrm{Val}$-sort and in the $\mathrm{Res}$-sort for elements of $k(\!(t)\!)$, resp.~of $k$, and we will denote this language also by $\mathcal{L}_{DP,P}$.

\subsubsection{Definable subassignments and definable morphisms}
Let $\varphi$ be a formula respectively in $m$, $n$ and $r$ free variables in the various sorts. For every field $K$ in
$\mathrm{Field}_{k}$, we denote by $h_{\varphi}(K)$ the subset of
$$h[m,n,r](K):=K(\!(t)\!)^{m}\times K^{n} \times \mathbb Z^{r}$$
consisting of points satisfying $\varphi$. The assignment $K\mapsto h_{\varphi}(K)$ is called a
\emph{definable subassignment} or \emph{definable set}. For instance we will denote by $\{*\}$ the definable subassignment $h[0,0,0]$ defined by $K \mapsto \mathrm{Spec}\: K$.
A \emph{definable morphism} $F$ between two definable subassignments $h_{\varphi}$
and $h_{\psi}$ is a collection of applications parametrized by $K$ in $\mathrm{Field}_{k}$ $$F(K):h_{\varphi}(K)\ra h_{\psi}(K)$$ such that the graph map
$K\mapsto \mathrm{Graph}F(K)$ is a definable subassignment. Definable subassignments and definable morphisms are precisely objects
and morphisms of the category of definable subassignments over $k$ denoted by $\mathrm{Def}_{k}$.
More generally, for any definable subassignment $S$ in $\Def_{k}$, we will consider the category $\Def_S$ of definable subassignments over $S$ whose objects are definable morphisms
$\theta_{Z}$ in $\Def_{k}$ from a definable $Z$ to $S$ and morphisms are definable maps $g:Y\ra Z$ such that
$\theta_Y = \theta_Z \circ g$.

\subsubsection{Finiteness of some definable functions}
We deduce from  Theorem \ref{QE} on quantifier elimination the following corollary which will be crucial in the proof of Proposition \ref{definablecompactness}  on definable compactness.
\begin{cor} \label{fini}
For non negative integers $m$,$n$ and $r$, every definable map from $h[0,n,0]$ to $h[m,0,r]$ or from $h[0,0,r]$ to $h[m,n,0]$ or from $h[0,n,r]$ to $h[m,0,0]$ takes finitely many values.
\end{cor}

\subsubsection{Points and fibers}
A \emph{point} $x$ of a definable set $X$ is by definition a couple $x=(x_{0},K)$ where $K$ is an extension of $k$ and $x_{0}$ is a point of $X(K)$. The field $K$ will be denoted by $k(x)$ and called
\emph{residue field} of $x$.

Let $f$ be a definable morphism from a subassignment of $h[m,n,r]$ denoted by  $X$ to a subassignment of $h[m',n',r']$ denoted by $Y$. Let $\varphi(x,y)$ be the formula which describes the graph of $f$, where $x$ runs over $h[m,n,r]$ and $y$ runs over $h[m',n',r']$.
For every point $y=(y_{0},k(y))$ of $Y$, the \emph{fiber} $X_y$ is the object of $\Def_{k(y)}$ defined by the formula $\varphi(x,y_0)$ which has coefficients in $k(y)$ and $k(y)(\!(t)\!)$. Taking fibers at $y$ gives rise to a functor $i_{y}^{*}:\Def_{Y}\ra \Def_{k(y)}$.

\subsubsection{Dimension}
For any positive integer $m$, an algebraic subvariety $Z$ of $\A_{k(\!(t)\!)}^{m}$ induces a definable subassignment $h_{Z}$ of $h[m,0,0]$ with $h_{Z}(K)$ equal to $Z(K(\!(t)\!))$ for any extension $K$ of $k$.
The \emph{Zariski closure} of a subassignment $S$ of $h[m,0,0]$ is by definition the subassignment of the intersection $W$ of all algebraic subvarieties $Z$ of $\A_{k(\!(t)\!)}^{m}$ such that $h_{Z}$ contained $S$. The \emph{dimension} $\Kdim S $ of $S$ is naturally defined as $\dim W$.
More generally, the dimension of a subassignment $S$ of $h[m,n,r]$ is defined as the dimension $\Kdim p(S)$ where $p$ is
the projection from $h[m,n,r]$ to $h[m,0,0]$. It is  proved in \cite{CluLoe08a},
using results of Pas \cite{Pas89} and van den Dries \cite{VDDries89}, that
 isomorphic definable subassignments in $\Def_{k}$ have the same dimension.

\subsection{Grothendieck rings and exponentials}

\subsubsection{The category $\RDef_{k}^{\exp}$}
For any definable subassignment $Z$ in $\Def_{k}$,
the  subcategory $\RDef_Z$ of $\Def_{Z}$ whose
objects are definable morphisms $\pi_Y$ with $Y$ a subassignment of a product
$Z\times h[0,n,0]$, with $n$ a non negative integer and $\pi_Y$ the canonical projection on $Z$,
has been introduced in \cite{CluLoe08a}.

\begin{example}
If $Z$ is the point $h[0,0,0]$, then the subcategory $\RDef_Z$ is the category of definable sets in the ring language with coefficients from $k$.
\end{example}

More generally, in \cite{CluLoe10a} \emph{motivic additive characters} were considered in this context through the category $\RDef_{Z}^{\:\exp}$ whose objects are triples $(\pi_Y,\xi, g)$ with $\pi_Y$ a definable set in $\RDef_{Z}$, $\xi$ a definable morphism from $Y$ to $h[0,1,0]$ and $g$ a definable morphism from $Y$ to $h[1,0,0]$. A morphism from $(\pi_{Y'},\xi', g')$ to $(\pi_Y,\xi, g)$ in
$\RDef_{Z}^{\:\exp}$ is a morphism $h$ from $Y'$ to $Y$ satisfying the equalities
$$\pi_{Y'}=\pi_{Y}\circ h,\:\: \xi'=\xi \circ h,\:\:g'= g \circ h.$$

\begin{rem}\label{rem:1.4}The functor $\pi_Y \mapsto (\pi_Y,0,0)$
allows to identify
 $\RDef_{Z}$ as  a full subcategory of $\RDef_{Z}^{\:\exp}$. \end{rem}

\subsubsection{The Grothendieck ring $K_{0}(\RDef_{Z}^{\:\exp})$}  \label{exponential}
As an abelian group it is the free abelian group over symbols $[\pi_Y,\xi,g]$
modulo the following relations:

 \emph{Isomorphism}.
 For any isomorphic $(\pi_Y,\xi, g)$ and $(\pi_{Y'},\xi', g')$, we consider the relation
 \begin{equation} \tag{R1}
  [\pi_Y,\xi,g]=[\pi_{Y'},\xi',g']
 \end{equation}

\emph{Additivity}.
 For $\pi_Y$ and $\pi_{Y'}$ definable subassignments of some $\pi_X$ in $\RDef_{Z}$ and for $\xi$ and $g$ defined on the union $Y\cup Y'$, we consider the relation
 \begin{equation} \tag{R2}
 [\pi_{Y\cup Y'},\xi,g]+[\pi_{Y\cap Y'},\xi_{\mid Y \cap Y'},g_{\mid Y \cap Y'}]
 =[\pi_{Y},\xi_{\mid Y},g_{\mid Y}]+[\pi_{Y'},\xi_{\mid Y'},g_{\mid Y'}].\end{equation}

 \emph{Compatibility with reduction}.
 For any $\pi_Y$ in $\Def_Z$, for any definable morphism $f$ from $Y$ to $h[1,0,0]$ with $\ord f(y) \geq 0$ for any $y\in Y$,  we consider the relation
 \begin{equation} \label{R3} \tag{R3}
  [\pi_Y,\xi,g+f]=[\pi_Y,\xi+\overline{f},g]
 \end{equation}
 with $\overline{f}$ the reduction
 of $f$ modulo $(t)$.\\

 \emph{Sum over the line}.
 Let $p$ be the canonical projection from $Y[0,1,0]$ to $h[0,1,0]$. If the morphisms $\pi_{Y[0,1,0]}$, $g$ and $\xi$ all factorize through the canonical projection from $Y[0,1,0]$ to $Y$, then we consider the relation
 \begin{equation} \label{R4} \tag{R4}
 [Y[0,1,0]\ra Z,\xi+p,g]=0.
 \end{equation}

 This Grothendieck group is endowed with a ring structure by setting
\begin{equation} \label{R5} \tag{R5}
 [\pi_Y,\xi,g] \cdot [\pi_{Y'},\xi',g']=[\pi_{Y\otimes_Z Y'},\xi\circ p_Y+\xi'\circ p_Y',g\circ p_Y+g'\circ p_Y']
\end{equation}
where $Y\otimes_Z Y'$ is the fiber product of $Y$ and $Y'$ above $Z$, $p_Y$ is the projection to $Y$ and $p_{Y'}$ is the projection to $Y'$. The element $[\Id_Z,0,0]$ is the multiplicative unit of $K_{0}(\RDef_Z^{\:\exp})$.
The Grothendieck ring $K_{0}(\RDef_Z)$ is defined as above and the functor defined in Remark \ref{rem:1.4} induces an injection $K_{0}(\RDef_Z) \to K_{0}(\RDef_Z^{\:\exp})$.

\begin{rem}
The element $[\pi_Y,\xi, g]$ of the Grothendieck ring $K_{0}(\RDef_{Z}^{\:\exp})$ 
will be denoted by $e^{\xi}E(g)[\pi_Y]$.  We will abbreviate $e^{0}E(g)[\pi_Y]$ by $E(g)[\pi_Y]$, $e^{0}E(0)[\pi_Y]$
by $[\pi_Y]$ and $e^{0}E(g)[\Id_Z]$ by $E(g)$.
\end{rem}

\begin{rem}[Interpretation of $E$] The element $E(g)$ in $K_{0}(\RDef_{Z}^{\:\exp})$ can be viewed as the exponential  (at the valued field level of the definable morphism $g$ from $Z$ to $h[1,0,0]$, said otherwise, it is a motivic additive character on the valued field evaluated in $g$. More precisely, by relations (\ref{R3}) and (\ref{R5}), $E$ can be interpreted as a universal additive character  which is trivial on the maximal ideal of the valuation ring. This is compatible with specialization to $p$-adic fields as explained in  Section 9 of \cite{CluLoe10a}.
\end{rem}

\begin{rem}[Interpretation of $e$]
 The element $e(\xi)$ in $K_{0}(\RDef_{Z}^{\:\exp})$ can be considered as the exponential (at the residue field level) of the definable morphism $\xi$ from $Z$ to $h[0,1,0]$.  By relation (\ref{R4}), $e$ can  be interpreted as a universal additive character on the residue field. For instance in the case where $Z$ is the point, the relation $[h[0,1,0]\ra \{*\},p,0] = 0$ should be interpreted as an abstraction of the classical nullity of the sum of a non trivial character over elements of a finite field. Relation (\ref{R3}) expresses compatibility under reduction modulo the uniformizing parameter between the exponentials over the valued field and over the residue field.
\end{rem}

\subsection{Constructible exponential functions}

\subsubsection{Constructible motivic functions}
In the $p$-adic context (see \cite{Denef84}, \cite{Igusabook}, \cite{CluLoe10a} and \cite{CluGorHal14a}), for instance over the field $\mathbb Q_{p}$ itself, one fixes $\Psi : K \ra \mathbb C^\times$ an additive character  trivial on $p\mathbb Z_p$ and non trivial on the set $\ord x =0$ and one denotes by $\mathbb A_{p}$ the ring
$\mathbb Z[1/p, 1/(1-p^{-i})]$.
For any $X$ contained in some $\mathbb Q_p^m$ and definable for the Macintyre language, it is natural to define the $\mathbb A_{p}$-algebra of constructible functions on $X$ denoted by $\mathcal C(X)$ and generated by function of the form
 $\abs{f}\ord(h)$ where $f$ and $h$ are definable functions from $X$ to $\mathbb Q_p$ and $h$ does not vanish. In \cite{CluLoe10a}, also a variant with additive characters is introduced, called constructible exponential functions on $X$ and denoted by $\mathcal C(X)^{\exp}$. The algebra $\mathcal C(X)^{\exp}$ is generated by $\mathcal C(X)$  and functions of the form $\psi(g)$ with $g:X\to\Qp$ with $\psi$ a non-trivial additive character on $\Qp$.   
Analogously, Cluckers -- Loeser consider in \cite{CluLoe08a} the ring
$$\mathbb A=\mathbb Z\left[\L,\L^{-1},\left(\frac{1}{1-\L^{-i}}\right)_{i>0}\right],$$
where $\L$ is a symbol, and they define the ring $\mathcal C(Z)$ of
\emph{constructible motivic functions} on a definable set $Z$ by
$$\mathcal C(Z):=K_{0}(\RDef_{Z}) \otimes_{\mathcal{P}^{0}(Z)} \mathcal{P}(Z),$$
where $\mathcal P(Z)$, called ring of \emph{Presburger constructible functions}, is the subring of the ring of functions from the set of points of $Z$ to $\mathbb A$, generated
by constant functions, definable functions from $Z$ to $\mathbb Z$ and functions of the form $\L^{\beta}$ with
$\beta$ a definable function from $Z$ to $\mathbb Z$. Here,  $\mathcal{P}^{0}(Z)$ is the subring of $\mathcal P(Z)$ generated by the constant function $\L$ and the characteristic functions $1_{Y}$ of definable subsets $Y$ of the base $Z$.
The tensor product is given by the morphism from $\mathcal{P}^{0}(Z)$ to $K_{0}(\RDef_{Z})$ sending $1_{Y}$ to the class $[Y \ra Z]$
of the canonical injection from $Y$ to $Z$ and sending $\L$ to the class $[Z[0,1,0]\ra Z]$ of the canonical projection to $Z$.

\subsubsection{Constructible exponential functions}
For any definable set $Z$ in $\Def_{k}$, the ring $\mathcal C(Z)^{\exp}$ of \emph{constructible exponential functions} is defined in
\cite{CluLoe10a} by
$$\mathcal C(Z)^{\exp}:=\mathcal C(Z)\otimes_{K_{0}(\RDef_{Z})}K_{0}(\RDef_{Z}^{\exp}),$$
where we use the morphism $a \mapsto a \otimes 1_Z$ from $K_{0}(\RDef_{Z})$ to $\mathcal C(Z)$.
For any integer $d$, we denote by $\mathcal C^{\leq d}(Z)^{\exp}$ the ideal generated by the characteristic functions $1_{Z'}$ of subassignments $Z'$ of $Z$ of dimension at most $d$. This family of ideals is a filtration of the ring $\mathcal C(Z)^{\exp}$ and
the graded ring associated
$$C(Z)^{\exp} = \oplus_{d\in \mathbb N}\mathcal C^{\leq d}(Z)^{\exp}/\mathcal C^{\leq d-1}(Z)^{\exp}$$
is called ring of \emph{constructible exponential $\ccF$-unctions}.

\begin{rem}
 Constructible $\ccF$-unctions can be compared to the equivalence classes of Lebesgue measurable functions (equality up to a zero measure set). In this article we will just write function for $\ccF$-unction; the difference still being visible in the notation $C(Z)^{\exp}$ versus $\mathcal C(Z)^{\exp}$.
\end{rem}

\subsection{Pull-back of constructible exponential functions}  \label{inverseimage-exp}
A definable map $f:Z\ra Z'$ in $\Def_{k}$ induces a pull-back morphism (cf. \S 5.4 in \cite{CluLoe08a}and \S 3.4 in \cite{CluLoe10a})
$$f^{*}:\mathcal C(Z')^{\exp} \ra \mathcal C(Z)^{\exp}.$$
Indeed, the fiber product along $f$ induces a pull-back morphism
 from $K_{0}(\RDef_{Z'})^{\exp}$ to $K_{0}(\RDef_Z)^{\exp}$ and the composition by $f$ induces also a pull-back morphism
from $\mathcal{P}(Z')$ to $\mathcal{P}(Z).$ These pull-backs are compatible with their tensor product.

\begin{rem}
 A constructible exponential function $E(g)e(\xi) \otimes \alpha \mathbb L^{\beta}$ can be thought of as
 $$z \in Z \mapsto E(g(z))e(\xi(z)) \otimes \alpha(z) \mathbb L^{\beta(z)} \in \mathcal C(\{z\})^{\exp}.$$
 More generally, the constructible exponential function $[\pi_Y] E(g)e(\xi) \otimes \alpha \mathbb L^{\beta}$ can be thought of as
 $$z \in Z \mapsto [Y_z]E(g_{\mid Y_z})e(\xi_{\mid Y_z}) \otimes \alpha_{\mid Y_z} \mathbb L^{\beta_{\mid Y_z}} \in \mathcal C(\{z\})^{\exp}.$$
 By Corollary \ref{fini}, the restrictions $\alpha_{\mid Y_z}$ and $\beta_{\mid Y_z}$ take finitely many values, but $[Y_z]$ should be considered as a kind of motive standing for a possibly infinite sum over elements in $Y_z$, which is a definable subset of some power of the residue field. With $E$ and $e$, the expression $[\pi_Y] E(g)e(\xi) $ is a kind of exponential motive, standing for possibly infinite exponential sums. In the $p$-adic case, the finiteness of the residue field allows one to see $[Y_z]$ as a finite sum again.
\end{rem}

\subsection{Push-forward of constructible exponential functions}  \label{push forward}
For $S$ in $\Def_{k}$, Cluckers -- Loeser construct in
\cite{CluLoe08a, CluLoe10a}
a functor $I_{S}^{\exp}$ from the category $\Def_{S}$ to the category $\underline{\mathrm{Ab}}$
 of abelian groups:
 $$  I_{S}^{\exp} \colon
 \begin{cases}
\Def_{S}  \longrightarrow  \underline{\mathrm{Ab}} \\
( \theta_Z : Z \to  S)  \longmapsto   (I_{S}C(\theta_Z)^{\exp} \subset C(Z)^{\exp}) \\
( \theta_Z \overset{f} \to \theta_Y)  \longmapsto  (I_{S}C(\theta_Z)^{\:\exp}\overset{f!}{\ra} I_{S}C(\theta_Y)^{\exp})
 \end{cases}
$$
satisfying natural axioms implying its uniqueness, see Theorems 10.1.1 and 13.2.1 in \cite{CluLoe08a} and
Theorem 4.1.1 in \cite{CluLoe10a}.
The elements of  $I_{S}C(\theta_Z)^{\exp}$ are called \emph{$\theta_Z$-integrable motivic constructible exponential functions} on $Z$ or simply \emph{$\theta_Z$-integrable functions}.
\begin{example} The ring $I_{S}C(\Id_S)^{\exp}$ is all of $C(S)^{\exp}$, namely, every function in $C(S)^{\exp}$ is already integrable up to $S$ itself, with the identity map $S\to S$ as structural morphism.
\end{example}

\begin{rem} We will often simply say $S$-integrable instead of $\theta_Z$-integrable and write $I_{S}C(Z)^{\exp}$ when the structural morphism $\theta_Z$ is implicitly clear.
\end{rem}

The functor $I_{S}^{\exp}$ and the integrable functions are constructed simultaneously.
The functor $I_{S}$ is first defined in \cite{CluLoe08a} in the setting without exponential and extended in \cite{CluLoe10a} in the exponential setting to $I_{S}^{\exp}$. In particular, for any $Z$ in
$\Def_{S}$, $I_{S}C(Z)^{\exp}$ is a graded subgroup of $C(Z)^{\exp}$ defined as
\begin{equation} \label{defexp} I_{S}C(Z)^{\exp}:=I_{S}C(Z)\otimes_{K_{0}(\RDef_Z)} K_{0}(\RDef_Z)^{\exp}, \end{equation}
where
$I_{S}C(Z)$ is a graded subgroup of $C(Z)$ called the group of
\emph{$S$-integrable constructible functions on $Z$}.
The natural morphism of graded groups from $I_{S}C(\theta_Z)$ to
$I_{S}C(\theta_Z)^{\exp}$ is injective.
We will use the following axioms  (see Theorem 10.1.1 in \cite{CluLoe08a} and \S 13.2 in \cite{CluLoe10a}):

\begin{axiom}[Fubini]
Let $S$ be in $\Def_{k}$. Let $f:\theta_X\ra \theta_Y$ be a definable morphism in $\Def_S$.
A constructible function $\varphi$ on $X$ is $\theta_X$-integrable if and only if $\varphi$ is $f$-integrable and
$f_{!} \varphi$ is $\theta_Y$-integrable namely:
$$\varphi \in I_{S}C(\theta_X)^{\exp} \Leftrightarrow  \varphi \in I_{Y}C(f)^{\exp}\:\:\mathrm{and}\:\: f_{!}\varphi \in I_{S}C(\theta_Y)^{\exp}.$$
\end{axiom}

\begin{axiom}[Projection formula]  \label{axiom-projection}
Let $S$ be in $\Def_k$. For every morphism $f: \theta_Z \ra \theta_Y$ in $\Def_S$, and every $\alpha$ in
$\mathcal C(Y)^{\exp}$ and $\beta$ in $I_{S}C(\theta_Z)^{\exp}$, if $f^{*}(\alpha)\beta$ belongs to $I_S C(\theta_Z)^{\exp}$, 
then $f_{!}(f^{*}(\alpha)\beta)=\alpha f_{!}(\beta)$.
\end{axiom}

\begin{axiom}[Volume of balls]  \label{volume-boule}
 Let $\theta_Y$ be in $\Def_{S}$, and
 $$Z=\{(y,z)\in Y[1,0,0] \mid \ord(z-c(y))=\alpha(y),\: \ac(z-c(y))=\xi(y)\}$$
 where $\alpha:Y\ra \mathbb Z$,
 $\xi:Y \ra h[0,1,0]\setminus \{0\}$ and
 $c:Y \ra h[1,0,0]$ are definable morphisms. Denote by
 $f:Z\ra Y$ the morphism induced
 by the projection from $Y \times h[1,0,0]$ to $Y$, and $\theta_Z$ its composition with $\theta_Y$. Then, the constructible function $[1_Z]$ is
 $\theta_Z$-integrable if and only if
 $\mathbb L^{-\alpha-1}[1_{Y}]$ is $\theta_Y$-integrable. In that case, in the ring $I_SC(Y)^{\exp}$ we have the equality $f_{!}([1_Z]) = \mathbb L^{-\alpha -1}[1_{Y}]$.
\end{axiom}

 By Axiom \ref{volume-boule}, the volume of a ball
 $\{z \in h[1,0,0] \mid \ord(z-c)=\alpha, \ac(z-c)=\xi\}$
 with $\alpha$ in $\mathbb Z$, $c$ in $k\llp t \rrp$ and $\xi$ in $k$, $\xi\neq 0$ is $\mathbb L^{\alpha-1}$.
 This is natural by analogy with the $p$-adic case.

\begin{axiom}[Relative balls of large volume]
 Let $\theta_Y$ be in $\Def_{S}$ and
 $$Z=\{(y,z)\in Y[1,0,0] \mid \ord z =\alpha(y),\: \ac z =\xi(y)\}$$
 where $\alpha:Y\ra \mathbb Z$, $\xi:Y \ra h[0,1,0]\setminus \{0\}$ are definable morphisms. Let $f:Z\ra Y$ be the morphism induced by the projection from $Y[1,0,0]$ to $Y$.
 Suppose that the constructible function $[1_{Z}]$ is $(\theta_Y \circ f)$-integrable and moreover $\alpha(y)<0$ holds for every $y$ in $Y$, then
 $f_{!}(E(z)[1_Z])=0.$
\end{axiom}

The previous axiom is also natural by analogy with the $p$-adic case, where an additive character evaluated in the identity function and integrated over a large ball is naturally zero.

  \begin{thm}[Change of variables formula, Theorem 5.2.1 of \cite{CluLoe10a}]\label{cov}
   Let $f:X\ra Y$ be a definable isomorphism between definable subassignments of dimension $d$. Let $\varphi$
   be in $\mathcal{C}^{\leq d}(Y)^{\exp}$ with a nonzero class in $C^{d}(Y)^{\exp}$. Then $[f^{*}(\varphi)]$ belongs to 
   $I_{Y}C^{d}(f)^{\exp}$ and
   $$f_{!}([f^{*}(\varphi)])=\mathbb L^{\ord (\mathrm{Jac} f) \circ f^{-1}}[\varphi].$$
  \end{thm}
  We give some ideas of the construction of this pushforward and refer to \cite{CluLoe08a} and \cite{CluLoe10a} and to the surveys \cite{CluHalLoe11}, \cite{CluLoe05a} and \cite{GorYaf09} for further details. For instance, we fix a base $S$, we consider a definable morphism $f:Y\ra S$ where $Y$ is a subassignment of some $h[m,n,r]$ and we denote by $\Gamma_f$ the graph of $f$. By fonctoriality the morphism $f_!$ is the composition $p_! \circ i_!$ where $i : Y\ra \Gamma_f$ and $p:\Gamma_f \ra S$ are the canonical injection and projection. Thus, it is enough to know how to construct the push-forward morphisms for injections and projections. The case of definable injection is done using extension by zero of constructible functions, and an adjustment with a Jacobian to match the induced measures.
  Using the axiom of the volume of balls and the change of variables formula we observe that the construction of the push-forward morphism for a projection is done by induction on the valued field dimension. For instance, $\Gamma_f$ can be seen as a definable subassignment of $S'[1,0,0]$ where $S'$ is the definable set $S[m-1,n,r]$ and the push-forward $p_{!}$ will be the composition $p^{(m-1)}_!\circ  \pi_{!}$ where $\pi : \Gamma_f \ra S'$ and $p^{(m-1)} : S' \ra S$ are canonical projections. The construction does not depend on the order of such projections and the main tool is the  cell decomposition theorem restated below. Once the valuative dimension is zero we have to define a push-forward of a projection from some $S[0,n',r']$ to
  $S$. This is done using the indepedance between the residue field and the value group, coming from Theorem \ref{QE}. The push-forward along residue variables is simply the push-forward induced by composition at the level of Grothendieck ring cf. \cite{CluLoe08a} \S 5.6. The integration along $\mathbb Z$-variables corresponds to summing over the integers, cf. \cite{CluLoe08a} \S 4.5.

  \begin{rem}[On the projection axiom] \label{remproj}
	  In \cite[Proposition 13.2.1,(2)]{CluLoe10a}, it is proved that for a definable morphism $f:X\to S$, $\alpha$ in $\mathcal C(S)$,
	  and $\beta$ in $I_{S}C(X)$, the constructible function $f^{*}(\alpha)\beta$ is $S$-integrable and 
	  $$f_{!}(f^{*}(\alpha)\beta)=\alpha f_{!}(\beta).$$
	  This result can be extend to the exponential case using the definition given by the formula \ref{defexp} and applying the point 
	  $(2)$ of Proposition 13.2.1 in \cite{CluLoe10a}. 
  \end{rem}
  \subsection{Cell decomposition}  \label{celldecomppart}

  Let $C$ be a definable subassignment in $\Def_{k}$. Let $\alpha : C \ra \Z$, $\xi : C\ra h[0,1,0]\setminus \{0\}$ and $c:C\ra h[1,0,0]$ be definable morphisms.\\

  $\bullet$ The cell $Z_{C,c,\alpha,\xi}$  with basis $C$, center $c$, order $\alpha$ and angular component $\xi$, is
  $$
  Z_{C,c,\alpha,\xi} =
  \left\{
  (y,z) \in C[1,0,0]\:
  \begin{array}{|l}
   \ord(z-c(y))=\alpha(y) \\
   \ac(z-c(y))=\xi(y)
  \end{array}
  \right\}
  $$
Note that this definable set is a family of balls $B(c(y)+\xi(y)t^{\alpha(y)},\alpha(y)+1)$ parametrized by the base $C$. The axiom \ref{volume-boule} gives the push-forward morphism corresponding to the projection of this cell on its base $C$, that is, integration in the fibers of this projection map.\\

$\bullet$ The cell $Z_{C,c}$ with basis $C$ and center $c$ is
$$
  Z_{C,c} =
  \left\{
  (y,z) \in C[1,0,0]\mid z=c(y)
  \right\}.
  $$
The change of variables formula gives in particular, the push-forward morphism corresponding to the projection of that cell on its base.
More generally, a definable subassignment $Z$ of $S[1,0,0]$ for some $S$ in $\Def_{k}$ is called a \emph{$1$-cell} or a \emph{$0$-cell} if there exists a definable isomorphism
$$\lambda : Z\ra Z_{C,c,\alpha,\xi}\subset S[1,s,r]\:\:\mathrm{or}\:\: \lambda : Z\ra Z_{C,c}\subset S[1,s,0],$$
called \emph{presentation} of the cell $Z$,
where the base $C$ is contained in $S[0,s,r]$ and such that the morphism $\pi \circ \lambda$  is the identity on $Z$ with $\pi$ the projection to $S[1,0,0]$.

Let us state  a variant of Denef-Pas Cell Decomposition theorem \cite{Pas89}, Theorem 7.2.1 of \cite{CluLoe10a}, that will be used in the proof of the definable compactness Proposition \ref{definablecompactness}.

\begin{thm}[Cell decomposition]  \label{celldecompthm}
 Let $X$ be a definable subassignment of $S[1,0,0]$ with $S$ in $\Def_{k}$.
 \begin{enumerate}
  \item The definable set $X$ is a finite disjoint union of cells.
  \item For every $\varphi$ in $\mathcal C(X)$ there exists a finite partition of $X$ into cells $Z_{i}$ with presentation
  $(\lambda_{i},Z_{C_{i}})$ and
  $\varphi_{\mid Z_{i}}=\lambda_{i}^{*}p_{i}^{*}(\psi_{i}),$ with
  $\psi_{i}$ in $\mathcal C(C_{i})$ and $p_{i}:Z_{C_{i}} \ra C_{i}$ the projection. Similar statements hold for $\varphi$
  in $\mathcal P(X)$ and in $K_{0}(\RDef_{X})$.
 \end{enumerate}
\end{thm}

\subsection{A compatibility relation between pull-back and push-forward}

\begin{prop}  \label{compatibilite-pull-back-push-forward}
	Let $S$ be a definable set in $\Def_k$ and $g : W \ra W'$ be a definable morphism in $\Def_S$. Let $X$ be a definable set in $\Def_S$. We denote by $\pi_W$ the projection from $W\times_S X$ to $W$ and by $\pi_{W'}$ the projection from $W' \times_S X$ to $W'$. Let $\varphi$ be a constructible exponential function in $\mathcal C(W'\times_S X)^{exp}$.
	\begin{enumerate}
		\item \label{thm:integrability-condition} If $[\varphi]$ is $\pi_{W'}$-integrable then 
			$[(g \times Id_X)^*\varphi]$ is $\pi_W$-integrable.
			Furthermore, if $g$ is onto then this implication is an equivalence.
		\item \label{thm:equality} If $[\varphi]$ satisfies the condition (\ref{thm:integrability-condition}) then  
			\begin{equation}
				\pi_{W!}\left[(g\times Id_X)^*\varphi \right] = g^*(\pi_{W'!} [\varphi]).
			\end{equation}
	\end{enumerate}
\end{prop}

The $p$-adic analogue of Proposition \ref{compatibilite-pull-back-push-forward} holds naturally, since evaluation in points determines both sides of the equality in the $p$-adic case.

\begin{proof}
	This formula can be easily checked at the level of evaluation of points and can be proved at the level of the ring of constructible exponential functions by induction on the dimension, using cell decompositions and the construction of the motivic integral. A complete proof is given in \cite{CR}.
\end{proof}

\subsection{Relative integration}\label{rel:int}

All of the previous notions can be done relatively to a parameter space, as is done throughout in \cite{CluLoe08a} and in \cite{CluLoe10a}. To this end, one works with $\Def_P$ for a definable subassignment $P$, with relative dimensions (relative to $P$) for objects of $\Def_P$ and with relative Jacobians for isomorphisms in $\Def_P$.

We fix some notations for the relative setting. Let $p : X \rightarrow P$ be a morphism in $\Def_k$, with all
fibers of dimension $d$. We denote by $\mathcal C_{P}^{\leq d}(X)^{\:\exp}$ the ideal of $\mathcal C(X)$ generated by the characteristic functions $1_{X'}$ of subassignments $X'$ of $X$ of relative dimension at most $d$. One forms $C_{P}(X)^{\:\exp}$ again as
$$
\oplus_{d\in \mathbb N}\mathcal C_{P}^{\leq d}(X)^{\:\exp}/\mathcal C_{P}^{\leq d-1}(X)^{\:\exp}.
$$
In particular, if $p$ is the identity $\text{id}:P\to P$ then $C_{P}(P)^{\:\exp}=\mathcal C(P)^{\:\exp}$.
One writes similarly ${\rm I}_{P} C_{P} ( X )^{\rm exp}$ for the constructible exponential functions relative to $P$ which are relatively integrable (namely relatively integrable for the structure morphism $p:X\to P)$.\\

We recall some notations of Section 7 of \cite{CluLoe10a}. Let $p : X \rightarrow P$ be a morphism in $\Def_k$, with all
fibers of dimension
$d$. We denote by $\I_{P} (X)^{\rm exp}$ or by $\I_{p} (X)^{\rm exp}$ the $\mathcal C (P)^{\rm exp}$-module of functions
$\varphi$ in $\mathcal C (X)^{\rm exp}$ whose class $[\varphi]$ in
$$
C_{P}^d (X
)^{\rm exp} := \mathcal C_{P}^{\leq d}(X)^{\:\exp}/\mathcal C_{P}^{\leq d-1}(X)^{\:\exp}
$$
lies in ${\rm I}_{P} C_{P} ( X )^{\rm exp}$.
For $\varphi$ in ${\rm I}_{P} C_{P} ( X )^{\rm exp}$ we write $\mu_{P}
(\varphi)$ or $\mu_{p} (\varphi)$, or $p_!(\varphi)$, to denote the function in $\mathcal C (P)^{\rm exp}$ which is the relative integral (in relative dimension $d$) in the fibers of $p$.

\subsection{Fourier transform, convolution, Schwartz-Bruhat functions}\label{sec:SB}
In this subsection we recall constructions of Fourier transform and convolution product from \S 7 of \cite{CluLoe10a}. 
We use notations from \ref{rel:int}.

\subsubsection{Fourier transform}  \label{Fourier}

\begin{notation}
For an integer $m$, we denote by $V_{x}$ and $V_{y}$ the definable set $h[m,0,0]$ with $x$ and $y$ variables. We denote by $V_{(x,y)}$ the product $V_x \times V_y$ and by $p_x$ and $p_y$ the canonical projections from $V_{(x,y)}$ to $V_x$, resp.~$V_y$.
For $P$ a definable set, we still denote by $p_x$ and $p_y$ the canonical projections from $P\times V_{(x,y)}$ to $P\times V_{x}$ and $P\times V_y$.
We extend this notation also for other products of this type.
\end{notation}
We consider the constructible exponential function in $\mathcal C(P\times V_{(x,y)})^{\exp}$
$$E(x\mid y)=E\left(\sum_{i=1}^{m}x_i y_i\right).$$
For any constructible exponential function $\varphi$ in $\I_P(P\times V_x)^{\exp}$, the constructible exponential function
$p_{x}^{*}(\varphi)E(x\mid y)$ in $\mathcal C(P\times V_{(x,y)})^{\:\exp}$ is $p_Y$-integrable and as usual the \emph{Fourier transform} is defined as the $\mathcal C(P)^{\exp}$-linear application
$$
\mathcal F\colon \begin{cases} \I_P(P\times V_x)^{\exp} \longrightarrow  \mathcal C(P \times V_y)^{\exp}\\  \varphi \longmapsto  \mathcal F(\varphi) =p_{y!}(p_{x}^{*}(\varphi)E(x\mid y)).
\end{cases}
$$
\begin{notation}
Instead of writing $p_{x}^{*}(\varphi)$ we will sometimes write  $\varphi(p,x)$, with $p$ running over $P$. Instead of writing
$p_{y!}$ we will sometimes write  $\int_{x \in V_x}$.
With this notation we have
$$\mathcal F(\varphi) : (p,y) \mapsto \int_{x\in V_x}\varphi(p,x)E(x\mid y)dx.$$
\end{notation}

\begin{example}  \label{Fourier-indicatrice}
 Consider a definable function $\alpha : P \ra \mathbb Z$, the ball
 $$B_{\alpha}=\{(p,x)\in P\times V_x \mid \min \ord x_i \geq \alpha(p)\},$$
 and its characteristic function $1_{B_{\alpha}}$.
 Then, by Proposition 7.3.1 of \cite{CluLoe10a}, we have $\mathcal F(1_{B_{\alpha}})=\L^{-m \alpha}1_{B_{-\alpha+1}}.$
\end{example}

\subsubsection{Convolution}
We denote by $x+y$ the morphism from $P\times V_{(x,y)}$ to $P\times V_z$, which maps $(p,x,y)$ to $(p,x+y)$.

\begin{defn} Let $f$ be in $\I_P(P\times V_x)^{\exp}$ and $g$ be in $\I_P(P\times V_x)^{\exp}$. The constructible exponential function 
	$p_x^{*}(f)p_y^{*}(g)$ lies in  $\I_{x+y}(P\times V_{(x,y)})^{\exp}$ and the \emph{convolution product} of $f$ and $g$ is the constructible exponential function in $\mathcal C(P\times V_z)^{\exp}$
$$f*g:=\mu_{x+y}(p_x^{*}(f)p_y^{*}(g)).$$
\end{defn}
\begin{rem}  \label{convolution-remark}
We consider the definable bijection $h$ from $P\times V_{(z,y)}$ to $P\times V_{(x,y)}$ which maps $(p,z,y)$ on $(p,z-y,y)$. The order of the relative Jacobian over $P$ of this map is equal to 0. Thus,  by the change of variables formula we have
 $$p_x^{*}(f)p_y^{*}(g)=\mu_h(h^{*}(p_x^{*}(f)p_y^{*}(g))).$$
 By the equality $p_z=(x+y) \circ h$ and functoriality of the construction we deduce
 $$f*g = \mu_{p_z}( h^{*}(p_x^{*}(f)p_y^{*}(g)) ).$$
 Using notation for integrals as in \ref{Fourier} we have the usual convolution formula:
 $$f*g = (p,z) \mapsto \int_{y \in V_y}f(p,z-y)g(p,y)dy = \int_{x \in V_x}f(p,x)g(p,z-x)dx.$$
\end{rem}

We denote by $V$ the definable set $h[m,0,0]$ for $m$ a positive integer.
\begin{prop}[\cite{CluLoe10a}, Proposition 7.4.2]\label{conv-lin}
The convolution product $(f,g)\mapsto f*g$ yields a $\mathcal C(P)^{\exp}$-linear map
$$
\I_P(P\times V)^{\exp} \times \I_P(P\times V)^{\exp} \to \I_P(P\times V)^{\exp}
$$
 and it endows
$\I_P(P\times V)^{\:\exp}$ with an associative and commutative law.
\end{prop}
\begin{prop}[\cite{CluLoe10a}, Proposition 7.4.3]
If $f$ and $g$ are two functions in $\I_P(P \times V)^{\exp}$, then we have
$$\mathcal F(f*g)=\mathcal F(f)\mathcal F(g).$$
\end{prop}

\subsubsection{Schwartz-Bruhat functions}
In the $p$-adic setting (see for instance Chapter 7 of \cite{Igusabook}), a Schwartz-Bruhat function is a locally constant function with compact support. In particular such a  function has a bounded support and is constant on balls of sufficiently large (valuative) radius $r$. In the motivic setting Cluckers and Loeser give \S 7.5 of  \cite{CluLoe10a} the following definition.

\begin{defn} \label{defn:schw}
Let $V$ be the definable set $h[m,0,0]$ for a positive integer $m$. Let $P$ be a definable set. 
The set of \emph{Schwartz-Bruhat functions} on $V$ with parameters in $P$ denoted by $S_{P}(P\times V)$ is the $\mathcal C(P)^{\exp}$-module of constructible functions $\varphi$ in $\mathcal I_P(P\times V)^{\exp}$ which satisfy two conditions \\

$\bullet$ Bounded support condition. There is a definable function $\alpha^{-}(\varphi) : P \ra \mathbb Z$ such that $\varphi.1_{B_{\alpha}}=\varphi$, for all definable function $\alpha : P \ra \mathbb Z$ with $\alpha<\alpha^{-}(\varphi)$.
In this situation we will say that $\varphi$ has support in the ball $B_{\alpha^{-}(\varphi)}$.\\

$\bullet$ Locally constant condition. There is a definable function $\alpha^{+}(\varphi): P \ra \mathbb Z$ such that 
$\varphi*1_{B_{\alpha}}=\mathbb L^{-\alpha m}\varphi$, for all definable function $\alpha : P \ra \mathbb Z$
with $\alpha>\alpha^{+}(\varphi)$.
Intuitively, this condition means that $\varphi$ is constant on balls of radius $\alpha$ big enough.
\end{defn}

\begin{rem} It is easily checked that given $\alpha^-$ as in Definition \ref{defn:schw}, any other definable function $\alpha:P\to \Z$ with $\alpha\leq \alpha^{-}$ can serve instead of $\alpha^-$ in the first condition of Definition \ref{defn:schw}, and likewise, given $\alpha^+$, any definable $\alpha$ with $\alpha\geq \alpha^+$ can serve instead of $\alpha^+$ for the second condition.
\end{rem}

Cluckers and Loeser prove  in \cite[Theorem 7.5.1]{CluLoe10a} the following theorem

\begin{thm}\label{thm:four} Fourier transform induces an isomorphism
$$\mathcal F : S_{P}(P \times V) \ra S_{P}(P\times V),$$
and for every Schwartz-Bruhat function $\varphi$ in $S_{P}(P\times V)$ we have
$$(\mathcal F \circ \mathcal F)(\varphi) = \mathbb L^{-m} \check{\varphi}, $$
where $\check{\varphi}$ is $i^{*}\varphi$ with $i:P\times V\ra P\times V$ which sends $x$ to $-x$.
\end{thm}

\begin{prop} \phantomsection \label{convolution}
The convolution product of Schwartz-Bruhat functions in $S_P(P\times V)$ is a Schwartz-Bruhat function.
\end{prop}

\begin{proof}
Let $\varphi$ and $\psi$ in $S_{P}(P\times V)$. 
The functions $\varphi$ et $\psi$ are integrable and the convolution product $\varphi*\psi$ is also integrable by proposition \ref{conv-lin}.
For any definable function $\gamma \geq \max(\alpha^{+}(\varphi),\alpha^{+}(\psi))$, we have
$$(\varphi*\psi)*1_{B_{\gamma}}=\varphi*(\psi*1_{B\gamma})=\varphi*(\mathbb L^{-\gamma m} \psi)=\mathbb L^{-\gamma m}\varphi*\psi$$
For any definable function $\gamma \leq \min(\alpha^{-}(\varphi),\alpha^{-}(\psi))$, by the projection axiom we obtain
$$(\varphi*\psi)1_{B_{\gamma}}= (x+y)_{!}[p_{x}^{*}(\varphi) p_{y}^{*}(\psi) (x+y)^{*}1_{B_{\gamma}}] = \varphi*\psi.$$
Indeed, $\varphi$ is supported in the ball $B_{\alpha^{-}(\varphi)}$, $\psi$ is supported in the ball $B_{\alpha^{-}(\psi)}$ and for any $x$ with $\ord x \geq \alpha^{-}(\varphi)$ and $y$ with $\ord y \geq \alpha^{-}(\psi)$ we have
$\ord x+y \geq \gamma$ impliying the equality
$$
(1_{B_{\alpha^{-}(\varphi)}} \circ p_{x}) (1_{B_{\alpha^{-}(\psi)}} \circ p_{y}) 
(x+y)^{*}1_{B_{\gamma}} = 
(1_{B_{\alpha^{-}(\varphi)}} \circ p_{x}) (1_{B_{\alpha^{-}(\psi)}} \circ p_{y}).$$
\end{proof}

\begin{prop} \phantomsection \label{productSB}
The product of Schwartz-Bruhat functions in $S_P(P\times V)$ is a Schwartz-Bruhat function in $S_P(P\times V)$.
\end{prop}
\begin{proof}Let $\varphi$ and $\psi$ be in $S_{P}(P\times V)$. By the inversion formula for the Fourier transform, there are $\varphi'$ and $\psi'$ 
two Schwartz-Bruhat functions such that $\varphi=\mathcal F(\varphi')$ and $\psi=\mathcal F(\psi')$.
Thus, we have the equalities $\varphi.\psi=\mathcal F(\varphi')\mathcal F(\psi')=\mathcal F(\varphi'*\psi')$,
by proposition \cite[7.4.3]{CluLoe10a}. By the proposition \ref{convolution} the convolution product $\varphi'*\psi'$ 
is a Schwartz-Bruhat function therefore the product $\varphi.\psi$ is a Schwartz-Bruhat function on $X$
by \cite[Theorem 7.5.1]{CluLoe10a}.
\end{proof}

\begin{defn}[Restriction of a Schwartz-Bruhat function]  \label{restriction-SB}
Let $V$ be the definable set $h[m,0,0]$ for a positive integer $m$ and
 let $X$ be an open definable subset of $V$. The set $S_P(P\times X)$ of Schwartz-Bruhat functions on $X$ with parameters in $P$ is the set of Schwartz-Bruhat functions $\varphi$ in $S_P(P\times V)$ such that $\varphi 1_{P\times X} = \varphi$.
\end{defn}

\begin{prop}  \phantomsection \label{SB-exp}
For any Schwartz-Bruhat function $\varphi$ in $S_{P}(P\times X)$ the constructible exponential function
 $\psi : (p,x,\xi) \mapsto \varphi(p,x)E(x\mid \xi)$ is a Schwartz-Bruhat function in $S_{P\times V_\xi}(P\times V_\xi \times X)$ with
 $$\alpha^{-}(\psi) : (p,\xi) \mapsto \alpha^{-}(\varphi)(p)\:\:
 \mathrm{and}\:\:
 \alpha^{+}(\psi) : (p,\xi) \mapsto \max(\alpha^{+}(\varphi)(p),-\ord \xi).$$
\end{prop}
\begin{proof}
The constructible exponential function $\psi$ is $(P\times V_{\xi})$-integrable with a bounded support because $\varphi$ is $P$-integrable with a bounded support $B_{\alpha^{-}(\varphi)}$.
For any $\alpha$ bigger than $\alpha^{+}(\psi)$, using
 Remark \ref{convolution-remark}, the multiplicativity of the additive character $E$ and the axiom of projection \ref{push forward} we obtain, in notation with integrals,
$$
\psi*1_{B_{\alpha}} =
(p,z,\xi) \mapsto E(z\mid \xi)\int_{y\in B_{\alpha(p,\xi)}}\varphi(p,z-y)E(-y \mid \xi)dy.
$$
Using the fact that $\alpha(p,\xi) \geq -\ord \xi $ and axiom (\ref{R3}) for exponentials we deduce that the constructible exponential function
$(z,y,\xi,p) \mapsto E(-y \mid \xi)$ is equal to 1 on the definable set
$$C_{\alpha}=\{(z,y,\xi,p)\mid \ord(y)\geq \alpha(p,\xi)\}.$$
Then, as $\alpha$ is bigger than $\alpha^{+}(\varphi)$, using the convolution identity on $\varphi$ we obtain $\psi*1_{B_{\alpha}} = \mathbb L^{-\alpha m} \psi$. 
\end{proof}

\subsubsection{Locally integrable function}
Let $V$ be the definable set $h[m,0,0]$.

\begin{defn}[Locally integrable functions]  \label{loc-integrable}
	Let $X$ be an open subassignment of $V$. A \emph{locally $P$-integrable} function on $P\times X$ is a constructible function $u$ in 
	$\mathcal C(P\times X)^{\: \exp}$ such that for any definable function $\alpha:P\to \Z$,  the function $1_{B_\alpha} u$ lies in $\I_P(P\times X)^{\: \exp}$.
\end{defn}
\begin{prop}  \label{loc-int-SB}
 Let $X$ be an open subassignment of $V$ and $u$ be a locally $P$-integrable function on $P\times X$. Then, for any Schwartz-Bruhat function $\varphi$ in $S_{P}(P\times V)$, the product $\varphi u$ lies in $\I_P(P\times V)^{\: \exp}$.
\end{prop}

\begin{proof}
By Proposition \ref{thm:four}, there is a Schwartz-Bruhat function $\psi$ in $S_{P}(P\times X)$ such that $\varphi$ is equal to $\mathcal F(\psi)$. The function
$$(p,x,y)\ra u(p,x)1_{B_{\alpha^{-}(\varphi)}}(p,x)\psi(p,y)E(x\mid y)$$ lies in $\I_P(P\times V_x\times V_y )$.
Therefore, by Fubini axiom the function $u\varphi$, equal to $u\mathcal F(\psi)$, lies in $\I_P(P\times V)$.
\end{proof}

\section{Bounded $\mathbb Z$-valued Presburger functions}  \label{dc}

In real or $p$-adic analysis, finiteness properties are often proved using compactness. For instance the corresponding versions of the following lemma \ref{lemme-cle} 
are easy consequences of the compactness of the spheres.
In our setting, finiteness will follow from definability via the quantifier elimination Theorem \ref{QE} and from cell decompositions. Lemma \ref{lemme-cle} will be a key tool in the proof of properties of the motivic wave front set.

\begin{lem}  \label{lemme-cle}
We consider $\mathbb Z$ endowed with the discrete topology and we denote by $B$ the unit ball of $h[1,0,0]$ defined by $\ord x\geq 0$.
For every integer $m\geq 1$, any continuous definable function $\beta$ from $B^{m}$ to
$\mathbb Z_{\geq 0}$ is bounded above.
\end{lem}

\begin{rem}For
analogous properties  in the context of algebraically closed valued fields, $\mathrm{ACVF}(0,0)$, see Lemma 11.6 in \cite{HruKaz06} and
Lemma 7.5 in \cite{Yimu-selecta}.
\end{rem}

\begin{proof}
We prove the lemma by induction on the number $m$ of valued field sort variables.

Suppose first that $m=1$, and let $\beta :  B\ra \mathbb Z_{\geq 0}$ be a continuous definable map.
This map is a Presburger function on $B$,
so by the Cell Decomposition Theorem \ref{celldecompthm}, there is a finite partition of $B$ in cells
$\lambda_{C}:B_{C}\ra Z_{C}$ where
$\lambda_{C}$ is a definable bijection, $Z_{C}$ is a definable subset of
$h[1,n_{C},r_{C}]$ and $C$ is a definable subset of $h[0,n_{C},r_{C}]$ endowed
with a definable morphism $\psi_{C}:C\ra \mathbb Z$
and the canonical projection $p_{C}:Z_{C}\ra C$, such that the following diagram
commutes :
$$
\def\commutatif{\ar@{}[rd]|{\circlearrowleft}}
\xymatrix{
B_{C} \ar[r]^{\lambda_{C}} \ar[d]_{\beta} \commutatif & Z_{C} \ar[d]^{p_{C}} \\
\mathbb Z_{\geq 0}& \ar[l]^{\psi_{C}} C
}
$$
We prove that $\beta$ takes finitely many values on every cell.\\

\noindent If $Z_{C}$ is a $0$-cell, then by definition $r_{C}$ equals $0$ and
it follows from Theorem 2.1.1 of
\cite{CluLoe08a}, see also Theorem \ref{QE}, that the range of $\psi_{C}:C \subset
h[0,n_{C},0] \ra \mathbb Z$ is finite. Hence, the restriction
$\beta_{\mid B_{C}}$ takes finitely many values.\\

\noindent If $Z_{C}$ is a $1$-cell, then by definition there are definable
morphisms
$\alpha:C\ra h[0,0,1]$, $\xi:C\ra h[0,1,0]\setminus \{0\}$ and $c:C\ra h[1,0,0]$
such that
$$Z_{C}=\{(\eta,l,z)\in C[1,0,0] \mid \ord(z-c(\eta,l))=\alpha(\eta,l),\:
\ac(z-c(\eta,l))=\xi(\eta,l)\}.$$
The projection $p_{C}$ is surjective, so the range of $\beta_{\mid B_{C}}$ is equal to the range of $\psi_{C}$.
By definition we have the following commutative diagram, where $i$ is the canonical
injection and $p$ the canonical projection
$$
\def\commutatif{\ar@{}[rrd]|{\circlearrowleft}}
\xymatrix{
h[1,0,0]  \supset  B_{C} \ar[rr]^{\lambda_{C}} \ar[dr]_{i} \commutatif  & & Z_{C} \ar[ld]^{p}  \subset
h[1,0,0]\times h[0,n_{C},r_{C}]
\\
 & h[1,0,0] &
}
$$
This means that $\lambda_{C}$ has the following form:
$$
\lambda_{C} \colon
\begin{cases}
 B_{C} \longrightarrow Z_{C} \\
 z      \longmapsto  (\eta(z),l(z),z).
\end{cases}
$$
In
particular, as $\lambda$ is a surjective function, if $(\eta,l,z)$
belongs to $Z_{C}$ then necessarily $\eta(z)=\eta$ and $l(z)=l$.
By definition of $Z_{C}$ we have the disjoint union over the base $C$
$$Z_{C}=\bigsqcup_{(\eta,l)\in C}\{\eta,l\}\times B_{\eta,l}$$
where $B_{\eta,l}$ is the fiber over $(\eta,l)$, which equals the ball $B\left(c(\eta,l)+\xi(\eta,l)t^{\alpha(\eta,l)},\alpha(\eta,l)+1\right)$.
By the previous remark
$$\lambda_{C}^{-1} (\{\eta,l\}\times B_{\eta,l}) =B_{\eta,l}$$
and the map $(\eta,l)\mapsto
B_{\eta,l}$ is
injective.
Indeed, if there is $(\eta',l')$ such that
$ B_{\eta,l} = B_{\eta',l'}$
then for any $z$ in the ball, $(\eta,l,z)$ and $(\eta',l',z)$ belongs to $Z_{C}$
and by the remark $(\eta,l)=(\eta(z),l(z))=(\eta',l')$.
Note that for every $z$ in the ball, $\beta$ is constant on the ball
$B_{\eta(z),l(z)}$ and equal to $\psi_{C}(\eta(z),l(z))$.

Again by the quantifier elimination theorem \ref{QE} and syntactical analysis of quantifier free formulas, 
$$c:C\subset h[0,n_{C},r_{C}]\ra h[1,0,0]$$
takes finitely many values. For notational simplicity we suppose that $c$ is constant.
We claim that there is $M\geq 0$ such that for all $(\eta,l)$, if
$\psi_{C}(\eta,l)>M$ then $\alpha(\eta,l)\leq M$. Otherwise,
for all $M\geq 0$ there is $(\eta_{M},l_{M})\in C$ such that
$\psi_{C}(\eta_{M},l_{M})>M$ and $\alpha(\eta_{M},l_{M})>M$, then
the sequence $(c_{M}:=c+\xi(\eta_{M},l_{M})t^{\alpha(\eta_{M},l_{M})})$ has a limit $c$
which clearly belongs to $B$.
On the one hand the sequence $(\beta(c_{M})=\psi_{C}(\eta_{M},l_{M}))$ goes to infinity, and on the other hand $\beta(c) \in \Z $ and thus $\beta(c) \neq \infty$. This contradicts that $\beta$ is continuous and in particular constant on a neighborhood of $c$.\\

\noindent So, there is $M\geq 0$ such that for all $(\eta,l)$, if
$\psi_{C}(\eta,l)>M$ then $\alpha(\eta,l)\leq M$. We define
$$C_{M}:= \{(\eta,l)\in C \mid \psi(\eta,l)>M,\: \alpha(\eta,l)\leq M\}.$$
Note that $\beta$ is bounded on $C\setminus C_{M}$ because $\psi_{C}$ is
bounded by $M$ on it.\\

\noindent We prove now that $\beta$ is also bounded on $C_M$. Let us focus on the part $C_{MM}$ where $\alpha=M$, the other parts $C_{Mi}$ of $C_M$ where $\alpha=i$ are treated similarly.\\

\noindent We denote by $\pi(C_{MM})$ the image of $C_{MM}$ under the projection to $h[0,n_C,0]$. For any $\eta$ in $\pi(C_{MM})$ we denote by $C_{MM,\eta}$ the set of $(\eta,l)$ which belong to $C_{MM}$.
Fix $\eta$ in $\pi(C_{MM})$, the map $$\xi(\eta,-):C_{MM,\eta} \subset \{\eta\} \times h[0,0,r_{C}]\ra h[0,1,0]\setminus\{0\}$$ is definable and by the quantifier elimination theorem,
it takes finitely many values. Hence, there are finitely many balls
$B(c+\xi(\eta,l)t^{M},M+1)$, so the restriction $\beta_{\mid C_{MM,\eta}}$ and
$\psi_{\mid C_{MM,\eta}}$ take finitely many values.
We can consider the map
$$
\begin{cases}
\pi(C_{MM}) \longrightarrow h[0,0,1] \\
\eta  \longmapsto  \max \psi_{\mid C_{MM,\eta}}.
\end{cases}
$$
This map is clearly definable.
By using the quantifier elimination theorem once more, and syntactical analysis, this map takes finitely many values, which implies that $\psi_{C}$ and $\beta_{\mid B_{C}}$ take finitely many values.
s the cell decomposition involves finitely many cells, $\beta$ takes finitely many values.

\par

Now let $m>1$ be general and let $r:B^{m}\ra \mathbb Z$ be a definable function.
Consider for every $y$ in $B^{m-1}$ the function $r_{y}:z\in B \mapsto r(y,z)$.
This is a definable function, which, by induction, has finite range.
Then, we can consider the function
$$
\mu \colon
\begin{cases}
 B^{m-1} \longrightarrow \mathbb Z_{\geq 0}\\
 y  \longmapsto  \max r_{y}.
\end{cases}
$$
This function is definable. One easily checks that this function is moreover continuous.
By the induction hypothesis applied to $\mu$, we conclude that $\mu$ takes finitely many values, and consequently $r$ takes finitely many values.
\end{proof}
\begin{prop}  \label{definablecompactness}
For every integer $m\geq 1$, any continuous definable function $\beta : X\ra \mathbb Z$ defined on a closed bounded definable subset of $h[m,0,0]$ takes only finitely many values.
\end{prop}

\begin{proof} The definable set $X$ of $h[m,0,0]$ being bounded, we can suppose it is contained in the ball $B^{m}$. As $X$ is closed, consider the continuous extension $\tilde{\beta}$ of $\beta$ by $0$ to the ball $B^{m}$. We are done by Lemma \ref{lemme-cle} applied once to $\tilde{\beta}_{\mid \tilde{\beta}^{-1}(\mathbb Z_{\geq 0})}$ and once to $-\tilde{\beta}_{\mid \tilde{\beta}^{-1}(\mathbb Z_{\leq 0})}$.
\end{proof}

\begin{rem}
\begin{enumerate}
\item The definable function $\ord$ on $B\setminus \{0\}$ is not bounded, but the punctured ball is not closed. Any extension of this function to a function $B\to \Z$ is not continuous.
\item The definable function $-\ord : h[1,0,0]\setminus B \ra \mathbb Z$ is not bounded, but $h[1,0,0]\setminus B$  is also not bounded.
\end{enumerate}
\end{rem}

%%%%%%%%%%%%%%%%%%%%%%%%%%%%%%%%%%%%%%%%%%%%%%%%%%%%%%%%%%%%%%%%%%%%%%%%%%%%%%%%%%%%%%%%%%%%%%

For all integers $m\geq 1$ and $n\geq 0$, by a closed bounded definable set of $h[m,n,0]$ we mean a definable set which is closed for the discrete topology on the residue field and the valuation topology on the valued field and which has bounded projection on
$h[m,0,0]$. We deduce the following corollaries.

\begin{cor}  \label{corrolaire-principal}
For every integers $m\geq 1$ and $n\geq 0$, any continuous definable function $\beta : X\ra \mathbb
Z$ defined on a closed bounded definable subset of $h[m,n,0]$ takes finitely many values, where $h[m,n,0]$ is endowed with the product topology of the $t$-adic topology on $h[m,0,0]$ and the discrete
topology on $h[0,n,0]$.
\end{cor}

\begin{proof} From $\beta$ one can easily construct a definable continuous map $\tilde \beta: h[m+n,0,0]\to \Z$ having the same range as $\beta$ and having bounded support. Apply Proposition  \ref{definablecompactness} to $\tilde \beta$ to finish the proof.
\end{proof}
\begin{cor}  \label{compacite-param}
 Let $P$ be a definable set in $\Def_{k}$ and $m\geq 1$ and $n\geq 0$ two integers.
 Let $X$ be a closed and bounded definable subset of $h[m,n,0]$.
 Let $\beta$ be a continuous definable function from $P\times X$ to $\mathbb Z$, where respectively $P$
 and $P\times X$ are endowed with the discrete topology and the product topology.
 For any $p$ in $P$, the restriction $\beta_{p}$ of $\beta$ to
 the product $\{p\}\times X$ take finitely many values and
 the maps
 $$ \begin{cases}
    P \longrightarrow \mathbb Z \\
    p  \longmapsto  \max \beta_{p}
    \end{cases}
    \:\:and\:\:
\begin{cases}
    P \longrightarrow \mathbb Z \\
    p  \longmapsto  \min \beta_{p}
    \end{cases}
 $$
 are well defined and definable.
\end{cor}
%%%%%%%%%%%%%%%%%%%%%%%%%%%%%%%%%%%%%%%%%%%%%%%%%%%%%%%%%%%%%%%%%%%%%%%%%%%%
%%%%%%%%%%%%%%%%%%%%%%%%%%%%%%%%%%%%%%%%%%%%%%%%%%%%%%%%%%%%%%%%%%%%%%%%%%%%
%%%%%%%%%%%%%%%%%%%%%%%%%%%%%%%%%%%%%%%%%%%%%%%%%%%%%%%%%%%%%%%%%%%%%%%%%%%%

\section{Motivic oscillatory integrals}  \label{integrales-oscillantes}

We develop a motivic analogue of Proposition 1.1 of \cite{Hei85a} about $p$-adic oscillatory integrals.
Over  the reals conic sets naturally occur; in the $p$-adic context Heifetz replaces the multiplicative group
$(\mathbb R^*_+,\times)$ by an arbitrary finite index subgroup $\Lambda$ of $\mathbb Q_p^\times$. As Forey in its definition of a $t$-adic tangent cone in \cite{Forey}, we use the following definable subgroups of $h[1,0,0]^{\times}$.

\begin{defn} \label{Lambdan}
Let $\mathfrak n\geq 1$ be an integer, we consider
$$\Lambda_{\mathfrak n}:=\{x\in h[1,0,0]\:\mid \: (\mathfrak n \mid \ord x)\: \land\: \ac x=1\}.$$
\end{defn}

Let $P$ be in $\Def_k$ and let $\n$, $m$, $m'$ be nonnegative integers. Let $X$ and $V$ be open definable subassignments of $h[m,0,0]$, resp.~of $h[m',0,0]$. Let $\varphi$ be a Schwartz-Bruhat function in $S_{P}(P\times X)$.
Let $g$ be a definable function from $P\times X\times V$ to $h[1,0,0]$.
We denote by $S$ the product $P\times V \times \Lambda_\n$.
On the product $S\times X$ we consider the definable function
$$
q:(p,\lambda,x,v)\mapsto \lambda g(p,x,v).
$$
Then the constructible exponential function $\pi_{P\times X}^{*}(\varphi)E(q)$
 belongs to
$\I_{S}(S \times X)^{\:\exp}$,
where $\pi_{P\times X}$ is the canonical projection from $S\times X$ to $P\times X$.
We define
$$
I_{\varphi}:=\pi_{P\times V \times \Lambda_\n !} (\pi_{P\times X}^{*}(\varphi) E(q)) \in \mathcal C(S)^{\:\exp},
$$
namely, in notation with integrals,
$$
I_{\varphi} : (p,\lambda,v) \mapsto \int_{X}\varphi(p,x)E(\lambda g(p,x,v))dx.
$$

With this notation we can now formulate the following analogue of Proposition 1.1 of \cite{Hei85a}.

\begin{prop}  \label{oscillante}

With notations from just before the proposition, we make the following assumptions. 
\begin{enumerate}
\item For each $p$ in $P$ and $v$ in $V$, the map $x\mapsto g(p,x,v)$ is $C^1$, we write $\mathrm{grad}_{x}g(p,x,v)$ for its gradient.
\item  There are definable maps $N_{R}$ and $N_{\mathrm{grad}}$ from $P$ to $\mathbb Z$ such that, for any $p$ in $P$, all $x$ with $\ord x \geq \alpha^{-}(\varphi)(p)$, all $y$ with $\ord y > \alpha^{+}(\varphi)(p)$, and all $v\in V$,
\begin{equation}\label{hypothese-gradient}
\ord \mathrm{grad}_{x}g(p,x,v) \leq N_{\mathrm{grad}}(p)
\:\:\text{and}\:\:
\ord R(p,x,y,v)\geq N_{R}(p),
\end{equation}
where $R$ is a definable function satisfying the equality
\begin{equation} \label{Taylor}
g(p,x+y,v)=g(p,x,v)+<\mathrm{grad}_{x}g(p,x,v),y> + <R(p,x,y,v)y,y>.
\end{equation}
 \end{enumerate}
Then, the constructible exponential function $I_{\varphi}$ has a bounded support in the $\lambda$-variable with bound only depending on $p$ in $P$.
More precisely, there is a definable function $A: P\to \Z$ such that
\begin{equation}\label{formule:oscillante}
I_{\varphi}1_{B_{-A-N_{\mathrm{grad}}}} = I_{\varphi},
\end{equation}
with
$$
B_{-A-N_{\mathrm{grad}}} = \{(p,v,\lambda)\in P\times V \times \Lambda_\n \mid \ord \lambda \geq -A(p)-N_{\mathrm{grad}}(p) \},
$$
in particular, $A$ can be chosen as $\max(N_{\mathrm{grad}}-N_R+1,\alpha^{+}(\varphi))$.
\end{prop}

\begin{proof}
  Using previous notations we prove the equality $I_{\varphi}1_{J_{A'}} = 0$ for any definable function $A':P\to \mathbb Z$, where
  $$J_{A'}=\{(p,\lambda,v)\in P\times V \times \Lambda_\n \mid -N_{R}(p)-2A'(p)<\ord \lambda < -A(p)-N_{\mathrm{grad}}(p)\}.$$
  We deduce \ref{formule:oscillante} by the inclusion of $(P \times V \times \Lambda_\n)\setminus B_{-A-N_{\mathrm{grad}}}$ in the union $\bigcup_{A'\geq A} J_{A'}$.	
  Consider now a definable function $A'\geq A$. By definition we have $A'\geq \alpha^{+}(\varphi)$. Then, the convolution equality on $\varphi$ implies the equality:
  $$ I_{\varphi} =  (p,\lambda,v)  \mapsto \int_{x \in B_{\alpha^{-}(\varphi)(p)}} \mathbb L^{A'(p) m}
  \left(\int_{z\in B_{\alpha^{-}(\varphi)(p)}} \varphi(p,z)1_{B_{A'(p)}}(x-z)dz\right)E(\lambda g(p,x,v))dx
  $$
   which is equal (by Fubini and projection axioms) to
  $$
  I_{\varphi} = (p,\lambda,v)  \mapsto  \mathbb L^{A'(p) m} \int_{z \in B_{\alpha^{-}(\varphi)(p)}}  \varphi(p,z)
  \left(\int_{x\in B_{\alpha^{-}(\varphi)(p)}} 1_{B_{A'(p)}}(x-z) E(\lambda g(p,x,v))dx\right)dz.
  $$
  By relation \ref{Taylor}, and the definition of $N_R$, we have the inequality 
  \begin{equation}  \label{R}
  \ord \lambda R(p,z,x,v)(x-z) \cdot (x-z) \geq \ord \lambda + N_{R}(p) + 2A'(p).
  \end{equation}
  for any $(p,x,z,v)$ with $\ord(z-x)\geq A'(p)$, and $\lambda$ in $\Lambda_\n$. 
  If $-N_{R}-2A'<\ord \lambda$, then $E\left(\lambda R(p,z,x,v)(x-z)\cdot (x-z)\right)=1$ by axiom \ref{R3} and using the morphism property of $E$ it is enough to prove that
  \begin{equation}  \label{eq:G}
  \left((p,\lambda,v,z)\mapsto \int_{x\in B_{\alpha^{-}(\varphi)(p)}} 1_{B_{A'(p)}}(x-z) E(\lambda \mathrm{grad}\:g(p,z,v)\cdot (x-z))dx \right)\cdot 1_{J_{A'(p)}}=0
  \end{equation}

  First by the change of variables formula this is equivalent to show
  $$ \left((p,\lambda,v,z)\mapsto \int_{x\in B_{\alpha^{-}(\varphi)(p)}} 1_{B_{A'(p)}}(x) E(\lambda \mathrm{grad}\:g(p,z,v)\cdot x)dx \right)\cdot 1_{J_{A'(p)}}
  =0$$
Below, we denote
$$B_{\alpha^{-}(\varphi)} = \{(p,z,v)\in P\times h[2m,0,0]\times V \mid
\ord(x-z)\geq \alpha^{-}(\varphi)(p)\}.$$
By assumption, for any $(p,x,z,v)$ in $B_{\alpha^{-}(\varphi)}$, the vector $\mathrm{grad}\:g(p,z,v)$ is different from zero, written as $(u_{i}(p,z,v))$, and we suppose that for any $(p,x,z,v)$ in $B_{\alpha^{-}(\varphi)}$ we have $$\ord u_{1}(p,z,v)=\min \ord u_{i}(p,z,v)=\ord \mathrm{grad}\:g(p,z,v)$$ 
otherwise, we stratify $B_{\alpha^{-}(\varphi)}$ and work stratum by stratum. We consider new variables
$$
(y_{1}=\lambda u_{1}(p,z,v)x_{1}+ \cdots + \lambda u_{m}(p,z,v)x_{m}, y_{2}=x_{2},
\hdots, y_{m}=x_{m}).
$$
Applying the change of variables formula and axioms of subsection \ref{push forward}, we obtain the equality between constructible functions in $(\lambda,p,z,v)$
$$
\int_{x \in B_{A'(p)}}E(\lambda \mathrm{grad}\:g(p,z,v)\cdot x)dx  =   \mathbb L^{-\lambda u_{1}(p,z,v)} \mathbb L^{-(m-1)A'(p)}
\int_{y_{1}\in B_{N(p,z,v,\lambda)}}E(y_{1})d_{y_{1}}$$
with
$N(p,z,v,\lambda)=\ord u_{1}(p,z,v)+A'(p)+\ord \lambda.$
In particular, for any $\lambda$ with $\ord \lambda < -A'(p)-N_{\mathrm{grad}}(p)$ and any $(p,x,z,v)$ in $B_{\alpha^{-}(\varphi)}$, 
we have $N(p,z,v,\lambda)<0$ by \ref{hypothese-gradient}.
Then, we conclude that the intersection $J_{A'}\cap \{(p,z,v,\lambda)\mid N(p,z,v,\lambda)\geq 0\}$ is $\emptyset$ and by \cite[Proposition 7.3.1]{CluLoe10a},
we deduce the equality of constructible functions in $(p,z,v,\lambda)$ 
$$\left(\int_{B_N}E(y_{1})d_{y_{1}}\right)1_{J_{A'}} = 
\mathbb L^{-N}
1_{\{(p,z,v,\lambda)\mid N(p,z,v,\lambda)\geq 1\}}1_{J_{A'}} = 0$$ 
which implies \ref{eq:G}.
\end{proof}

\section{Definable distributions}  \label{definable distributions}

\subsection{Operations on Schwartz-Bruhat functions}

\begin{notation} \label{notation:distributions}
For $V=h[m,0,0]$, the notation $V_x$ means $V$ using variables $x$. We will also use for a product
$V_{x}\times V_{y}$ the notation $V_{x,y}$ which elements are denoted by $(x,y)$. We will also use this notation for other Cartesian products.
For any definable set, we denote by $\pi_{P}^{P\times V}$ (or simply by $\pi_P$ when the context is clear) the canonical projection from $P\times V$ to $P$.
\end{notation}

Before giving the definition of a definable distribution, we give some properties of Schwartz-Bruhat functions.

\begin{lem}[Pull-back of Schwartz-Bruhat functions]  \label{pull-backSB}
 Let $m>0$ be an integer and $V$ be the definable set $h[m,0,0]$.
 Let $g:W \ra W'$ be a definable morphism.
 For any Schwartz-Bruhat function $\varphi$ in $S_{W'}(W'\times V)$, the pull-back $(g\times \Id_{V})^{*}\varphi$ is a Schwartz-Bruhat function in
 $S_{W}(W\times V)$.
\end{lem}
\begin{proof}
Let $\varphi$ be a Schwartz-Bruhat function in $S_{W'}(W'\times V)$. We denote by $\beta^-$ the definable function
$\alpha^{-}(\varphi)\circ g : W \ra \mathbb Z$.
Using the equality
 $\varphi=\varphi 1_{B_{\alpha^{-}(\varphi)}}$ we obtain
 $$(g\times \Id_V)^*\varphi = \left((g\times \Id_V)^*\varphi\right)\left( (g\times \Id_V)^*1_{B_{\alpha^{-}(\varphi)}}\right) =
 \left((g\times \Id_V)^*\varphi\right) 1_{B_{\beta^{-}}}.$$

 In particular, for any $\beta < \beta^{-}$  from the equality
 $1_{B_{\beta}}1_{B_{\beta^-}} = 1_{B_{\beta^-}}$ we deduce
 $$\left((g\times \Id_V)^* \varphi\right)1_{B_\beta} =
 (\left((g \times \Id_V)^* \varphi\right)1_{B_{\beta^-}})1_{B_\beta}  =
\left((g \times \Id_V)^* \varphi\right) 1_{B_{\beta^-}} = (g \times \Id_V)^* \varphi.$$
Let $\beta^+ = \alpha^+(\varphi) \circ g : W \ra \mathbb Z$.
By definition of the convolution product we have
$$\varphi*1_{B_{\alpha^{+}}} = \pi_{W'\times V_x!}[(\pi_{W'\times V_z})^*\varphi \cdot  1_{B_{\alpha^{+}(\varphi)}}\circ d]$$
with $V=V_x=V_z$ and $d$ the definable map from $ V_x\times V_z$ to $h[m,0,0]$ which maps $(x,z)$ on $x-z$.
In notation with integrals,
$$ \varphi*1_{B_{\alpha^{+}}} : (w,x) \mapsto \int_{\{(w,x)\}\times V_z} \varphi(w,z)1_{B_{\alpha^{+}(\varphi)(w)}}(x-z) dz.$$
By Proposition \ref{compatibilite-pull-back-push-forward} we have the equality
$$(g\times \Id_{V_x})^*\circ \pi_{W'\times V_x!} = \pi_{W \times V_x!}
\circ (g \times \Id_{V_x\times V_z})^{*} $$
implying
$$\begin{array}{cclcl}
(g\times \Id_{V_x})^*(\varphi*1_{B_{\alpha^{+}(\varphi)}}) & = & 
\pi_{W\times V_x!}
[(g\times \Id_{V_x\times V_z})^*
[(\pi_{W\times V_z})^*\varphi \cdot 1_{B_{\alpha^{+}(\varphi)}}\circ d]
] & & \\
& = & 
 \pi_{W \times V_x!}
[((g \times \Id_{V_z})^{*}\varphi) (1_{B_{\beta^{+}}} \circ d)] 
& = &
((g \times \Id_V)^{*}\varphi) * 1_{B_{\beta^+}}.
\end{array}
$$

In particular we have the equality:
$(g\times \Id_{V_x})^*\varphi = \mathbb L^{\beta^+ m }((g\times \Id_{V_x})^* \varphi)*1_{B_{\beta^{+}}}.$

Let $\beta \geq \beta^{+}$. By associativity of the convolution product we obtain

$$
\begin{array}{cclclcl}
(g \times \Id_V)^{*}\varphi & = &
\mathbb L^{\beta^+ m}\left((g \times \Id_V)^{*}\varphi\right) * 1_{B_{\beta^{+}}} 
& = & \mathbb L^{\beta^+ m}\left((g \times \Id_V)^{*}\varphi\right) * \left(1_{B_{\beta^{+}}}*\mathbb L^{\beta m}1_{B_{\beta}}\right) & &\\
& = & 
\left(\mathbb L^{\beta^+ m}\left((g \times \Id_V)^{*}\varphi\right) * 1_{B_\beta^{+}}\right)*
\mathbb L^{\beta m}1_{B_{\beta}} & = & \mathbb L^{\beta m}\left((g \times \Id_V)^{*}\varphi\right) * 1_{B_\beta}.
\end{array}
$$

We conclude that $(g \times \Id_V)^{*}\varphi$ is a Schwartz-Bruhat function in
$S_{W}(W\times V)$.
\end{proof}

\begin{rem}[Pull-back and restriction of Schwartz-Bruhat functions]
 In this remark, we generalize the previous lemma to the case of restrictions of Schwartz-Bruhat functions.
  Let $V$ be the definable set $h[m,0,0]$ for $m>0$. Let $g:W\ra W'$ be a definable morphism.  Let $X$ be an open definable subset of $V$. Let $\varphi$ be a Schwartz-Bruhat function in $S_{W'}(W'\times X)$ which means, following Definition \ref{restriction-SB}, that  $\varphi$ belongs to $S_{W'}(W'\times V)$ and satisfies the equality $\varphi 1_{W'\times X} = \varphi$.
 As, the constructible exponential function $(g\times \Id_{V})^* 1_{W'\times X}$ is equal to
 $1_{W \times X}$, we have the equalities
 $$(g\times \Id_{V})^*\varphi =
 (g\times \Id_{V})^*(\varphi 1_{W' \times X}) =
 ((g\times \Id_{V})^*\varphi)1_{W'\times X},$$
 and we conclude that $(g\times \Id_{V})^* \varphi$ belongs to $S_{W}(W\times X)$.
\end{rem}

\begin{defn}[$(g\times Id_X)$-convenient Schwartz-Bruhat functions]
 Let $m>0$ be an integer. Let $g:W \ra W'$ be a definable morphism. Let $X$ be an open definable set of $h[m,0,0]$.  
 A Schwartz-Bruhat function $\varphi$ in $S_{W}(W\times X)$ is said to be 
 \emph{$(g\times Id_X)$-convenient} if and only if 
 \begin{itemize}
	 \item[$\bullet$] $\varphi$ is $(g\times Id_X)$-integrable over $W\times X$ and $(g\times Id_X)_{!}\varphi$ belongs to $\mathcal I_{W'}(W'\times X)^{exp}$,
 \item[$\bullet$] there is a definable function $\beta^+ : W' \ra \mathbb Z$ such that 
 $\alpha^+(\varphi) = \beta^+ \circ g$,
 \item[$\bullet$] there is a definable function $\beta^- : W' \ra \mathbb Z$ such that 
 $\alpha^-(\varphi) = \beta^- \circ g$.
 \end{itemize}
\end{defn}

\begin{lem}[Push-forward of Schwartz-Bruhat functions] \label{push-forwardSB}
 Let $V$ be the definable set $h[m,0,0]$ with $m>0$. 
 Let $g:W \ra W'$ be a definable morphism. Let $\varphi$ be a Schwartz-Bruhat function in $S_{W}(W\times V)$.
 If $\varphi$ is $(g\times Id_{V})$-convenient then the push-forward 
 $(g\times Id_{V})_!\varphi$ belongs to $S_{W}(W\times V)$.
\end{lem}
\begin{proof}
By assumption the constructible function $\varphi$ is $(g\times Id_V)$-integrable and $(g\times Id_V)_{!}\varphi$ belongs to $\mathcal I_{W'}(W'\times V)$.
There are definable functions $\beta^+$ and $\beta^-$ from $W'$ to $Z$ such that 
 $\alpha^+(\varphi) = \beta^+ \circ g$ and $\alpha^-(\varphi) = \beta^- \circ g$.
Using the previous notations, for any definable function $\beta$ from $W'$ to $\mathbb Z$ with $\beta \leq \beta^{-}$, by the projection formula we have 
 $$ ((g\times Id_V)_{!}\varphi)\cdot 1_{B_\beta} =
 (g \times Id_V)_{!} \left(\varphi 1_{B_{\beta \circ g}} \right) 
 $$
 By definition $\beta \circ g \leq \alpha^{-}(\varphi)$ then $\varphi 1_{B_{\beta \circ g}} =  \varphi $
 and we can conclude that
 $ (g\times Id_V)_{!}\varphi 1_{B_\beta} = (g\times Id_V)_{!}\varphi.$\\
We prove now the equality
$$\left((g\times Id_V)_! \varphi\right) * 1_{B_{\beta^+}} = \mathbb L^{\beta^+ m}
(g\times Id_V)_! \varphi.$$
Indeed, using these definitions, projection formula and Fubini axiom we obtain
 $$
 \begin{array}{cclcl}
 \left((g\times Id_V)_! \varphi\right) * 1_{B_{\beta^+}} & = & \pi_{W'\times V!} 
 ((\pi_{W'\times V_z})^*((g \times Id_{V_{z}})_{!}\varphi)\cdot 1_{B_{\beta^+}}\circ d )\\ 
 & = & \pi_{W'\times V!} \left((g \times Id_{V \times V_z})_{!}((\pi_{W\times V_z})^*\varphi \cdot 1_{B_{\beta^+ \circ g}}\circ d )\right) \\
 & = & (g \times Id_V)_{!} \pi_{W\times V!}  ((\pi_{W\times V_z})^*\varphi \cdot 1_{B_{\alpha^+(\varphi)}}\circ d)) 
 & = & \mathbb L^{\beta^+ m} (g\times Id_V)_{!}\varphi.
 \end{array}
 $$
Let $\beta$ be a definable function from $W'$ to $\mathbb Z$ with $\beta \geq \beta^{+}$. Using again the associativity of the convolution product we have 
 $$
 \begin{array}{cclcl}
 (g\times Id_V)_! \varphi * 1_{B_\beta} & = & 
 (\mathbb L^{\beta^+m}(g\times Id_V)_! \varphi * 1_{B_{\beta^+}})*1_{B_\beta} 
 & = &  \mathbb L^{\beta^+m}(g\times Id_V)_! \varphi * (1_{B_{\beta^+}}*1_{B_\beta}) \\
 & = & \mathbb L^{\beta^+m}(g\times Id_V)_! \varphi * \mathbb L^{-\beta m} 1_{B_{\beta^+}}
 & = & \mathbb L^{-\beta m}(g\times Id_V)_! \varphi.  
 \end{array}
 $$
We conclude $(g\times Id_{V})_!\varphi$ belongs to $S_{W}(W\times V)$.
\end{proof}

\begin{rem}[Pushforward and restriction of Schwartz-Bruhat functions]
 In this remark, we generalize the previous lemma to the case of restrictions Schwartz-Bruhat functions. 
 Let $g:W\ra W'$ be a definable morphism. Let $V$ be the definable set $h[m,0,0]$ for $m>0$. Let $X$ be an open definable subset of $V$. 
 Let $\varphi$ be a Schwartz-Bruhat function in 
 $S_{W}(W\times X)$ which means following the definition \ref{restriction-SB} that  $\varphi$ belongs to $S_{W}(W[m,0,0])$ and satisfies $\varphi\cdot 1_{W\times X} = \varphi$. We assume $\varphi$ to be $(g\times Id_X)$-compatible. As the constructible function $(g\times Id_{V})^* 1_{W'\times X}$ is equal to 
 $1_{W \times X}$, we have by the projection axiom \ref{axiom-projection} the equalities
 $$(g\times Id_{V})_{!}(\varphi 1_{W \times X}) = 
 (g\times Id_{V})_{!}(\varphi (g\times Id_X)^* 1_{W'\times X}) = 
 1_{W'\times X}\cdot (g\times Id_{V})_{!} \varphi.$$
 meaning that $(g\times Id_{V})_{!}\varphi$ belongs to $S_{W'}(W'\times X)$.
\end{rem}

\subsection{Definable distributions}
We introduce in this subsection a notion of definable distributions. % in the \change{present} context.
In the $\mathrm{ACVF}(0,0)$ context, Hrushovski and Kazhdan introduced a notion of definable distribution in \S 11 of \cite{HruKaz06}, see also \S 5 of Yin's paper \cite{Yimu-selecta}.

\begin{defn}[Definable distributions]  \label{definable distribution}
Let $P$ be a definable set in $\Def_k$. Let $V$ be the definable set $h[m,0,0]$ for $m>0$. A \emph{definable distribution} $u$ on $V$ with parameters relative to $P$, will be a family $u = (u_{\Phi_W})_{\Phi_W \in \Def_P}$ where for each $\Phi_W$ in $\Def_P$, $u_{\Phi_W}$ is a $\mathcal C(W)^{\exp}$-linear map
$$
u_{\Phi_W} \colon S_{W}(W\times V) \longrightarrow \mathcal C(W)^{\exp}
$$
such that for any $\Phi_W$ and $\Phi_{W'}$ in $\Def_P$, for any definable morphism
$g:W\ra W'$
two compatibility conditions are satisfied:

\begin{enumerate}
 \item Pull-back condition: for any $\varphi$ in $S_{W'}(W'\times V)$ we have
 $$g^{*}<u_{\Phi_{W'}},\varphi>\: = \:<u_{\Phi_W},(g\times \Id_{V})^{*}\varphi>
 \:\in \: \mathcal C(W)^{\exp}.$$
 Note that by Lemma \ref{pull-backSB}, $(g\times \Id_{V})^{*}\varphi$ belongs to $S_{W}(W\times V)$.\\

\item Push-forward condition:  for any $(g\times Id_{V})$-convenient Schwartz-Bruhat function $\varphi$ in $S_{W}(W\times V)$ we have
 \begin{enumerate}
 \item $<u_{\Phi_W},\varphi>$ is $g$-integrable,
 \item $g_{!}<u_{\Phi_W},\varphi>\: = \: <u_{\Phi_{W'}},(g\times Id_{V})_{!}\varphi> 
 \:\in\: 
 \mathcal C(W')^{\exp}.$
 \end{enumerate}
 Remark that by lemma \ref{push-forwardSB}
  $(g\times Id_{V})_{!}\varphi$ belongs to $S_{W'}(W'\times V)$.\\ 
\end{enumerate}
The set of definable distributions on $V$ with parameters relative to $P$ will be denoted by $S'_{P}(V)$.
\end{defn}

\begin{notations} Using notations \ref{notation:distributions}, for any definable morphism $\Phi_{W}$ in $\Def_{P}$, for any definable set $V$, we denote
$$u_{\Phi_{W}\times V} := u_{\pi_{P}^{P\times V}\circ (\Phi_{W}\times V)},\:\:
u_{P\times V}:=u_{\pi_{P}^{P\times V}}\:\:\mathrm{and}\:\:u_{P}:=u_{\Id_P}.$$
\end{notations}

\begin{rem}[$\mathcal C(P)^{\exp}$-module structure on $S'_{P}(V)$]
For any $\Phi_W$ in $\Def_{P}$, for any $l$ in $\mathcal C(P)^{\exp}$, the pull-back $\Phi_{W}^{*}(l)
$ belongs to $\mathcal C(W)^{\exp}$. 
There is a natural $\mathcal C(P)^{\exp}$-module structure on $S'_{P}(V)$: for any definable distribution $u$, we define a distibution $l \cdot u$ by 
$$(l\cdot u)_{\Phi_W}=(\Phi_{W}^{*}(l)u_{\Phi_W}) : \varphi \mapsto <u_{\Phi_W},\Phi_{W}^{*}(l)\varphi>,$$
for any $\Phi_W$ in $\Def_{P}$.
Indeed, let $\Phi_W$ and $\Phi_{W'}$ be two definable subsets in $\Def_{P}$ and
$g:\Phi_W \ra \Phi_{W'}$ be a definable morphism in $\Def_{P}$.
Using the definition of $(l \cdot u)_{\Phi_W'}$, the equality between $\Phi_W$ and $\Phi_{W'}\circ g$, the pull-back condition of $u$ and the definition of $(l \cdot u)_{\Phi_W}$, we obtain
 $$g^{*}<(l \cdot u)_{\Phi_W'},\varphi> =
 <(l \cdot u)_{\Phi_W},(g\times \Id_{V})^*\varphi>.$$
Let $\varphi$ be $(g\times \Id_{V})$-convenient Schwartz-Bruhat function in $S_{W}(W\times V)$. Using the projection formula, the equality
between $\Phi_W$ and $\Phi_{W'}\circ g$ and the assumption on $\varphi$, we have
 $$(g\times \Id_{V})_{!}\left(\Phi_W^*(l) \varphi \right) =
 \Phi_{W'}^*(l) (g\times \Id_{V})_{!} \varphi \in S_{W'}(W'\times V).$$
 Then, we conclude using the push-forward condition on $u$:
 $$(g\times \Id_{V})_{!}<(l \cdot u)_{\Phi_W},\varphi> =
 <(l \cdot u)_{\Phi_W'},(g\times \Id_{V})_{!}\varphi>.$$
\end{rem}

\begin{prop}[Product by a Schwartz-Bruhat function] \label{produitparSB}
Let $V$ be the definable subset $h[m,0,0]$ for $m>0$. Let $\phi$ be a Schwartz-Bruhat function in $S_{P}(P\times V)$ and $u$ be a definable distribution in $S'_{P}(V)$.
For any $\Phi_W$ in $\Def_P$ we define the $\mathcal C(W)^{\exp}$-linear form
$\left((\Phi_W \times \Id_V)^* \phi\right) \cdot u_{\Phi_W}$
by
$$<\left((\Phi_W \times \Id_V)^* \phi\right) \cdot u_{\Phi_W}, \varphi> :=
<u_{\Phi_W}, \left((\Phi_W \times \Id_V)^* \phi\right) \cdot \varphi>$$
The family $(\left((\Phi_W \times \Id_V)^* \phi\right).u_{\Phi_W})_{\Phi_W \in \Def_P}$ is a definable distribution denoted by $\phi \cdot u$.
\end{prop}
\begin{proof}
Let $V$ be the definable subset $h[m,0,0]$ for $m>0$. Let $\phi$ be a Schwartz-Bruhat function in $S_{P}(P\times V)$ and $u$ be a definable distribution in $S'_{P}(V)$.
 By Lemma \ref{pull-backSB} and by Corollary \ref{productSB}, for any definable morphism $\Phi_W$ in $\Def_P$,  the constructible exponential function
 $(\Phi_W \times \Id_V)^* \phi$ is  a Schwartz-Bruhat function in $S_{W}(W\times V)$ and for any Schwartz-Bruhat function $\varphi$ in $S_{W}(W\times V)$, the product $\left((\Phi_W \times \Id_V)^* \phi\right)\cdot \varphi$ is a Schwartz-Bruhat function in
 $S_{W}(W\times V)$. Thus, the linear form $\left((\Phi_W \times \Id_V)^* \phi\right)\cdot u_{\Phi_W}$ is well defined.
 We prove now that $\phi \cdot u$ is a definable distribution.
 Let $g$ be a definable morphism from $W$ to $W'$.
 Let $\Phi_W$ and $\Phi_{W'}$ be two definable morphisms in $\Def_P$ such that $\Phi_W=\Phi_{W'}\circ g$.  We have the identity 
 \begin{equation} \label{pullbackphi}
   (\Phi_W \times \Id_{V})^*\phi = (g\times \Id_X)^* (\Phi_{W'}\times \Id_{V})^* \phi.
 \end{equation}
 Let $\varphi$ be a Schwartz-Bruhat function in $S_{W'}(W'\times V)$. Using the definition, \ref{pullbackphi} and the compatibility relations of $u$ we deduce
 $$g^*<\left((\Phi_{W'}\times \Id_{V})^*\phi\right)\cdot u_{\Phi_{W'}},\varphi> =
 <\left((\Phi_{W}\times \Id_{V})^*\phi\right)\cdot u_{\Phi_{W}},(g\times \Id_V)^*\varphi>.
 $$
 Let $\varphi$ be a Schwartz-Bruhat function in $S_{W}(W\times V)$ which is
 $(g\times \Id_V)$-compatible, by relation \ref{pullbackphi} and projection formula \ref{axiom-projection}, the constructible exponential function
   $\left((\Phi_W\times \Id_V)^*\phi\right)\cdot \varphi$ is also $(g\times \Id_V)$-integrable with the relation
   $$(g\times \Id_V)_{!}(\left((\Phi_W\times \Id_{V})^*\phi\right)\cdot\varphi) =
   \left((\Phi_{W'}\times \Id_{V})^*\phi\right)\cdot(g\times \Id_V)_{!}\varphi.$$
The product $\left((\Phi_{W}\times \Id_{V})^*\phi\right)\cdot\varphi$ is a Schwartz-Bruhat in
  $S_{W}(W\times V)$ with
  $$\alpha^{-}(\left((\Phi_{W}\times \Id_{V})^*\phi\right)\cdot \varphi) =
  \min (\Phi_{W'}(\alpha^-(\phi))\circ g, \alpha^{-}(\varphi))$$
  and
  $$\alpha^{+}(\left((\Phi_{W}\times \Id_{V})^*\phi\right)\cdot\varphi) =
  \max (\Phi_{W'}(\alpha^+(\phi))\circ g, \alpha^{+}(\varphi)),$$
  using Proposition\ref{productSB} and the equalities
  $$
  \alpha^{-}((\Phi_{W}\times \Id_{V})^*\phi) = \Phi_{W'}(\alpha^-(\phi))\circ g \:\:\mathrm{and} \:\:
  \alpha^{+}((\Phi_{W}\times \Id_{V})^*\phi) = \Phi_{W'}(\alpha^+(\phi))\circ g.
  $$
  We deduce from that point and the $(g\times \Id_{V})$-compatibility of $\varphi$ that
  $\left((\Phi_W \times \Id_{V})^*\phi\right)\cdot\varphi$ is also $(g\times \Id_{V})$-compatible and in particular 
  $<\left((\Phi_W \times \Id_{V})^*\phi\right)\cdot u_{\Phi_W}, \varphi>$ equal to
  $<u_{\Phi_W},\left((\Phi_W \times \Id_{V})^*\phi\right)\cdot\varphi>$ is $g$-integrable
  and by the projection formula we have
  $$
  \begin{array}{ccl} 
	  g_{!}<\left((\Phi_W \times \Id_{V})^*\phi\right)\cdot u_{\Phi_W}, \varphi> & = & 
	  <u_{\Phi_{W'}},(g\times \Id_{V})_{!}(\left((\Phi_W \times \Id_{V})^*\phi\right)\cdot \varphi)> \\
  &= & 
  <u_{\Phi_{W'}},\left((\Phi_{W'}\times \Id_{V})^*\phi\right)\cdot(g\times \Id_V)_{!}\varphi> \\
  & = &
  <\left((\Phi_{W'}\times \Id_{V})^*\phi\right) u_{\Phi_{W'}},(g\times \Id_V)_{!}\varphi>,
  \end{array}
  $$
  and the result follows.
\end{proof}

\begin{example}[Restriction of a definable distribution]
Let $V$ be the definable subset $h[m,0,0]$ for $m>0$. Let $\phi$ be a Schwartz-Bruhat function in $S_{P}(P\times V)$ and $u$ be a definable distribution in $S'_{P}(V)$. Let $X$ be an open definable set of $h[m,0,0]$. By the previous proposition we define the restriction of $u$ to $X$ as the
product $1_{P\times X}\cdot u$ meaning that for any definable morphism $\Phi_W$ in $\Def_P$ and for any Schwartz-Bruhat function in $S_{W}(W\times X)$ we have
$$<(1_{P\times X}u)_{\Phi_W},\varphi> : = <u_{\Phi_W}, 1_{W\times X}\cdot \varphi>.$$
We will denote by $S'_P(X)$ the set of definable distribution $1_{P\times X} \cdot u$.
\end{example}

\begin{example}[Locally integrable functions]  \label{ex-loc-int}
Let $P$ be a definable subset in $\Def_k$. Let $m>0$ be an integer and $X$ be a definable open subset of $h[m,0,0]$.
Let $u$ be a $P$-\emph{locally integrable function} on $P\times X$ (see Definition \ref{loc-integrable}).
For any definable map $\Phi_{W}:W\ra P$ and the projection $\pi_W : W\times X \ra W$,  we denote by $u_{\Phi_W}$ the pull-back $(\Phi_{W}\times Id_X)^*u$. By Proposition \ref{loc-int-SB}, that function induces a $\mathcal C(W)^{\exp}$-linear 
\begin{equation}\label{distributionfct}
	<u_{\Phi_W},\varphi> = \pi_{W!} \left(\left((\Phi_W \times \Id_{X})^{*}u\right)\cdot \varphi\right) \in \mathcal C(W)^{\exp}
\end{equation}
namely, in notation with integrals:
$$ <u_{\Phi_W},\varphi> = w\in W \mapsto \int_{X}u(\Phi_{W}(w),x)\varphi(w,x)dx.$$
We prove that $(u_{\Phi_W})_{\Phi_W \in \Def_P}$ is a definable distribution. Let's prove the compatibility conditions. Let $\Phi_W$ and $\Phi_{W'}$ be two definable subsets in $\Def_{P}$ and
$g:\Phi_W \ra \Phi_{W'}$ be a definable morphism in $\Def_{P}$.
Let $\varphi$ be a Schwartz-Bruhat function in $S_{W'}(W'\times X)$. Using definition \ref{distributionfct} the equality $\Phi_{W'}\circ g = \Phi_{W}$,  Proposition \ref{compatibilite-pull-back-push-forward} we obtain
 $$g^{*}<u_{\Phi_{W'}},\varphi> =  <u_{\Phi_{W}},(g\times \Id_{X})^{*}\varphi>.$$
 Let $\varphi$ be a Schwartz-Bruhat function in $S_{W}(W\times X)$ which is $(g\times \Id_X)$-compatible.
 The constructible function $u_{\Phi_{W}}\cdot\varphi$ is equal to $\left((g\times \Id_X)^*(u_{\Phi_{W'}})\right)\cdot\varphi$. 
 As $\varphi$ is $(g\times Id_X)$-integrable, we deduce from the projection axiom (and remark \ref{remproj}) that $u_{\Phi_{W}}\cdot\varphi$
 is $(g \times Id_X)$-integrable with the equality 
  $$(g\times \Id_X)_{!}(u_{\Phi_W}\cdot\varphi) = u_{\Phi_{W'}}\cdot((g\times \Id_{X})_{!}\cdot\varphi),$$
  As $(g\times \Id_{X})_{!}\varphi$ is a Schwartz-Bruhat function by lemma \ref{push-forwardSB}, the constructible function 
  $u_{\Phi_{W'}}\cdot((g\times \Id_{X})_{!}\cdot\varphi)$ is $\pi_{W'}$-integrable by Proposition \ref{loc-int-SB}, then using the equality 
  $\pi_{W'}\circ (g\times Id_X) = g \circ \pi_{W}$ and Fubini axiom, we deduce that the constructible function
  $<u_{\Phi_W},\varphi>$ is $g$-integrable, with the equality
  $$g_{!}<u_{\Phi_W},\varphi>\: = \: <u_{\Phi_{W'}},(g\times Id_{V})_{!}\varphi>.$$
\end{example}
\begin{example}[Dirac measures]  \label{Dirac}
Let $P$ be a definable subset, $m$ be a positive integer and $X$ be an open definable subset of
$h[m,0,0]$. Fix a point $x_{0}$ in $X$ and $(Id_W\times i_{x_0})^*$ be the restriction morphism from $\mathcal C(W\times X)^{\exp}$ to 
$\mathcal C(W\times \{x_0\})^{\exp}$. For any $\Phi_W$ in $\Def_P$ we define the $\mathcal C(W)^{\exp}$-linear form $\delta_{x_0,\Phi_W}$ by
$$<\delta_{x_0,\Phi_W},\varphi> = (Id_W\times i_{x_0})^*(\varphi).$$
We observe that the family $(\delta_{x_0,\Phi_W})$ is a distribution (keeping track of the rational point $\{x_0\}$).
Let's check the compatibility conditions.
Let $\Phi_W$ and $\Phi_{W'}$ be two definable sets in $\Def_{P}$ and $g:\Phi_W \ra \Phi_{W'}$ be a definable morphism in $\Def_{P}$.
For any Schwartz-Bruhat function $\varphi$ in $S_{W'}(W'\times X)$, using the equality of functions
$$(g\times Id_{X}) \circ (Id_{W}\times i_{x_0}) =  (Id_{W'}\times Id_{x_0})\circ (g \times Id_{x_0})$$
on $W\times \{x_0\}$, we deduce the equality
$$(g\times Id_{x_0})^{*}<\delta_{x_0,\Phi_{W'}},\varphi>=<\delta_{x_0,\Phi_{W}},(g\times \Id_X)^{*}\varphi>.$$

For any $(g\times Id_{X})$-compatible Schwartz-Bruhat function $\varphi$ in $S_{W}(W\times X)$, $\varphi$ is $(g\times Id_X)$-integrable then, (for instance by Proposition 14.2.1 in \cite{CluLoe08a}), the constructible function $(Id_W\times i_{x_0})^*(\varphi)$ is $(g\times Id_{x_0})$-integrable 
with the equality
$$(g\times Id_{x_0})_{!}<\delta_{x_0,\Phi_W},\varphi>=<\delta_{x_0,\Phi_{W'}},(g\times \Id_{X})_{!}\varphi>.$$
\end{example}

\subsection{Definable distribution and average formula}

In the $p$-adic setting, a distribution $u \in S'(\mathbb Z_{p})$ is a linear form on the space $S(\mathbb Z_{p})$ of Schwartz-Bruhat functions on $\mathbb Z_p$ see for instance
Chapter 7 of \cite{Igusabook}. A Schwartz-Bruhat function $\varphi$ on $\mathbb Z_{p}$, is a locally constant function with a bounded support. In particular there is an integer $r$ such that the functions $\varphi$ and $x\mapsto \varphi(x) <u,1_{B(x,l+1)}>$ are constant on any ball of valuative radius $l\geq r$.
As all the balls of valuative radius $l+1$ are disjoint, we can write
$$\varphi = \sum_{(a_{i})\in \mathbb F_{p}^{l+1}} \varphi(a)1_{B(a,l+1)}$$
where for any $(a_{i})$ in  $F_{p}^{l+1}$, $a$ is the sum $\sum_{i=0}^{l+1}a_{i}p^i$.
Then, for any such $a$ we have
$$\int_{B(a,l+1)} \varphi(x) <u,1_{B(x,l+1)}>dx = p^{-(l+1)} \varphi(a) < u, 1_{B(a,l+1)}>,$$
and by Chasles relation we obtain
$$<u,\varphi> =  \sum_{(a_{i})\in \mathbb F_{p}^{l+1}} <u,\varphi(a)1_{B(a,l+1)}> =
p^{(l+1)}\int_{\mathbb Z_{p}}\varphi(x)<u,1_{B(x,l+1)}>dx.$$
We conclude that a distribution is known as soon as it is known on characteristic functions of balls, and furthermore evaluations should satisfy this average formula.
The situation is similar in the motivic case. We introduce the notation

\begin{notation} \label{deftalpha}
 Let $P$ be a definable set in $\Def_k$. Let $V$ be the definable set $h[m,0,0]$ with $m>0$.
 Let $\Phi_W$ be a definable morphism in $\Def_P$.
 For any definable function $\alpha^-$ and $\alpha^+$ from $W$ to $\mathbb Z$, satisfying $\alpha^+ \geq \alpha^-$ 
 we consider
 $$t_{\alpha^-,\alpha^+} := (w,z,x) \mapsto
 1_{B_{\alpha^-}(w)}(x)1_{B_{\alpha^+}(w)}(x-z) \in \mathcal S_{W\times V_z}(W\times V_{z,x}).$$
\end{notation}

\begin{prop}[Motivic average formula]  \label{average-formula}
 Let $P$ be a definable set in $\Def_k$. Let $V$ be the definable set $h[m,0,0]$ with $m>0$.
 Let $u$ be a definable distribution in $S'_{P}(V)$. Let $\Phi_W$ be a definable morphism in $\Def_P$.
 Let $\varphi$ be a Schwartz-Bruhat function in $S_{W}(W\times V_z)$.
 For any definable function $\alpha^-$ and $\alpha^+$ from $W$ to $\mathbb Z$ with $\alpha^+ \geq \alpha^{+}(\varphi)$ and $\alpha^- \leq \alpha^{-}(\varphi)$
 we have the equality
 $$<u_{\Phi_W},\varphi>=
 \pi_{W!}\left(\mathbb L^{\alpha m} \varphi
 <u_{\Phi_W \times V_z},t_{\alpha^-,\alpha^+} > \right)
 \in \mathcal C(W)^{\exp}$$
 with $\pi_W : W\times V_z \ra W$ the canonical projection.
In notation with integrals, the formula is
 $$<u_{\Phi_W},\varphi>=w \mapsto
 \mathbb L^{\alpha^+(w) m} \int_{z\in V_z}\varphi(w,z)
 <u_{\Phi_W\times V_z},1_{B_{\alpha^+(w)}}(\cdot - z)1_{B_{\alpha^-}}>dz.
$$
\end{prop}
\begin{proof}
 By the inequalities
 $$\alpha^+ \geq \alpha^+(\varphi) \geq \alpha^-(\varphi) \geq \alpha^- $$
 and the convolution identity on $\varphi$, we have
 $$\varphi = \mathbb L^{\alpha^+ m}(\varphi * 1_{B_{\alpha^+}}) \cdot 1_{B_{\alpha^-}}=
 \mathbb L^{\alpha^+ m} (\pi_{W}\times \Id_{V_x})_{!} \left(\varphi \cdot t_{\alpha^-,\alpha^+}\right).
 $$
 We denote by $\psi$ the constructible exponential function $\varphi \cdot t_{\alpha^-,\alpha^+}$.
 It is a $(\pi_W\times \Id_{V_x})$-convenient Schwartz-Bruhat function in $S_{W \times V_z}(W\times V_z \times V_x)$ with
 $$ \alpha^{-}(\psi) \colon   \begin{cases}
 W\times V_z \longrightarrow \mathbb Z \\
 (w,z)  \longmapsto  \alpha^{-}(\psi)(w,z) = \alpha^{-}(\varphi)(w)
    \end{cases}
    \:\text{and}\:
 \alpha^{+}(\psi) \colon
    \begin{cases}
W \times V_z \longrightarrow \mathbb Z \\
     (w,z)  \longmapsto  \alpha^{+}(\psi)(w,z) = \alpha^+(w).
    \end{cases}
 $$
The result follows from the push-forward compatibility conditions on $u$.
\end{proof}

\begin{example}[Splitting balls] \label{splitballs}
Let $P$ be a definable set in $\Def_k$.
Let $V$ be the definable set $h[m,0,0]$ with $m>0$.
Let $u$ be a definable distribution in $S'_{P}(V)$.
Let $\Phi_W : W \ra P$ be a definable morphism.
Following the motivic average formula, for any definable function $\beta^+$, $\beta^-$, $\alpha^+$, $\alpha^-$ from $W$ to $\mathbb Z$ satisfying
$\beta^+ \geq \alpha^+ \geq \alpha^- \geq \beta^-$ we have the relation
$$<u_{\Phi_W \times V_z},t_{\alpha^-,\alpha^+}>=(\pi_{W \times V_z })_{!}
\left( \mathbb L^{\beta^+ m} t_{\alpha^-, \alpha^+}
<u_{(\Phi_W \times V_z) \times V_x},s\mapsto t_{\beta^-, \beta^+}(w,x,s)>\right).
$$
\end{example}

\begin{rem}
By the motivic average formula, the following extension theorem shows that it is enough to define a definable distribution $(u_{\Phi_W})$ in $S'_{P}(X)$ only on the Schwartz-Bruhat functions of type $t_{\alpha^-, \alpha^+}$
with convenient compatibility relations.
\end{rem}

\begin{thm}[Extension theorem]  \label{extension-theorem}
Let $P$ be a definable set in $\Def_k$. Let $V$ be the definable set $h[m,0,0]$ with $m>0$.
For any definable morphism $\Phi_W$  in $\Def_P$, we consider a $\mathcal C(W\times V_z)^{\exp}$-linear map
$u_{\Phi_W\times V_z}$ defined on the
$\mathcal C(W\times V_z)^{\exp}$-submodule of $S_{W\times V_z}(W\times V_{z,x})$ generated by the constructible exponential functions $t_{\alpha^-, \alpha^+}$ defined in \ref{deftalpha}, with values in $\mathcal C(W\times V_z)^{\exp}$.
We assume:
\begin{itemize}
\item[$\bullet$] \emph{Integrability condition}. 
For any $t_{\alpha^-,\alpha^+}$ defined in \ref{deftalpha}, the constructible function $<u_{\Phi_{W \times V_z}}, t_{\alpha^-,\alpha^+}>$ is $W$-integrable. 

\item[$\bullet$] \emph{Pullback condition}.
For any definable set $\Phi_W$ and $\Phi_{W'}$ in $\Def_P$, for any definable morphism
$g:W\ra W'$ in $\Def_P$ such that $\Phi_W = g \circ \Phi_{W'}$,
for any definable functions $\alpha^-$ and $\alpha^+$ from $W$ to $\mathbb Z$, satisfying $\alpha^+ \geq \alpha^-$ 
we have 
\begin{equation} \label{pull} (g\times \Id_{V_z})^{*}(<u_{\Phi_{W'}\times V_z},t_{\alpha^-,\alpha^+}>)\: = \:
	<u_{\Phi_W\times V_z},(g\times \Id_{V_z \times V_x})^{*}(t_{\alpha^-,\alpha^+})> \:\in \:\mathcal C(W\times V_z)^{\exp}.
\end{equation}

\item[$\bullet$] \emph{Pushforward condition}. 
	Let $\Phi_W : W \ra P$ be a definable morphism. For any definable function $\beta^+$, $\beta^-$, $\alpha^+$, $\alpha^-$ from $W$ to $\mathbb Z$ with
$\beta^+ \geq \alpha^+ \geq \alpha^- \geq \beta^-$ we have the relation
\begin{equation} \label{push}
	<u_{\Phi_W \times V_z},t_{\alpha^-,\alpha^+}> = 
	(\pi_{W \times V_z })_{!} \left( \mathbb L^{\beta^+ m} t_{\alpha^-, \alpha^+} <u_{(\Phi_W \times V_z) \times V_x},s\mapsto t_{\beta^-, \beta^+}(w,x,s)>\right).
\end{equation}
\end{itemize}
With these assumptions, this family extends uniquely as a definable distribution $(u_{\Phi_W})$ in $S'_{P}(V)$ using the \textit{motivic average formula} : for any $\Phi_W$ in $\Def_P$, $\varphi$ in $S_{W}(W\times V)$
\begin{equation} \label{defmoyenne}
	<u_{\Phi_W},\varphi> :=
 \mathbb L^{\alpha^+(\varphi) m}\pi_{W!}\left( \varphi
 <u_{\Phi_W \times V_z},t_{\alpha^-(\varphi), \alpha^+(\varphi)}> \right)
 \in \mathcal C(W)^{\exp}.
 \end{equation}
\end{thm}

\begin{proof}
Uniqueness comes from the motivic average formula. The existence is the main point to prove.
Let $\Phi_W$ be in $\Def_P$. Let $\varphi$ be a Schwartz-Bruhat function in $S_{W}(W\times V)$.
Using the fact that the constructible exponential function
$<u_{\Phi_W \times V_z},t_{\alpha^-(\varphi),\alpha^+(\varphi)}>$
is $\pi_W$-integrable by the integrability condition, using Proposition \ref{loc-int-SB}, we define
$<u_{\Phi_W},\varphi>$ by equation \ref{defmoyenne}. We check that this definition does not depend on the choice of $\alpha^{-}(\varphi)$ and $\alpha^{+}(\varphi)$. 
Let $\beta^{-}\leq \alpha^{-}(\varphi)$ and $\beta^{+} \geq \alpha^{+}(\varphi)$.
By \ref{push}, we have the equality
$$\mathbb L^{\alpha^+(\varphi)m}\left(\pi_{W}^{W\times V_z}\right)_{!}(\varphi <u_{\Phi_W\times V_z},t_{\alpha^{-},\alpha^{+}}>) =$$ 
$$\mathbb L^{\alpha^+(\varphi)m}\left(\pi_{W}^{W\times V_z}\right)_{!}\left(\varphi \left(\pi_{W\times V_z}^{W\times V_{z,x}}\right)_{!} 
\left(\mathbb L^{\beta^{+}m}t_{\alpha^-,\alpha^+}<u_{(\Phi_W \times V_z)\times V_x},s\to t_{\beta^-,\beta^+}(w,x,s)>\right)\right)$$
Applying Fubini axiom, we obtain
$$\mathbb L^{\alpha^+(\varphi)m}\left(\pi_{W}^{W\times V_z}\right)_{!}(\varphi <u_{\Phi_W\times V_z},t_{\alpha^{-},\alpha^{+}}>) =$$ 
$$\mathbb L^{\beta^+(\varphi)m}\left(\pi_{W!}^{W\times V_x}\right)_{!}\left(\pi_{W\times V_x}^{W\times V_z \times V_x}\right)_{!} 
\left(\mathbb L^{\alpha^{+}m}\varphi t_{\alpha^-,\alpha^+}<u_{(\Phi_W \times V_z)\times V_x},s\to t_{\beta^-,\beta^+}(w,x,s)>\right).$$
By the pull-back condition, we have the equality,
$$<u_{(\Phi_W \times V_z)\times V_x},s\to t_{\beta^-,\beta^+}(w,x,s)> = \left(\pi_{W\times V_x}^{W\times V_z \times V_x}\right)^{*}\left(
<u_{\Phi_W \times V_x},s\to t_{\beta^-,\beta^+}(w,x,s)>
\right)$$
then using the projection axiom (and remark \ref{remproj}) we obtain the equality
$$\mathbb L^{\alpha^+(\varphi)m}\left(\pi_{W}^{W\times V_z}\right)_{!}(\varphi <u_{\Phi_W\times V_z},t_{\alpha^{-},\alpha^{+}}>) =$$ 
$$\mathbb L^{\beta^+(\varphi)m}\left(\pi_{W!}^{W\times V_x}\right)_{!}
\left(<u_{\Phi_W \times V_x},t_{\beta^-,\beta^+}>
\left(\pi_{W\times V_x}^{W\times V_z \times V_x}\right)_{!} 
\left(\mathbb L^{\alpha^{+}m}\varphi t_{\alpha^-,\alpha^+}\right)
\right)
,$$
we conclude applying the convolution condition on $\varphi$ 
$$\mathbb L^{\alpha^+(\varphi)m}\left(\pi_{W}^{W\times V_z}\right)_{!}(\varphi <u_{\Phi_W\times V_z},t_{\alpha^{-},\alpha^{+}}>) = 
\mathbb L^{\beta^+(\varphi)m}\left(\pi_{W!}^{W\times V_x}\right)_{!} \left(<u_{\Phi_W \times V_x},t_{\beta^-,\beta^+}>\varphi \right).$$

It follows by $\mathcal C(W)^{\exp}$-linearity of the integral that $u_{\Phi_W}$ is also
$\mathcal C(W)^{exp}$-linear.
We prove now the compatibility relations.
Let $g:W \ra W'$ be a definable morphism. Let $\Phi_W$ and $\Phi_{W'}$ be two definable sets in $\Def_P$ such that $\Phi_W = \Phi_{W'}\circ g$.
\begin{itemize}
	\item[$\bullet$] 
Let $\varphi$ be a Schwartz-Bruhat function in $S_{W'}(W'\times V)$.
From equality \ref{defmoyenne} for $\Phi_{W'}$ and Proposition \ref{compatibilite-pull-back-push-forward} we get
$$g^{*}<u_{\Phi_{W'}},\varphi> =  \pi_{W'!}
\left((g\times \Id_{V_z})^* \left(\mathbb L^{\alpha^+(\varphi) m}\varphi
<u_{\Phi_{W'}\times V_z},t_{\alpha^-(\varphi),\alpha^+(\varphi)}>\right)\right),$$
applying the pull-back assumption we obtain
$$g^{*}<u_{\Phi_W'},\varphi> =  \pi_{W'!} \left(\mathbb L^{(\alpha^+(\varphi) \circ g) m}
 <u_{\Phi_{W}\times V_z},t_{\alpha^-(\varphi)\circ g, \alpha^+(\varphi)\circ g}>(g\times \Id_V)^{*}\varphi\right)$$
 and using equality \ref{defmoyenne} we conclude
$$g^{*}<u_{\Phi_{W'}},\varphi> =<u_{\Phi_W},(g\times \Id_{V_z})^{*}\varphi>.$$
by the equalities
$\alpha^{+}((g\times \Id_{V_z})^{*}\varphi)=g^{*}\alpha^{+}(\varphi)$ and $\alpha^{-}((g\times \Id_{V_z})^{*}\varphi)=g^{*}\alpha^{-}(\varphi)$.

\item[$\bullet$] Let $\varphi$ be a $(g\times \Id_V)$-convenient Schwartz-Bruhat function in $S_{W}(W\times V)$.
Let's prove that $<u_{\Phi_W},\varphi>$ is $g$-integrable.
By definition, there are two definable functions $\beta^+$ and $\beta^-$ from $W'$ to $\mathbb Z$ such that
$\beta^+ \circ g = \alpha^+(\varphi)$ and
$\beta^- \circ g = \alpha^{-}(\varphi)$.
By Proposition \ref{push-forwardSB} and its proof, the constructible exponential function
$(g\times \Id_{V_z})_{!}\varphi$ is a Schwartz-Bruhat function in $S_{W}(W\times V)$ with $\alpha^{+}((g\times \Id_{V_z})_{!}\varphi)$ equal to $\beta^+$ and
$\alpha^{-}((g\times \Id_{V_z})_{!}\varphi)$ equal to $\beta^-$.
By the relation $\Phi_W = \Phi_{W'}\circ g$ and the compatibility relation on pull-back we have the equality
$$<u_{\Phi_W \times V_z,t_{\alpha^-(\varphi),\alpha^+(\varphi)}}> =
(g\times \Id_{V_z})^* <u_{\Phi_{W'} \times V_z},t_{\beta^-,\beta^+}>.
$$
But the Schwartz-Bruhat function $\varphi$ is $(g\times \Id_{V_z})$-integrable, then by the projection axiom (and the remark \ref{remproj}) the constructible exponential function
$\left((g\times \Id_{V_z})^*<u_{\Phi_{W'}\times V_z},t_{\beta^-,\beta^+}>\right)\varphi \mathbb L^{\alpha^+(\varphi)m}$
is also $(g\times \Id_{V_z})$-integrable with the equality in $\mathcal C(W'\times V_z)$
\begin{equation} \label{*equality}
 (g\times \Id_{V_z})_! \left((g\times \Id_{V_z})^*<u_{\Phi_{W'}\times V_z},t_{\beta^-,\beta^+}>\varphi \mathbb L^{\alpha^+m}\right)
 = <u_{\Phi_{W'}\times V_z},t_{\beta^-,\beta^+}> (g\times \Id_{V_z})_{!}(\varphi \mathbb L^{\alpha^+ m}).
\end{equation}
This constructible function is $W'$-integrable by application of the integrability assumption, Proposition \ref{loc-int-SB} and the fact that the constructible function $(g\times \Id_{V_z})_{!}(\varphi \mathbb L^{\alpha^+ m})$ is in $S_{W'}(W'\times V)$. 
Thus, by Fubini, the constructible exponential function
$<u_{\Phi_W \times V_z},t_{\alpha^-(\varphi),\alpha^+(\varphi)}>\varphi \mathbb L^{\alpha^+(\varphi) m}$ is
 $\pi_{W'}\circ (g\times \Id_{V_z})$-integrable, namely
 $(g\circ \pi_W)$-integrable,
 but it is also $\pi_W$-integrable by the integrability condition and proposition \ref{loc-int-SB}, then, by Fubini,
 $<u_{\Phi_W},\varphi>$ is $g$-integrable
 and by  equality (\ref{*equality}) we get
 $$g_{!}<u_{\Phi_W},\varphi> = <u_{\Phi_{W'}},(g\times \Id_{V_z})_{!}\varphi>,$$
 which finishes the proof.
 \end{itemize}
\end{proof}

\begin{rem}  \label{remark-extension}
Note that, a priori, the data $\alpha^+(\varphi)$ and $\alpha^-(\varphi)$ are not canonically defined. Following the context, we can fix for any Schwartz-Bruhat function such data. We can use the extension theorem to obtain a definable distribution. Thanks to the motivic average formula we obtain the independence of the definable distribution from the data.
\end{rem}

\subsection{Fourier transform on definable distributions}
\begin{defn}[Fourier transform of a definable distribution]
Let $P$ be a definable set. Let $V$ be the definable subset $h[m,0,0]$ for a positive integer $m$.
Let $(u_{\Phi_{W}})$ be a definable distribution in $S'_{P}(V_\xi)$. We define $\mathcal F(u)$ in $S'_{P}(V_x)$ as
the definable distribution $(\mathcal F(u_{\Phi_{W}}))$ where for any definable $\Phi_{W}$ in $\Def_{P}$,
$$\forall \varphi \in S_{W}(W\times V),\:\: <\mathcal F u_{\Phi_{W}},\varphi>:=
<u_{\Phi_{W}},\mathcal F \varphi> \in \mathcal C(W)^{\exp}.$$
\end{defn}
\begin{proof}
Let $\Phi_W$ and $\Phi_{W'}$ be two definable subsets in $\Def_{P}$ and $g:W \ra W'$ be a definable morphism with
$\Phi_W = g\circ \Phi_{W'}$.
Let $\varphi$ be a Schwartz-Bruhat function in $S_{W'}(W'\times V_x)$. By the definition and compatibility condition on $u$ we have
$$g^{*}<\mathcal F u_{\Phi_{W'}},\varphi> = <u_{\Phi_W},(g\times \Id_{V_\xi})^{*}\mathcal F \varphi>.$$
  By Proposition \ref{compatibilite-pull-back-push-forward} we obtain the equality
  $$(g\times \Id_{V_{\xi}})^{*}(\mathcal F \varphi) = \mathcal F ((g\times \Id_{V_{x}})^{*} \varphi) :
  (w,\xi) \mapsto \int_{x\in V}\varphi(g(w),x)E(x\mid \xi)dx$$
  which implies the equality
  $$g^{*}<\mathcal F u_{\Phi_{W'}},\varphi> = <\mathcal F u_{\Phi_W},(g\times \Id_{V_x})^{*}\varphi>.$$
  Let $\varphi$ be a $(g\times \Id_{V_x})$-compatible Schwartz-Bruhat function $S_{W}(W \times V_x)$.
  By Fubini as in Example \ref{ex-loc-int} we obtain the equality
  $$(g\times \Id_{V_{\xi}})_{!}\left(\mathcal F \varphi \right) =
  \mathcal F\left( (g\times \Id_{V_{x}})_{!} \varphi \right) \in
  S_{W'}(W'\times V_\xi).$$
  By the definition and push-forward relation we have
  $$
  g_{!}<\mathcal F u_{\Phi_W},\varphi > =
  <u_{\Phi W'},(g\times \Id_{V_{\xi}})_{!}\left(\mathcal F \varphi \right)>
  = <\mathcal F u_{\Phi_W'},(g\times \Id_{V_{\xi}})_!\varphi >.
  $$
\end{proof}

\begin{prop}
Let $P$ be a definable set. Let $V$ be the definable subset $h[m,0,0]$ for a positive integer $m$.
The Fourier transform is an isomorphism on $S'_{P}(V)$ with inverse defined for any definable distribution $u$ in $S'_{P}(V)$ by: for any definable $\Phi_{W}$ in $\Def_{P}$,
$$\forall \varphi \in S_{W}(W\times V),\:\: <\overline{\mathcal F} u_{\Phi_{W}},\varphi>:=
<u_{\Phi_{W}},\mathbb L^{m} \mathcal F \check{\varphi}>.$$
where $\check{\varphi}$ is $(w,x)\mapsto \varphi(w,-x)$.
\end{prop}
\begin{proof}
 This follows from the isomorphism of the Fourier transform at the level of Schwartz-Bruhat functions.
\end{proof}

The following theorem is a motivic analogous of the usual Paley-Wiener theorem in real analysis. It will be used in the proof of Theorem \ref{projectionWF}.

\begin{thm}[Motivic Paley-Wiener theorem]  \label{rem-distrib}
Let $P$ be a definable set. Let $V_x$ and $V_\xi$ be
the definable subset $h[m,0,0]$ with $m$ a positive integer.
Let $u$ be a definable distribution in $S'_{P}(V_x)$.
Let $\varphi$ be a Schwartz-Bruhat function in $S_{P}(P\times V_x)$. Denote by $u_\varphi$
the constructible exponential function $<u_{P\times V_\xi},\varphi E( \cdot \mid  \cdot)>$ in $\mathcal C(P\times V_\xi)^{\exp}$ with
$$\varphi E( \cdot \mid  \cdot) :(p,x,\xi)\mapsto \varphi(p,x)E(x\mid \xi).$$
If there is a definable function $\alpha^{-}(u_\varphi)$ from $P$ to $\mathbb Z$ such that
$u_\varphi = u_\varphi 1_{B_{\alpha^{-}(u_\varphi)}}$, then the constructible exponential function $u_\varphi$ is a Schwartz-Bruhat function in $S_{P}(P\times V_\xi)$ and as a definable distribution
$\overline{\mathcal F}u_\varphi$ is equal to $(\mathbb L^{-m}\varphi) u.$
\end{thm}

\begin{proof}
We prove that  $u_\varphi$ is a Schwartz-Bruhat function in $S_{P}(P\times V_\xi)$ with data $\alpha^{-}(u_\varphi)$ given by assumption, and 
$\alpha^{+}(u_\varphi)$ be the definable function from $P$ to $\mathbb Z$ equal to $\max (1-\alpha^{-}(\varphi),\alpha^{-}(u_\varphi)).$
\begin{itemize}
	\item[$\bullet$] For any definable function $\beta^-$ from $P$ to $\mathbb Z$, with $\beta^-\leq \alpha^{-}(u_\varphi)$  we have the equalities
$$u_{\varphi} \cdot 1_{B_{\beta^-}} = \left(u_{\varphi} \cdot 1_{B_{\alpha^{-}(u_\varphi)}}\right) \cdot 1_{B_{\beta^-}} = 
u_{\varphi} \cdot 1_{B_\alpha^{-}(u_\varphi)}=u_{\varphi}.$$
\item[$\bullet$] We denote by $\phi$ the constructible exponential function
$$\phi : (p,x,\xi) \mapsto \varphi(p,x)E(x\mid \xi)1_{B_{\alpha^{-}(u_\varphi)(p)}}(\xi).$$
By Proposition \ref{SB-exp}, $\phi$ is a Schwartz-Bruhat function in $S_{P\times V_\xi}(P\times V_{\xi,x})$
with $\alpha^{+}(\phi)$ chosen as $\max(\alpha^{+}(\varphi),-\alpha^{-}(u_\varphi))$ and $\alpha^{-}(\phi)$ equal to $\alpha^{-}(\varphi)$.
With that choice, this function is $(\pi_{P}\times \Id_{V_x})$-compatible. By the push-forward condition \ref{definable distribution} on the definable distribution $u$, we deduce that the constructible exponential function $u_\varphi$ is $\pi_P$-integrable.
\item[$\bullet$]Let $\beta^+$ be a definable function from $P$ to $\mathbb Z$ with $\beta^+ \geq \alpha^{+}(u_\varphi)$.
By definition of the convolution product and the pull-back conditions on $u$ we have
$$u_{\varphi} * 1_{B_\beta}  =
\pi_{P\times V_{\eta}!}
\left<u_{P\times V_{\eta,\xi}}, \psi \right>.$$
where
$$ \Psi : (p,\eta,\xi,x)\mapsto \phi(p,x,\xi)1_{B_{\beta^+(p)}}(\eta-\xi)$$
is $(\pi_{P\times V_\eta} \times \Id_{V_x})$-compatible Schwartz-Bruhat function in $S_{P \times V_\eta}(P \times V_{\eta, \xi,x})$.
By the push-forward compatibility on $u$ we have
$$u_{\varphi} * 1_{B_\beta^+}  =
\pi_{P\times V_{\eta}!}
\left<u_{P\times V_{\eta,\xi}}, \Psi \right>
=
\left<u_{P\times V_{\eta}},(\pi_{P\times V_{\eta}} \times \Id_{V_x})_{!}\Psi \right> = \mathbb L^{-m\beta^+}u_{\varphi}.
$$
Indeed, using the inequality
$\beta^+ \geq \alpha^{-}(u_\varphi)$ we obtain the relation
$$
1_{B_{\alpha^{-}(u_\varphi)}}(\xi)1_{B_{\beta^+}}(\eta-\xi) =
1_{B_{\alpha^{-}(u_\varphi)}}(\xi)1_{B_{\alpha^{-}(u_\varphi)}}(\eta)1_{B_{\beta^+}}(\eta-\xi)
=1_{B_{\alpha^{-}(u_\varphi)}}(\eta)1_{B_{\beta^+}}(\eta-\xi)
$$
which implies by the projection axiom \ref{axiom-projection} and exponential properties \ref{exponential}
$$ (\pi_{P\times V_{\eta}} \times \Id_{V_x})_{!} \Psi = \phi \mathcal F(1_{B_{\beta^+}}) $$
and we conclude by the relations $\mathcal F(1_{B_{\beta^+}}) = \mathbb L^{-m\beta^+}1_{B_{1-\beta^+}}$ and
$\beta^+\geq 1-\alpha^-(\varphi)$.
Thus, the constructible exponential function $u_\varphi$ is a Schwartz-Bruhat function in $S_{P}(P\times V_\xi)$.
\end{itemize}
Let's prove that as a distribution $\overline{\mathcal F}u_\varphi$ is equal to $(\mathbb L^{-m}\varphi) u.$
By Proposition \ref{produitparSB} and Example \ref{ex-loc-int} we want to prove the following equality 
for any $\Phi_W$ in $\Def_P$ and $\psi$ in $S_{W}(W\times V)$
$$\pi_{W!}\left((\Phi_W\times \Id_{V_y})^{*}\left(\overline{\mathcal F} u_\varphi\right) \cdot \psi\right) =
 <u_{\Phi_W},\left(\left(\Phi_W\times \Id_{V_x}\right)^{*}\left(\mathbb L^{-m}\varphi\right)\right) \psi>.$$
 
Let $\beta^-$ be a definable function from $W$ to $\mathbb Z$ such that $\beta^- \leq \alpha^{-}(u_\varphi)$ and
 $1-(\beta^{-}\circ \Phi_W)\geq \alpha^{+}(\psi)$.

Using the definition of $\overline{\mathcal F}u_\varphi$, the pull-back and push-forward compatibility relations on $u$ and the Fourier transform of the characteristic function $1_{B_{\beta^-}}$ we have
$$\overline{\mathcal F}u_\varphi  =
\left<u_{P \times V_y}, \mathbb L^{-m\beta^-} \varphi \left(1_{B_{1-\beta^-}}\circ d\right) \right>.
$$
By the pull-back compatibility relation we have
$$(\Phi_W\times \Id_{V_y})^{*}\left(\overline{\mathcal F} u_\varphi\right) =
<u_{ \Phi_W \times V_y},
(\Phi_W\times \Id_{V_{y,x}})^{*}(\mathbb L^{-m\beta^-} \varphi (1_{B_{1-\beta^-}}\circ d))>.$$
We consider now the constructible exponential function
$$\Psi = (w,x,y) \mapsto (\Phi_W\times \Id_{V_{y,x}})^{*}(\mathbb L^{-m\beta^-} \varphi [1_{B_{1-\beta^-}}\circ d])(w,x,y)\psi(w,y) $$
which is a $(\pi_{W}\times \Id_{V_x})$-compatible Schwartz-Bruhat function in $S_{W\times V_y}(W\times V_{y,x})$ with
$$(\pi_{W}\times \Id_{V_x})_{!}\Psi = \mathbb L^{-m}\left((\Phi_W\times \Id_{V_x})^*\varphi\right)\psi,$$
using the convolution formula on $\psi$ with the assumption $1-(\beta^-\circ \Phi_W) \geq \alpha^{+}(\psi)$.\\
By Example \ref{ex-loc-int}, the previous point and the pull-back compatiblity relation 
we conclude
$$<(\Phi_W\times \Id_{V_y})^{*}\overline{\mathcal F} u_\varphi,\psi> =
\pi_{W !}(\left((\Phi_W\times \Id_{V_y})^{*}\left(\overline{\mathcal F} u_{\varphi}\right)\right)\cdot\psi) = <u_{\Phi_W},\mathbb L^{-m}\left((\Phi_W\times \Id_{V_y})^{*}\varphi\right)\psi>.$$
\end{proof}

\section{Motivic wave front sets of a definable distribution}  \label{section-microlocal}
\subsection{Notations}
Let $V$ be the definable set $h[m,0,0]$ for a positive integer $m$. Let $T^{*}(V)\setminus \{0\}$ be the \emph{cotangent space} $V\times (V\setminus \{0\})$ of $V$ without the zero section.
We denote by $(x,\xi)$ the elements of $T^{*}(V)$.
For a positive integer $\n$ we consider the definable subgroup $\Lambda_{\n}$ of $(h[1,0,0]\setminus \{0\},\times)$ (defined in definition \ref{Lambdan}) and
the closed and bounded definable set of $V_\xi$
$$\mathscr B_{\mathfrak n}:=
\bigsqcup_{r=0}^{\mathfrak n-1} B_{r}\setminus B_{r+1}=\{\xi \in V \mid 0\leq \ord \xi \leq \mathfrak n-1\},$$
with $\ord \xi = \min \{\ord \xi_i \mid 1 \leq i \leq m\}$.
\begin{rem}
 For any $\xi$ in $V\setminus \{0\}$, there is a unique $r$ in $\{0, \dots, \n-1\}$ equal to the rest of the euclidean division of $\ord \xi$ by $\n$, such that there is $\lambda$ in $\Lambda_\n$ and $\tilde{\xi}$ in $B_{r}\setminus B_{r+1}$ with $\xi = \lambda \tilde{\xi}$.
\end{rem}

\subsection{Singular support of a definable distribution}

\begin{defn}[Smooth point and singular support]  \label{singular-support}
Let $P$ be a definable set and $V$ be the definable subset $h[m,0,0]$ for a positive integer $m$.
A definable distribution $u$ in $S'_{P}(V)$ is said to be \emph{smooth} at a point $x$ in $V$ if and only if
there are a definable function $r_{x}$ from $P$ to $\mathbb Z$
and a Schwartz-Bruhat function $\psi$ in $S_P(P\times V)$
 such that the definable distributions $1_{B(x,r_x)}u$ and $1_{B(x,r_x)}\psi$ are equal. Namely
 for any
 definable set $\Phi_{W}$ in $\Def_{P}$,
 for any Schwartz-Bruhat function $\varphi$ in $S_{W}(W\times V)$ we have
 $$ <u_{\Phi_{W}},\left((\Phi_W \times \Id_{V})^*1_{B(x,r_x)}\right)\cdot\varphi>=<(\Phi_W\times \Id_V)^{*}\left(1_{B(x,r_x)}\psi\right),\varphi>.$$
 The complement of the set of smooth points is called \emph{singular support} of $u$ denoted by
  $\SS\:u$. 
\end{defn}

\begin{rem}
As the smoothness condition is open, the singular support is a closed subset of $X$. These sets are not defined by a first order condition and a priori are not definable.
\end{rem}

\subsection{Motivic microlocal smooth data of a definable distribution}
We recall first the definition by Heifetz in \cite{Hei85a} of the $p$-adic wave front set of a definable distribution.
\begin{defn}[$p$-adic wave front set of a definable distribution]
 Let $u$ be a definable distribution in $S'(\mathbb Q_{p}^m)$.
 Let $\Lambda$ be a subgroup of $\mathbb Q_p^*$ with finite index.
 Let $\Psi$ a non trivial additive character on $\mathbb Q_p$ which is trivial on $p\mathbb Z_p$.
 A point $(x_0,\xi_0)$ in $T^*(\mathbb Q_p^m)$ with $\xi_0 \neq 0$ is called
 \emph{$\Lambda$-microlocally smooth},
 if and only if there are integers $r_{x_0,\xi_0}>0$ and $\check{r}_{x_0,\xi_0}>0$ such that for any Schwartz-Bruhat function $\varphi$ in $S(B(x_0,r_{x_0,\xi_0}))$, there is an integer $N_\varphi$ such that for any $\lambda$ in $\Lambda$
 $$\ord \lambda \leq N_\varphi \:\:\Rightarrow \:\:
 <u,\varphi \Psi( \cdot \mid \lambda \xi)>1_{B(\xi,\check{r}_{x_0,\xi_0})} = 0.$$
 The complement in $T^{*}(\mathbb Q_{p}^m)\setminus \{0\}$ of the set of $\Lambda$-microlocally smooth points is called \emph{$\Lambda$-wave front set} of $u$.
\end{defn}

\begin{rem}
This definition implies two problems in the $k\llp t \rrp$-setting.
\begin{enumerate}
\item This definition is local and not global, and globalisation arguments use the compactness of the $p$-adic sphere.
\item Furthermore, the induced functions $r$ and $\check r$ in $(x_0,\xi_0)$ are a priori not definable, because the microlocal statement is not first order.
\end{enumerate}
\end{rem}
As a solution of these problems, we introduce the following notion.

\begin{defn}[Motivic microlocal smooth data]   \label{mls-data}
Let $V$ be the definable set $h[m,0,0]$  with $m$ a positive integer. Let $P$ be a definable set in $\Def_k$.
We assume $P$ and $\mathbb Z$ endowed with the discrete topology.
Let $u$ be a definable distribution in $S'_{P}(V)$. A \emph{$\Lambda_\n$-microlocally smooth data} of $u$ is a quadruple $(\mA,r,\check{r},N)$ with
\begin{enumerate}
	\item \label{mAdef} $\mA$ is a definable subset of $P\times V_x \times \Bn$.
	\item \label{condr} $r:\mA\ra \mathbb Z$ and $\check{r}:\mA \ra \mathbb Z$ are two definable and continuous maps such that for any $(p,x_0,\xi_0)$ in $\mA$, 
		the product $$\{p\}\times B(x_0,r(p,x_0,\xi_0)) \times B(\xi_0,\check{r}(p,x_0,\xi_0))$$ 
		is contained in $\mA$ and we have the inequalities $\check{r}(-,-,\xi_0)\geq \n \geq \ord \xi_0 + 1.$
	\item \label{condN} $N:\mathcal B\ra \mathbb Z$ is a definable and continuous map with 
		$$\mathcal B=\{\left((p,x_0,\xi_0),x',r'\right)\in \mA \times V_{x'} \times \mathbb Z \mid B(x',r')\subset B(x_0,r(p,x_0,\xi_0))\}$$
\end{enumerate}
such that we have the equality of constructible exponential functions in $\mathcal C(\Lambda_\n \times \D)^{\exp}$
$$\left(<u_{\Lambda_\n \times \D},T>1_{\E}\right)1_{B_N} = <u_{\Lambda_\n \times \D},T>1_{\E}$$
where
 $$\E=\left\{\left((p,x_0,\xi_0),x',r',\xi\right)\in \mA \times V_{x'} \times \mathbb Z \times V_{\xi} \: \:
 \begin{array}{|l}
 B(x',r')\subset B(x_0,r(p,x_0,\xi_0))\\
 \xi \in B(\xi_0,\check{r}(p,x_0,\xi_0))
 \end{array}
 \right\},$$
$$\D = P \times V_{x_0} \times V_{\xi_0} \times V_{x'} \times \mathbb Z \times V_{\xi}$$
$T$ is the Schwartz-Bruhat function in
 $S_{\Lambda_\n \times \D}(\Lambda_\n \times \D \times V)$ defined by
 $$T(\lambda,\left(p,x_0,\xi_0,x',r',\xi \right),x)=
 1_{B(x',r')}(x)E(x\mid \lambda \xi),$$
 and
 $$B_{N}=\{(\lambda,\left(p,x_0,\xi_0,x',r' \right),\xi)\in \Lambda_\n \times \mathcal B \times V_\xi \mid
 \ord \lambda \geq N(p,x_0,\xi_0,x',r')\}.$$
\end{defn}

\begin{rem}
 In the previous definition and below we will denote in the same way, $1_\E$ as a function on
 $\mathcal C(\Lambda_\n \times \mathcal D)^{\exp}$ or in $\mathcal C(\mathcal D)^{\exp}$. As well, we will denote in the same way $1_{B_N}$ as
 a function on $\mathcal C(\Lambda_\n \times \mathcal D)^{\exp}$ or in $\mathcal C(\Lambda_\n \times \mathcal B)^{\exp}$.
\end{rem}
\begin{prop}[Restriction of a $\Lambda_\n$-definable data]  \label{restriction-data}
 Let $P$ be a definable set. Let $V$ be the definable subset $h[m,0,0]$ with $m>0$. Let $u$ be a definable distribution in $S'_{P}(V)$.
 Let $( \mathcal A_1,r_{1},\check{r}_{1},N_{1})$ be a $\Lambda_\n$-definable data of $u$.
 Let $(\mathcal A_2,r_{2},\check{r}_{2},N_{2})$ be a quadruple such that
 $\mathcal A_2$ is a definable subset of $\mathcal A_1$,
  $r_2$ and $\check{r}_2$ are two definable continuous maps from $\mathcal A_2$ to $\mathbb Z$ such that $r_2 \geq r_1$ and
  $\check{r}_2 \geq \check{r}_1$. We assume also that for any $(p,x_0,\xi_0)$ in $\mathcal A_2$, the product
  $$\{p\}\times B(x_0,r_{2}(p,x_0,\xi_0)) \times (B(\xi_0,\check{r}_{2}(p,x_0,\xi_0)) \cap \Bn)$$
  is contained in $\mathcal A_2$ and
  $N_{2}$ is a definable continuous map from
  $\mathcal B_{2} = \mathcal B\cap (\mathcal A_2 \times V_{x'} \times \mathbb Z)$ to $\mathbb Z$
  such that $N_2 \leq N_{1\mid \mathcal B_{2}}$.
 With these assumptions, $(\mathcal A_2,r_2,\check{r}_2,N_2)$ is a $\Lambda_\n$-definable data.
 \end{prop}
\begin{proof}
 By assumptions the constructible exponential function $1_{B_{N_1\mid \mathcal B_2}}$ is equal to $1_{B_{N_2}}1_{B_{N_1\mid \mathcal B_2}}$ and
 $$\E_2:=\left\{(p,x_0,\xi_0,x',r')\in \mathcal B_2 \times V_\xi \:
 \begin{array}{|l}
  B(x',r') \subset B(x_0,r_2(p,x_0,\xi_0)) \subset B(x_0,r_1(p,x_0,\xi_0)) \\
  \xi \in B(\xi_0,\check{r}_2(p,x_0,\xi_0)) \subset B(\xi_0,\check{r}_1(p,x_0,\xi_0))
 \end{array}
\right\}
 $$
 is contained in $\E_1$ giving the equality $1_{\mathcal E_2}1_{\mathcal E_1}=1_{\mathcal E_2}$.
Using these equalities and the fact that $(\mathcal A,r_1,\check r_1,N_1)$ is a $\Lambda_\n$-definable data, we obtain the equality
$$<u_{\Lambda_\n \times \mathcal D},T>1_{\mathcal E_2}1_{B_{N_2}} = <u_{\Lambda_\n \times \mathcal D},T>1_{\mathcal E_2}.$$
\end{proof}

\subsection{Motivic wave front set}

 \begin{defn}[$\Lambda_\n$-motivic microlocally smooth points]  \label{microlocally smooth point}
  \label{motivic wave front set}
Let $u$ be a definable distribution in $S'_{P}(V)$ with $V$ equal to $h[m,0,0]$ with $m$ a positive integer.

$\bullet$ A point $(x_0,\xi_0)$ in $V\times \Bn$ is a \emph{$\Lambda_\n$-microlocally smooth point} of $u$ if and only if there are definable functions  $r_{x_0,\xi_0}$ and $\check{r}_{x_0,\xi_0}$ from $P$ to $\mathbb Z_{\geq \n}$, there is a continuous and definable function
$N_{x_0,\xi_0}:B_{x_0,\xi_0} \to \mathbb Z$ from the definable set
$$B_{x_0,\xi_0} = \{(p,x',r')\in P \times V_{x'} \times \mathbb Z \mid B(x',r') \subset B(x_0,r_{x_0,\xi_0}(p)) \}$$
such that
$$\left(<u_{\Lambda_\n \times \D_{x_0,\xi_0}},T_{x_0,\xi_0}>1_{\E_{x_0,\xi_0}}\right)\cdot 1_{B_{N_{x_0,\xi_0}}} =
<u_{\Lambda_\n \times \D_{x_0,\xi_0}},T_{x_0,\xi_0}>1_{\E_{x_0,\xi_0}}$$
where $\D_{x_0,\xi_0}$ is the product
$P \times V_{x'} \times \mathbb Z \times V_{\xi},$
$$\E_{x_0,\xi_0}=\left\{\left[(p,x',r'),\xi \right] \in \mathcal B_{x_0,\xi_0} \times V_{\xi} \: \mid \:
 \xi \in B(\xi_0,\check{r}_{x_0,\xi_0}(p)) \right\},$$
$T_{x_0,\xi_0}$ is the Schwartz-Bruhat function in
 $S_{\Lambda_\n \times \D_{x_0,\xi_0}}(\Lambda_\n \times \D_{x_0,\xi_0} \times V)$ defined by
 $$T_{x_0,\xi_0}(\lambda,\left[p,x',r',\xi \right],x)=
 1_{B(x',r')}(x)E(x\mid \lambda \xi),$$
 and
 $$B_{N_{x_0,\xi_0}}=\{(\lambda,\left[p,x',r' \right],\xi)\in \Lambda_\n \times \mathcal B_{x_0,\xi_0} \times V_\xi \mid
 \ord \lambda \geq N_{x_0,\xi_0}(p,x',r')\}.$$

 $\bullet$ A point $(x_0,\xi_0)$ in $T^*(V)$ is $\Lambda_n$-microlocally smooth if and only if there is a $\lambda$ in $\Lambda_{\n}$ such that $\lambda \xi_0$ belongs to $\Bn$ and $(x_0,\lambda \xi_0)$ is $\Lambda_n$-microlocally smooth.
 We denote by $\mathcal S_{\Lambda_\n}(u)$ the set of $\Lambda_\n$-microlocally smooth points of $u$.\\
\end{defn}

\begin{rem} \label{egalitepointmicrolocallisse}
By restriction, a point $(x_0,\xi_0)$ in $ V\times \Bn$ which belongs to the underset $A$ of a $\Lambda_\n$-microlocally smooth data
$(\mA,r,\check{r},N)$ is a $\Lambda_\n$-microlocally smooth point. Inversely, let $(x_0,\xi_0)$ be a $\Lambda_\n$-microlocally smooth point with data $r_{x_0,\xi_0}$,
$\check{r}_{x_0,\xi_0}$ and $N_{x_0,\xi_0}$. We consider the definable set
$$\mathcal A=\{(p,x,\xi)\in P\times V\times \Bn \mid \ord(x-x_0)\geq r(p,w_0,\xi_0), \ord(\xi-\xi_0)\geq \check{r}(p,x_0,\xi_0)\}$$
and the continuous definable functions
$$ r \colon
 \begin{cases}
  \mathcal A \longrightarrow \mathbb Z \\
  (p,x,\xi)  \longmapsto & r(p,x_0,\xi_0)
 \end{cases}
 \:\:,\:\:
  \check r \colon
\begin{cases}
 \mathcal A \longrightarrow \mathbb Z \\
 (p,x,\xi)  \longmapsto  \ord \check{r}(p,x_0,\xi_0)
\end{cases}
,\:
 N \colon
 \begin{cases}
 \mathcal B \longrightarrow \mathbb Z \\
 (p,x_0,\xi_0,x',r')  \longmapsto  N_{x_0,\xi_0}(p,x',r').
 \end{cases}
 $$
The data $(\mA,r,\check{r},N)$ is a $\Lambda_{\n}$-definable data of $u$.
\end{rem}

\begin{defn}[$\Lambda_\n$-motivic wave front set]
Let $V$ be the definable set $h[m,0,0]$ with $m$ a positive integer.
Let $u$ be a definable distribution in $S'_{P}(V)$.
The \emph{$\Lambda_\n$-motivic wave front set} of $u$ denoted by $\WF_{\Lambda_\n}(u)$ is the complement of the set of the $\Lambda_\n$-microlocally smooth points of $u$ in $T^{*}(V)\setminus \{0\}$. 
\end{defn}

\begin{rem}
 By definition the set of $\Lambda_\n$-microlocally smooth points of $u$ and
 the $\Lambda_n$ motivic wave front set of $u$ are conical, in the following, we will just consider covector $\xi$ in $\Bn$. It follows also from the remark \ref{egalitepointmicrolocallisse} that 
 $$WF_{\Lambda_\n}(u) \cap (V\times \Bn) = \bigcap_{\mA \in \mathscr A}\pi_{V\times \Bn} (\mA^{c}),$$
where $\mathscr A$ is the set of support of $\Lambda$-definable microlocally smooth $\mA$, and for such support of $\mA$, $\mA^{c}$ denotes its complement in 
$P\times V\times \Bn$.
\end{rem}

\begin{rem}[Recovering Heifetz definition]  \label{recovering Heifetz}
Assume here the parameter set $P$ is a point. Let $u$ be a definable distribution in $S'(V)$ with $V$ the definable set $h[m,0,0]$ with $m>0$.
Let's prove that any $\Lambda_\n$-motivic microlocally smooth point of $u$ is also $\Lambda_\n$-microlocally smooth in the sense of Heifetz. Let $(x_0,\xi_0)$ be a $\Lambda_\n$-motivic microlocally smooth point of $u$ and consider the data $r_{x_0,\xi_0}$, $\check{r}_{x_0,\xi_0}$ and $N(x_0,\xi_0)$ given in definition \ref{microlocally smooth point} 
By application of the definable compactness Lemma \ref{definablecompactness} on the continuous and definable function $N(x_0,\xi_0)$, for any Schwartz-Bruhat function $\varphi$ in $S(B(x_0,r_{x_0,\xi_0}))$ we define
$$N_\varphi = \min_{z\in B(x_0,r_{x_0,\xi_0})}N_{x_0,\xi_0}(z,\alpha^{+}(\varphi)).$$
We denote by $B_{x_0}$ and $B_{\xi_0}$ the balls $B(x_0,r_{x_0,\xi_0})$ and $B(\xi_0,\check{r}_{x_0,\xi_0})$. Using the convenient notation $[x_0,\xi_0]$ for the product $\{(x_0,\xi_0)\}\times\Lambda_\n\times V_{\xi}$, we prove the equality
\begin{equation}  \label{Heifetz-formula}
<u_{[x_0,\xi_0]},\varphi E(\cdot \mid \cdot )>
1_{B_{\xi_0}} 1_{B_{N_\varphi}} =
<u_{[x_0,\xi_0]},\varphi E( \cdot \mid \cdot)>
1_{B_{\xi_0}}
\end{equation}
with $$\varphi E( \cdot \mid \cdot ) : (\lambda,\xi,x) \mapsto \varphi(x) E(x \mid \lambda \xi).$$
Let $\psi$ be the constructible exponential function
$$\psi : (\lambda,\xi,z,x)\mapsto 1_{B_{\alpha^{+}(\varphi)}}(x-z) E(x \mid \lambda \xi)$$
which is a $(\pi_{\Lambda_\n\times V_\xi}\times \Id_{V_x})$-compatible
Schwartz-Bruhat function in $S_{\Lambda_\n \times V_{\xi,z}}(\Lambda_\n \times V_{\xi,z,x})$.
By the convolution identity for $\varphi$ and the push-forward compatibility conditions of $u$ we obtain the equality
$$
<u_{[x_0,\xi_0]},\varphi E( \cdot \mid \cdot )> = \pi_{[x_0,\xi_0] !} \left( \mathbb L^{-\alpha^{+}(\varphi) m} \varphi <u_{[x_0,\xi_0]\times V_z},\psi>\right).
$$
With notations of Definition \ref{mls-data} and the pull-back relations on $u$ we have
$$i^{*}_{[x_0,\xi_0]\times V_z\times \{\alpha^{+}(\varphi)\}}
\left(<u_{\Lambda_\n \times \mathcal D},T>1_{\E} \right) =
<u_{[x_0,\xi_0]\times V_z},\psi>1_{B_{\xi_0}} 1_{B_{x_0}}.
$$
Thus, we obtain  equality (\ref{Heifetz-formula}) using the equality
$$<u_{\Lambda_\n \times \mathcal D},T>1_{\E}1_{B_N} = <u_{\Lambda_\n \times \mathcal D},T>1_{\E}$$
and  from the definition of $N_\varphi$ the equality
$$ i^{*}_{[x_0,\xi_0]\times V_z\times \{\alpha^{+}(\varphi)\}} 1_{B_N} = \left(i^{*}_{[x_0,\xi_0]\times V_z\times \{\alpha^{+}(\varphi)\}} 1_{B_N}\right)\cdot 1_{B_{N_\varphi}}.$$
\end{rem}

\begin{example}[The Schwartz-Bruhat function case]  \label{cas-SB}
 Let $V$ the definable set $h[m,0,0]$ with $m>0$. Let $P$ be a definable set.
 We consider $u$ a Schwartz-Bruhat function in $S_{P}(P\times V)$. By example
 \ref{ex-loc-int}, this constructible exponential function defines a definable distribution with an empty wave front set.
 Indeed, we prove that the quadruple $(\mathcal A,r,\check r, N)$ is a $\Lambda_\n$-definable data where $\mathcal A$ is the product
 $P \times V\times \Bn$, and $r$, $\check{r}$ and $N$ are definable functions defined by
 $$ r \colon
 \begin{cases}
  \mathcal A \longrightarrow \mathbb Z \\
  (p,x,\xi)  \longmapsto &\alpha^{-}(u)(p)
 \end{cases}
 \:\:,\:\:
  \check r \colon
\begin{cases}
 \mathcal A \longrightarrow \mathbb Z \\
 (p,x,\xi)  \longmapsto  \ord \xi +1
\end{cases}$$
and
 $$N \colon
 \begin{cases}
 \mathcal B \longrightarrow \mathbb Z \\
 (p,x_0,\xi_0,x',r')  \longmapsto  1-\max\left(\alpha^{+}(u)(p),r'\right)- \ord \xi_0.
 \end{cases}
 $$
 Indeed, by definition of $T$ we obtain (using variables to be explicit)
  $$  <u_{\Lambda_\n \times \D},T>1_{\E} =
  \mathcal F\left(u(p,-)1_{B(x',r')}\right)(\lambda \xi) 1_{\E}
  $$
 By the definition of $N$ we obtain the relation
 $$ <u_{\Lambda_\n \times \D},T>1_{\E} = <u_{\Lambda_\n \times \D},T>1_{\E} 1_{B_N}$$
  using the fact that $\mathcal F\left(u(p,-)1_{B(x',r')}\right)$ is a Schwartz-Bruhat function with support in the ball
  $B_{1-\max\left(\alpha^{+}(u)(p),r'\right)}$ because
  $\max\left(\alpha^{+}(u)(p),r'\right) \geq \alpha^{+}\left(u(p,-)1_{B(x',r')}\right)$.
  \end{example}

  \begin{example}[Wave front set of a Dirac measure]
	  By Example \ref{Dirac}, the motivic wave front set of a Dirac measure at a point $x_0$ of $h[m,0,0]$ is $\{x_0\}\times (V\setminus \{0\})$.
\end{example}

\begin{example}[Wave front set of a distribution defined by a smooth variety]  \label{wf-variete}
 Assume the base field $k=\mathbb Q$, the parameter space $P$ is a point, and $\n\geq 1$. Let $d\geq 2$ be a integer. Let $V_x$ be the definable set $h[d-1,0,0]$ and $V_y$ be the definable set $h[1,0,0]$.
 Let $g$ be a polynomial map in $V_x$ and $X$ be the graph of $g$ in $V_x \times V_y$.
 We consider the definable distribution $u$ on $S'(V)$ defined for any $\varphi$ in $S(V)$ by
 $$<u,\varphi> = \int_X \varphi_{\mid X}d\mu_X = \int_{V_x} \varphi(x,g(x))dx.$$
 As $g$ has $\mathbb Q$-coefficients, we can consider the $p$-adic versions of that distribution. By the same proof as in the real case \cite[Example 8.2.5]{Hormander83},
 the $\Lambda_\n$-$p$-adic wave front set is the conormal space to the variety $X_{p}$. We show that the result is the same in this motivic setting.
 The function $(\Id_{V_x}\times g)^*\varphi$ is a Schwartz-Bruhat function in $S(V_x)$ with the data
 $$\alpha^{-}((\Id_{V_x}\times g)^*\varphi) = \alpha^{-}(\varphi)\:\:\mathrm{and}\:\:
 \alpha^{+}((\Id_{V_x}\times g)^*\varphi) = \alpha^{+}(\varphi)-\min \ord_{x\in B_{\alpha^{-}(\varphi)}} dg(x),$$
 where the existence of $\min \ord_{x\in B_{\alpha^{-}(\varphi)}} dg(x)$ follows from the definable compactness lemma \ref{definablecompactness}.
 The $\Lambda_\n$-motivic wave front set of $u$ is equal to the conormal bundle of $X$. Indeed, the definable distribution vanishes on any Schwartz-Bruhat function with support in 
 $(V_x\times V_y)\setminus X$, then its singular support is contained in $X$. Let $(x_0,g(x_0))$ be a point on $X$.
 Let $r$ be an integer. We consider a non zero covector $(\xi,\eta)$ and the integral
 $$<u,1_{B_{(x_0,g(x_0)),r}}E(\cdot \mid \lambda(\xi,\eta))>= \int_{B_r}E(\lambda(((x+x_0)\mid\xi) + g(x+x_0)\eta))dx.$$
Using Taylor expansion we have
$$E(\lambda(x\mid\xi + g(x)\eta)) =
E(\lambda(x_0\mid\xi + g(x_0)\eta))E(\lambda((\xi+\: \eta ^{t}dg(x_0))\mid x))
E(\lambda \eta R_g(x_0,x)x\mid x).$$
By definable compactness the function $x\mapsto \ord R_g(x_0,x)$ admits a minimum $N_{R}$ on the ball $B_r$.
\begin{itemize}
	\item  If $(\xi,\eta)$ is not colinear to $(-\:^{t}dg(x_0), 1)$ then by Proposition \ref{oscillante}, 
$\lambda \mapsto <u,1_{B_{(x_0,g(x_0)),r}}E(\cdot \mid \lambda(\xi,\eta))>$ has a bounded support.
This implies that the point $((x_0,g(x_0)),(\xi,\eta))$ is microlocally smooth.
\item  If $(\xi,\eta)$ is colinear to $(-\:^{t}dg(x_0), 1)$ then by specialization on $p$-adic integrals, the integral $ \int_{x\in B_r}E(\lambda \eta R_g(x_0,x)x\mid x)dx$
does not have a bounded support in $\lambda$. Hence, the point
$((x_0,g(x_0)),(\xi,\eta))$ is not microlocally smooth.
\end{itemize}

For a definable set $X$ which is locally a graph of a definable function, we can define a definable distribution by
$$<u,\varphi> = \int_X \varphi_{\mid X}d\mu_X,$$
Its singular support will be contained in $X$, and its $\Lambda_\n$-motivic wave front set will be contained in the conormal bundle of $X$.
 \end{example}

\subsection{Projection}
We recall first the statement in the $p$-adic setting of Heifetz \cite{Hei85a}.
\begin{prop}
 Let $U$ be an open set of $\mathbb Q_{p}^m$ with $\pi$ the canonical projection of $T^{*}(U)$ on $U$. Let $u$ be a distribution in $S'(U)$.
 Let $\Lambda$ be a subgroup of $\mathbb Q_p^*$ with finite index.
 The projection $\pi(\WF_{\Lambda}(u))$ is equal to the singular support of $u$.
\end{prop}
\begin{rem}
The inclusion of the projection $\pi(\WF_{\Lambda}u)$ in the singular support of $u$ is easy. It is still true in the motivic setting by proposition
\ref{p(WF)-included-in-suppsing}.
The main point to prove the inverse inclusion is the compactness of the $p$-adic sphere. In the motivic setting we have such inequality up to a $\Lambda_\n$-definable data, see Proposition \ref{projectionWF}. Indeed, the finiteness will come from the application of the compactness Lemma \ref{definablecompactness} on definable and continuous functions defined on a definable data.
\end{rem}
\begin{prop}  \label{p(WF)-included-in-suppsing}
	Let $P$ be a definable set, $\n \geq 1$, $V=h[m,0,0]$ with $m\geq 1$ and $\pi_{V}$ be the projection from $T^{*}(V)\setminus \{0\}$ to $V$.
Let $u$ be a definable distribution in $S'_{P}(V)$.
The projection $\pi_{V}(\WF_{\Lambda_\n}(u))$ is contained in the singular support of $u$.
\end{prop}
\begin{proof}
Let $a$ be a smooth point of $u$. By Definition \ref{singular-support} there is a definable function $r_{a}$ from $P$ to $\mathbb Z$ and a Schwartz-Bruhat function $\psi$ in $S_P(P\times V)$ such that for any definable set $\Phi_{W}$ in $\Def_{P}$,
 for any Schwartz-Bruhat function $\varphi$ in $S_{W}(W\times V)$ we have the equality in $\mathcal C(W)^{\exp}$
 $$<u_{\Phi_{W}},\left((\Phi_W\times \Id_V)^{*}1_{B(a,r_a)}\right)\cdot \varphi>=<(\Phi_W\times \Id_V)^{*}\psi,\varphi>.$$
 We consider the definable set $$\mA=\{(p,x,\xi)\in P\times T^{*}(V)\setminus \{0\} \mid x \in B(a,r_{a}(p)) \}$$ and the definable functions from $\mathcal A$ to $\mathbb Z$ defined by
$$r \colon \begin{cases}
    \mathcal A \longrightarrow \mathbb Z \\
     (p,x,\xi) \longmapsto  r_a(p)
   \end{cases}
\:\:,\:\:
   \check r \colon
   \begin{cases}
      \mathcal A \longrightarrow \mathbb Z \\
     (p,x,\xi)  \longmapsto 1+\ord \xi.
   \end{cases}
$$
The constructible exponential functions $<u_{\pi_{\Lambda_\n \times \mathcal D}},T>$ and $\mathcal F(\phi)(\lambda\xi)$ are equal
 with $\phi$ the Schwartz-Bruhat function in $S_{\D}(\Lambda_\n \times \D \times V)$ defined by
 $$\phi(\lambda,\left(p,x_0,\xi_0,x',r',\xi \right),x)=\psi(p,x)1_{B(x',r')}(x).$$
 The Fourier transform of $\phi$ is still a Schwartz-Bruhat function. Thus we obtain the relation
 $$ <u_{\pi_{\Lambda_\n \times \mathcal D}},T>1_{\E} = <u_{\pi_{\Lambda_\n \times \mathcal D}},T>1_{\E} 1_{B_N} $$
 with
 $$
 N \colon\begin{cases}
 \mathcal B \longrightarrow \mathbb Z \\
 \left(p,x_0,\xi_0,x',r'\right)  \longmapsto
  -\ord \xi_0 - \alpha^{-}(\mathcal F(\phi))(p,x_0,\xi_0,x',r').
\end{cases}
  $$
  Then, the quadruple $(\mA,r,\check{r},N)$ is a $\Lambda_\n$-microlocally smooth data of $u$. As there is no condition on $\xi$ in the definition of $\mA$, this prove that $a$ does not belong to the projection $\pi_{V}(\WF_{\Lambda_\n}(u))$.
 \end{proof}

\begin{thm}  \label{projectionWF}
	Let $P$ be a definable set, $\n \geq 1$, $V=h[m,0,0]$ with $m\geq 1$. We denote by $\pi:P\times T^{*}(V)\to V$ and $\pi_{V}:T^{*}(V)\to V$ the canonical projections.
Let $u$ be a definable distribution in $S'_{P}(V)$.
For any $\Lambda_{\n}$-definable microlocally smooth data $(\mA,r,\check{r},N)$ of $u$, the singular support of $u$ is contained in the projection $\pi(\mA^{c}),$
where $\mA^{c}$ is the complement of $\mA$ in $P\times (T^{*}(V)\setminus\{0\})$.
As a consequence, we have
the following inclusions
$$
\pi_{V}(\WF_{\Lambda_{\n}}(u))\subset \SS (u) \subset \bigcap_{\mA \in \mathscr A}
\pi(\mA^{c}),
$$
where $\mathscr A$ is the set of support of $\Lambda$-definable microlocally smooth
data.
\end{thm}
\begin{proof}
Let  $(\mA,r,\check{r},N)$ be a $\Lambda_{\n}$-definable microlocally smooth data.
Let $a$ be a point which is not in the projection $\pi(\mA^{c})$.
Then, for every $p$ in $P$, for every $\xi_0$ in $\Bn$,
the point $(p,a,\xi_0)$ belongs to $\mA$.
 For any parameter $p$ in $P$, by Proposition \ref{definablecompactness}, the restrictions of the definable and continuous maps $r(p,a,-)$ and $\check{r}(p,a,-)$ defined on the closed bounded and definable set $\mathscr B_{\mathfrak n}$
admit a maximum $r_{a}(p)$ and $\check{r}_{a}(p)$.
Applying Corollary \ref{compacite-param} the following functions are definable
$$
 r_{a} \colon
\begin{cases}
 P \longrightarrow \mathbb Z \\
 p  \longmapsto  r_a(p)=\max r(p,a,-)
\end{cases}
\:\:\text{and}\:\:
 \check{r}_{a} \colon
\begin{cases}
 P  \longrightarrow \mathbb Z \\
p  \longmapsto \check{r}_a(p)=\max \check{r}(p,a,-)
\end{cases}.
$$
Using these morphisms we define a convenient $\Lambda$-definable data $(B_a,R_a,\check{R}_a, N_{a})$.
\begin{itemize}
	\item  We consider
$$B_a = \{(p,x_0,\xi_0) \in P\times V_{x_0} \times \Bn \mid x_0 \in B(a,r_a(p))\}.$$
 By construction of $r_a$, $B_a$ is contained in $\mA$.
\item  For all $(p,x_0,\xi_0)$ in $B_a$ we denote
$$
 R_{a} \colon\begin{cases}
B_a \longrightarrow \mathbb Z \\
 (p,x_0,\xi_0) \longmapsto  r_{a}(p)
\end{cases}
\:\:\text{and}\:\:
 \check{R}_{a} \colon
 \begin{cases}
 B_a \longrightarrow \mathbb Z \\
 (p,x_0,\xi_0) \longmapsto \check{r}_a(p).
 \end{cases}
$$
 We set
$$\mathcal B_a:=\{((p,x_0,\xi_0),x',r') \in B_a\times V_{x'} \times \mathbb Z \mid B(x',r')\subset  B(x_0,r_a(p))\}$$
and we consider
$$\mathcal E_{a} := \{((p,x_0,\xi_0,x',r'),\xi) \in \mathcal B_a \times V_\xi \mid \xi \in B(\xi_0, \check{r}_a(p)) \}.$$
\item  Using Corollary \ref{compacite-param} we consider the definable and continuous functions
$$ N_{a} \colon
\begin{cases}
 \mathcal B_a \longrightarrow \mathbb Z \\
 (p,x_0,\xi_0,x',r')  \longmapsto  \min  N(p,x_0,-,x',r')
\end{cases}
\:\text{and}\:\:
\mathcal N_{a} \colon
\begin{cases}
 P \longrightarrow \mathbb Z \\
 p  \longmapsto \min  N(p,a,-,a,2r_{a}(p))
\end{cases}.
$$
\end{itemize}
By Proposition \ref{restriction-data} the quadruple $(B_a,R_a,\check{R}_a, N_{a})$ is a definable data of $u$ and we prove below the equality in $\mathcal C(P\times V_{\xi'})^{\exp}$
\begin{equation}  \label{egalite}
 <u_{P\times V_{\xi'}},1_{B(a,2r_{a})}E(\cdot\mid \cdot))> 1_{B_{\N_{a}}} =
<u_{P\times V_{\xi'}},1_{B(a,2r_{a})}E(\cdot\mid \cdot))>.
\end{equation}
Then, using Proposition \ref{rem-distrib} this constructible exponential function is also a Schwartz-Bruhat function in $S_{P}(P\times V_{\xi'})$, its inverse Fourier transform represents
the definable distribution $\mathbb L^{-m} u1_{B(a,2r_{a})}$ and is a Schwartz-Bruhat function. Thus, $a$ does not belong to
the singular support of $u$.\\
We prove now equality (\ref{egalite}).
By definition
we have the equality
\begin{equation}  \label{relation}
<u_{\Lambda_\n \times \D},T>1_{\E_a}1_{B_{N_a}} =
<u_{\Lambda_\n \times \D},T>1_{\E_a} \in \mathcal C(\Lambda_\n \times \D)^{\exp}.
\end{equation}
We consider the morphisms
$$
\begin{array}{ccc}
i_{a} \colon
\begin{cases}
 \Lambda_\n \times P \times V_{\xi_0,\xi} \longrightarrow \Lambda_\n \times \D \\
 (\lambda,p,\xi_0,\xi)  \longmapsto  (\lambda,(p,a,\xi_0,a,2r_{a}(p),\xi))
\end{cases}
& \text{and} & 
\mathscr{M} \colon
\begin{cases}
\Lambda_\n \times P \times V_{\xi_0,\xi} \longrightarrow P \times V_{\xi'} \\
 (\lambda,p,\xi_0,\xi)  \longmapsto  (p, \xi' = \lambda \xi)
\end{cases}
\end{array}.
$$
By relation (\ref{relation}) and the pull-back compatibility relation of $u$
we obtain the equality
$$<u_{\Lambda_\n \times P \times V_{\xi_0,\xi}},(i_{a}\times Id_{V})^{*}T>i^{*}_{a}
(1_{\E_a}1_{B_{N_a}}) =
<u_{\Lambda_\n \times P \times V_{\xi_0,\xi}},(i_{a}\times Id_{V})^{*}T>i^{*}_{a}(1_{\E_a}),$$
with
$$(i_{a}\times Id_V)^{*}T \colon  (\lambda,p,\xi_0,\xi,y)   \longmapsto  1_{B(a,2r_{a}(p))} E(y\mid \lambda \xi),\:
i_{a}^{*}1_{\E_a} \colon (\lambda,p,\xi_0,\xi)  \longmapsto  1_{B(\xi_0,\check{r}_a(p)}(\xi)$$
and
$$i_{a}^{*} 1_{B_{N_a}} \colon (\lambda,p,\xi_0,\xi)  \longmapsto 1_{B_{\N_a(p)}}(\lambda).$$
By Proposition \ref{SB-exp} the constructible functions
$(i_{a}\times Id_{V})^*(T)i_{a}^{*}(1_{\E_a}1_{B_{N_a}})$ and $(i_{a}\times Id_{V})^{*}(T)i_{a}^{*}(1_{\E_a})$ are $(\mathscr M \times \Id_V)$-compatible Schwartz-Bruhat functions in
$S_{P\times V_{\xi'}}(P\times V_{\xi',x})$. Then,
applying the push-forward morphism $\mathscr M_!$ and the push-forward compatibility relation of $u$ we obtain the equality which induces equality \ref{egalite}
$$<u_{P \times V_{\xi'}},(\mathscr M \times \Id_V)_{!} \left((i_{a}\times Id_{V})^*(T) \cdot i_{a}^{*}(1_{\E_a}1_{B_{N_a}}))\right)>  =
  <u_{P \times V_{\xi'}},(\mathscr M \times \Id_V)_{!} \left((i_{a}\times Id_{V})^{*}(T) \cdot i_{a}^{*}(1_{\E_a})\right)>.
$$
\end{proof}

\subsection{Pull-back}
In the following theorem we explain how to construct the pull-back $f^*u$ of a definable distribution
relatively to a convenient $\Lambda$-microlocally smooth data. For any $m_x$, $m_y$, $\n$ positive integers, when there is no confusion, we will simply denote by $\Bn$ the definable sets $\Bn^{(m_x)}$ of $h[m_x,0,0]$ and
$\Bn^{(m_y)}$ of $h[m_y,0,0]$.

\begin{thm}  \label{thmf^{*}u}
Let $m_x$, $m_y$ and $\n$ be positive integers. Let $\pi$ be the canonical projection of $T^*(V_y)$ to $V_y$. Let
$$f:V_x=h[m_x,0,0]\ra V_y=h[m_y,0,0]$$
be a $\mathcal C^1$ definable function. Let $P$ be a definable set and $u$ be a definable distribution in $S'_{P}(V_y)$. We assume 
\begin{enumerate}
\item  \label{defdataprod}
the distribution $u$ admits a $\Lambda$-microlocally smooth data $(\mA=P\times A,r,\check{r},N)$, with $A$ an open and closed definable set of $V_y\times \Bn$. We assume the projection $\pi(A)$ to be an open and closed definable set of $V_y$ and for any $y_0$ in $\pi(A)$, we will suppose
$A_{y_0} = \{\xi \in \Bn \mid (y_0,\xi)\in A\}$ to be closed and open. 
\item  \label{delta} there is $\delta>0$, $N_\delta$ in $\mathbb Z$ such that 
for any $(x,\xi)$ in $f^{-1}(\pi(A)) \times \Bn$, if $(f(x),\xi)\notin A$ then $\abs{^{t}df(x)\xi}\geq \delta >0$ and $\ord \:^{t}df(x)\xi \leq N_{\delta}$.
In particular this implies the inclusion
$$N_{f}:=\{(y,\xi)\in T^{*}(\pi(A))\setminus \{0\} \mid \exists x \in f^{-1}(\pi(A)),\: y=f(x), \:^{t}df(x)\xi=0\}\subset A.$$
\item  \label{Rf}
there is a definable function $R_{f}$ such that
\begin{equation} \label{eqTaylor} (f(x+z)\mid \xi)=(f(x)\mid \xi)+(df(x)z\mid \xi)+(R_{f}(x,z,\xi)z\mid z). \end{equation}
There is a constant $N_{R}$, such that for any $x$ in $f^{-1}(\pi(A))$, for any $z$ in $V_x$ and any
$\xi$ in $\Bn$ we have $\ord R_{f}(x,z,\xi) \geq N_{R}$.
\end{enumerate}
With these assumptions
\begin{enumerate}
\item there is a definable distribution $f^{*}u$ in
$S'_P(f^{-1}(\pi(A)))$ such that
we have the inclusion
$$\WF_{\Lambda_\n}(f^{*}u) \:\: \subset \:\: f^{*}(\Lambda_\n A^{c})\cup U_{f,A} $$
where
$$\Lambda_\n A^{c} = \{(y,\lambda \xi) \in T^{*}(V_y) \setminus \{0\} \mid (y,\xi) \notin A\},\:
U_{f,A}=\{(x,\eta)\in f^{-1}(\pi(A))\times V_\eta \mid \eta \notin \mathrm{im}\: ^{t}df(x)\}$$
and
$$f^{*}(\Lambda_\n A^{c})=\{(x,\eta)\in T^{*}(f^{-1}(\pi(A)))\setminus \{0\} ,\: \exists (y,\xi)\in \Lambda_\n A^{c}, y=f(x),\: ^{t}df(x)\xi=\eta\}.$$

\item for any $(B,r_B, \check{r}_{B}, N_B)$ a $\Lambda$-definable microlocally smooth data of $u$ with same assumptions such that $\pi(B)\cap \pi(A)$ is non empty the two pull-backs are equal on the Schwartz-Bruhat functions of
$S_{W}(f^{-1}(\pi(A)\cap \pi(B)))$. \\

\item if $u$ is a Schwartz-Bruhat function in $S_{P}(P\times V)$ the construction gives
$$f^{*}u= \mathbb L^{-m_x} u\circ (\Id_P\times f).$$
\end{enumerate}

\end{thm}

\begin{rem}
By the second point of the theorem we can patch all the $f^*u$ along some $\Lambda$ definable data (all of them if $P$ is for instance a point). In the real and $p$-adic settings it is proved that extension of the pull-back $f^*$ from the definable distribution defined by smooth functions to the definable distribution is unique (see \cite[Theorem 8.2.4]{Hormander83} and \cite[Theorem 2.9.3]{CHLR}). In our setting, there is no topology on motives which can give such unicity. Nevertheless, the construction here is parallel to the construction of the $p$-adic and real case. By specialization theorems in \cite{CluLoe10a} of motivic integrals to $p$-adic integrals, this construction specializes on the construction of the pull-back in the $p$-adic setting. In particular, in  Remark \ref{localization} below, we recover  Theorem 2.8 of Heifetz  \cite{Hei85a}.
\end{rem}

\begin{proof}
 We start defining optimal radius functions, useful for the proof.\\

$\bullet$ We define a continuous definable map $R_{y}:P \times \pi(A)\ra \mathbb Z$ such that
\begin{equation} \label{*-property}
 \forall p\in P,\: \forall y_0 \in \pi(A),\:\forall (y,\xi)\in B(y_0,R_y(p,y_{0}))\times \Bn,
(y,\xi)\in A \Leftrightarrow (y_0,\xi)\in A.
\end{equation}
Indeed, for any $p$ in $P$, for any $y_0$ in $\pi(A)$ by definable compactness (Corollary \ref{compacite-param}) the definable and continuous function $r(p,y_0,-)$ admits a maximum on the closed and bounded definable set $A_{y_0}$ and we consider the definable function
$$r_{y_0,\max} \colon
\begin{cases}
 P \longrightarrow \mathbb Z \\
 p  \longmapsto  \max r(p,y_0,-)_{\mid A_{y_0}}.
  \end{cases}
$$
Similarly, we consider the definable and continuous function
$$R_{y} \colon
\begin{cases}
 P \times \pi(A) \longrightarrow \mathbb Z \\
 (p,y_0)  \longmapsto  \max r(p,-,-)_{\mid A\cap (B(y_0,r_{y_{0},\max}(p))\times \Bn)}.
  \end{cases}
$$
Let $p$ be in $P$, $y_0$ be in $\pi(A)$ and $(y,\xi)$ be in $B(y_0,R_y(p,y_{0}))\times \Bn$.
\begin{itemize}
	\item  If the point $(y_0,\xi)$ belongs to $A$ then by Definition \ref{mls-data} we have the inclusion
		$$B(y_0,r(p,y_0,\xi))\times (B(\xi, \check{r}(p,y_0,\xi)) \cap \Bn) \subset A,$$
		and $(y,\xi)$ belongs to $A$ because by definition the ball $B(y_0,R_y(p,y_{0}))$ is contained in the ball $B(y_0,r(p,y_0, \xi))$.
	\item  If the point $(y,\xi)$ belongs to $A$, then again by definition of $R_y(p,y_{0})$ we have the inclusion
		$$B(y,R_y(p,y_{0}))\times (B(\xi, \check{r}(p,y,\xi)) \times \Bn) \subset
		B(y,r(p,y,\xi))\times (B(\xi, \check{r}(p,y,\xi)) \times \Bn) \subset A$$
		and we conclude that $(y_0,\xi)$ belongs to $A$ by the equality of balls
		$ B(y,R_y(p,y_{0})) $ and $B(y_0,R_y(p,y_{0})).$
\end{itemize}

 $\bullet$ By Corollary \ref{compacite-param}, the construction and the continuity of $\check r$, we define a continuous definable function
$$ \check R_{y} \colon
\begin{cases}
P \times \pi(A) \longrightarrow \mathbb Z \\
(p,y_0)  \longmapsto \max \check{r}(p,y_0,-)_{\mid A_{y_0}}.
  \end{cases}
$$
By assumption \ref{condr} on $\check{r}$, we have $\check R_{y}(p,y_0)\geq \n \geq \ord \xi_0+1$ for any $\xi_0$ in $\Bn$.

$\bullet$ By Corollary \ref{compacite-param}, the construction and the continuity of $N$, we define a continuous definable map
$$ \overline{N} \colon
\begin{cases}
\mathcal B \longrightarrow \mathbb Z \\
 (p,y_0,\xi_0,y',r')  \longmapsto  \min N(p,y_{0},-,y',r')_{\mid A_{y_0}}
  \end{cases}
$$
 with
  $$\mathcal B=\{\left((p,y_0,\xi_0),y',r'\right)\in \mA \times V_{y'} \times \mathbb Z \mid
  B(y',r')\subset B(y_0,r(p,y_0,\xi_0))\}.$$
  By definition of $\overline{N}$ and by Proposition \ref{restriction-data}, the quadruple
$(A,r,\check r, \overline{N})$ is a definable data.\\

$\bullet$ We define a continuous definable map
$$  \mathcal{N} \colon
\begin{cases}
 P\times \pi(A) \longrightarrow \mathbb Z \\
   (p,y_0)  \longmapsto \min N(p,y_{0},-,y_0,R_y(p,y_{0})+1)_{\mid A_{y_0}}.
  \end{cases}
$$
Again the existence follows from Corollary \ref{compacite-param}.
In particular, for any $(p,y_0)$ in $P\times \pi(A)$, for any $\xi_0$ in $A_{y_0}$, we have
$$\mathcal N(p,y_0)=\overline{N}(p,y_0,\xi_0,y_0,R_y(p,y_{0})+1)$$
The continuity comes from the construction and the continuity of $N$ and $R_y$.\\

$\bullet$ We define a definable and continuous map $R_x:P\times f^{-1}(\pi(A))\ra \mathbb Z$.\label{R_x(x_0)}
As $f$ is continuous, for any $p$ in $P$, for any $y_0$ in $\pi(A)$, the inverse image
$f^{-1}(B(y_{0},R_y(p,y_{0})))$ is open,
and for any $x_0$ in the fiber of $y_0$, we denote  by $R_x(p,x_0)$ the smallest integer such that the ball $B(x_{0},R_x(p,x_0))$ is contained in the inverse image  $f^{-1}(B(y_{0},R_y(p,y_{0})))$
$$
 R_x(p,x)  = \min \left\{a\in \mathbb Z \mid f(B(x,a))\subset B\left(f(x),R_y(p,f(x))\right)\right\}.$$
namely
 $$R_x(p,x) = \min \{a \in \mathbb Z \mid \forall x'\in V_x,\: \ord (x'-x)\geq a \Rightarrow \ord (f(x')-f(x))\geq R_y(p,f(x))\}.$$
As the function $(p,x)\mapsto R_y(p,f(x))$ is definable, $R_x$ is also definable and continuous (locally constant by construction).\\

$\bullet$ For any $\Phi_W$ in $\Def_P$ we denote by
$$R_{y,\Phi_W}=R_{y}\circ \left(\Phi_W \times \Id_{\pi(A)}\right),
\:\check{R}_{y,\Phi_W}=\check{R}_y\circ \left(\Phi_W \times \Id_{\pi(A)}\right),$$
and
$$
\mathcal N_{\Phi_W}=\mathcal N\circ \left(\Phi_W \times \Id_{\pi(A)}\right),
\:R_{x,\Phi_W}=R_{x}\circ \left(\Phi_W \times \Id_{f^{-1}(\pi(A))}\right).$$

 \textbf{1. Definition of $(f^*u)_{\Phi_W}$}.
  In order to define the definable distribution $f^*u$ we will use the extension Theorem \ref{extension-theorem}.
 We fix a definable morphism $\Phi_W$ in $\Def_P$ and a definable function $\alpha^{-} : W \ra \mathbb Z$.
 By assumption $\pi(A)$ is closed, then $f^{-1}(\pi(A))$ is closed and by Proposition \ref{definablecompactness} we may consider the definable map
 $$ R_{x,\alpha^{-}} \colon
  \begin{cases}
   W \longrightarrow \mathbb Z \\
   w  \longmapsto  \max \{R_{x,\Phi_W}(w,x)\mid x \in f^{-1}(\pi(A)),\: \ord x \geq \alpha^{-}(w)\}.
    \end{cases}
 $$
 We consider a definable function $\alpha^{+}:W\ra \mathbb Z$ satisfying the inequalities
\begin{equation}  \label{inegalites-alpha}
\alpha^{+} \geq R_{x,\alpha^{-}}\:\:\mathrm{and}\:\:\alpha^+ \geq \alpha^-.
\end{equation}
Inspired by the proof in the $p$-adic case in \cite{Hei85a} or \cite{CHLR}, we explain the strategy of the construction.
 \begin{enumerate}
	 \item We show the existence of the following constructible exponential function in $\mathcal C(W \times V_z)^{\exp}$
 \begin{equation}  \label{def-f*u-alpha}
  <(f^*u)_{\Phi_W \times V_z}, t_{\alpha^-,\alpha^+}> := \pi_{W\times V_z !} (\psi I_{f,\alpha^{-},\alpha^{+}})
 \end{equation}
 where $\psi$ is the constructible exponential function in $\mathcal C(W\times V_{z,\xi})^{\exp}$
 $$
 \psi : (w,\xi,z) \mapsto \left<u_{\Phi_W \times V_{\xi,z}}, y\mapsto \chi_{\Phi_W}(w,z,y)E(y\mid \xi)\right>,
 $$
 with
 $$\chi_{\Phi_{W}} = (\Phi_{W} \times \Id_{V_{z,y}})^{*}\chi,\:\:\:\: \chi : (p,z,y)\mapsto 1_{B(f(z),R_y(p,f(z))}(y)$$
 and $I_{f,\alpha^{-},\alpha^{+}}$ is the constructible exponential function in $\mathcal C(W\times V_{z,\xi})^{\exp}$
 $$
 I_{f,\alpha^{-},\alpha^{+}} : (w,\xi,z)  \mapsto
 \int_{x\in V_x} 1_{B_{\alpha^{-}(w)}}(x) 1_{f^{-1}(\pi(A))}(x)
 1_{B_{\alpha^{+}(w)}}(x-z) E(-f(x)\mid \xi)dx.
 $$
  \item We prove the assumptions of the extension Theorem \ref{extension-theorem} and by its application
 we define for any
 $\varphi$ in $S_{W}(W\times V)$ the constructible exponential function in $\mathcal C(W)^{\exp}$
 \begin{equation} \label{evalf*u}
	 <(f^*u)_{\Phi_W},\varphi>=
 \pi_{W!}\left(\mathbb L^{\alpha^{+}(\varphi) m_x} \mathbb  \varphi
 <(f^*u)_{\Phi_W \times V_z},t_{\alpha^-(\varphi),\alpha^+(\varphi)}> \right)
 \end{equation}
 choosing $\alpha^{+}(\varphi)\geq R_{x,\alpha^{-}(\varphi)}$.
 \item We consider the restriction of $f^*u$ to $f^{-1}(\pi(A))$ in $S'_{P}(f^{-1}(\pi(A)))$.

 \end{enumerate}

  \textbf{Step 1. Existence of $<(f^*u)_{\Phi_W \times V_z},t_{\alpha^-,\alpha^+}>$.}
  Let $\Phi_W$ a definable morphism in $\Def_P$ and $\alpha^+$ and $\alpha^-$ two definable morphims from $W$ to $\mathbb Z$ satisfying conditions \ref{inegalites-alpha}. We consider
  $$(\Lambda_\n A)_{W} = \{(w,z,\xi) \in W\times V_{z,\xi} \mid
\ord z \geq \alpha^{-}(w)\:\:\mathrm{and}\:\:\exists \lambda \in \Lambda_\n, (f(z),\lambda^{-1}\xi)\in A\}$$
and
$$(\Lambda_\n A^{c})_{W} = \{(w,z,\xi)\in W\times V_{z,\xi} \mid \ord z \geq \alpha^{-}(w)\:\:\mathrm{and}\:\:\forall \lambda \in \Lambda_\n, (f(z),\lambda^{-1}\xi)\notin A\}.$$
These sets are definable and we have the decomposition
$$B_{\alpha^{-}(w)}\times V_{\xi} = \Lambda_\n A_w \sqcup \Lambda_\n A^{c}_w.$$

$\bullet$ By $\Lambda_{\n}$-microlocally smoothness, we define
$C_A : \{(w,z)\in W \times V_z \mid z\in B_{\alpha^{-}(w)}\} \ra \mathbb Z$ a definable function such that
\begin{equation}  \label{relation-psi}
\psi \cdot 1_{(\Lambda_{\n} A)_W}1_{B_{C_A}} = \psi\cdot1_{(\Lambda_{\n} A)_W},
\end{equation}

with
$B_{C_{A}} = \{(w,z,\xi)\in W\times V_{z,\xi} \mid \ord \xi \geq C_{A}(w,z)\}.$

$\bullet$ By assumption on $f$ and $A$, we define $C_{A^c} : \{(w,z)\in W \times V_z \mid z\in B_{\alpha^{-}(w)}\} \ra \mathbb Z$ a definable function
such that
\begin{equation}  \label{relation-I}
I_{f,\alpha^{-},\alpha^{+}}\cdot1_{(\Lambda_\n A^{c})_W}1_{B_{C_{A^c}}}=I_{f,\alpha^{-},\alpha^{+}}\cdot1_{(\Lambda_\n A^{c})_W},
\end{equation}

with
$B_{C_{A^c}} = \{(w,z,\xi)\in W\times V_{z,\xi} \mid \ord \xi \geq C_{A^c}(w,z)\}.$

$\bullet$ Defining the constructible exponential function $\mathcal I_{A}$ by
$$\mathcal I_A(w,z,\xi) =
1_{(\Lambda_\n A)_W}(w,z,\xi) 1_{B_{C_{A}(w)}}(\xi) +
1_{(\Lambda_\n A^c)_W}(w,z,\xi) 1_{B_{C_{A^c}(w)}}(\xi),$$
 we obtain the equality
$$\psi I_{f,\alpha^-, \alpha^+} = \psi I_{f,\alpha^-,\alpha^+}\mathcal I_A.$$
The function $\psi I_{f,\alpha^{-},\alpha^{+}}$ will be $\xi$-integrable which will imply by \ref{def-f*u-alpha} the existence of the constructible exponential 
function $<(f^*u)_{\Phi_W \times V_z},t_{\alpha^-,\alpha^+}>$ in $\mathcal C(W\times V_z)^{\:\exp}$.\\

\textbf{Existence of $C_A$}.
As $(A,r,\check r, \overline{N})$ is a $\Lambda$-microlocally smooth data, using notations $\E$, $\mathcal D$ and $T$ of Definition \ref{mls-data}  we have the equality 
$$<u_{\Lambda_\n \times \mathcal D }, T>1_{\E}1_{B_{\overline{N}}} = <u_{\Lambda_\n \times \mathcal D }, T>1_{\E} \in C(\Lambda_\n \times \mathcal D)^{\exp}.$$
We consider now the morphism
$$i_{f}  \colon
\begin{cases}
 P\times \Lambda_\n \times V_{z,\xi_0,\xi} \longrightarrow \Lambda_\n \times \mathcal D \\
 (p,\lambda,z,\xi_{0},\xi)  \longmapsto  \left(\lambda, \left(p,f(z),\xi,f(z),R_{y}(p,f(z)),\xi\right)\right).
\end{cases}
$$
By the pull-back compatibility relation of $u$ we obtain the equality in the ring $\mathcal C(P\times \Lambda_\n \times V_{z,\xi_0,\xi})^{\exp}$
$$i_{f}^{*}\left(<u_{\Lambda_\n \times \mathcal D }, T>1_{\E}1_{B_{\overline{N}}}\right) =
<u_{P\times \Lambda_\n \times V_{z,\xi_0,\xi} },(i_{f}\times Id_{V})^{*}T>i_{f}^{*}(1_{\E})i_{f}^{*}(B_{\overline{N}})
$$
with
$$(i_{f}\times Id_V)^* T \colon ((p,\lambda,z,\xi_{0},\xi),y) \longmapsto  \chi(p,z,y)E( y\mid \lambda \xi)\:\:,\:\:
i_{f}^{*}(1_{\E}) \colon(p,\lambda,z,\xi_{0},\xi)  \longmapsto  1_{B(\xi_{0},\check{R}_{y}(p,f(z)))}(\xi)
$$
and
$$
i_{f}^{*}(1_{B_{\overline{N}}}) \colon (p,\lambda,z,\xi_{0},\xi)   \longmapsto  1_{B_{\mathcal{N}(p,f(z))}}(\lambda).
$$
We consider the morphism
$$m \colon
\begin{cases}
 P\times \Lambda_\n \times V_{z} \times V_{\xi_{0}}\times V_{\xi}  \longrightarrow  P\times V_{z} \times V_{\xi'} \\
 (p,\lambda,z,\xi_{0},\xi) \longmapsto  (p,z,\xi':=\lambda\xi).
\end{cases}
$$
The constructible exponential function $(i_{f}\times Id_{V})^{*}(T) \cdot i_{f}^{*}(1_{\E} 1_{B_{N}})$ is a $(m\times \Id_{V_y})$-compatible Schwartz-Bruhat function in
$S_{P\times \Lambda_\n \times V_{z,\xi_0,\xi}}(P\times \Lambda_\n \times V_{z,\xi_0,\xi}\times V_y)$, then by the push-forward compatibility relation of $u$ we obtain the equality
$$m_{!} \left<u_{P\times \Lambda_\n \times V_{z,\xi_0,\xi}},(i_{f}\times Id_{V})^{*}(T) \cdot i_{f}^{*}(1_{\E} 1_{B_{N}}) \right> =
\left<u_{P\times V_{z,\xi'} },(m \times \Id_{V_{y}})_{!}\left(\left(i_{f}\times Id_{V})^{*}(T) \cdot i_{f}^{*}(1_{\E} 1_{B_{N}}\right)\right)\right>
$$
and we deduce
$$<u_{P\times V_{z,\xi'}}, \zeta > 1_{\{\ord \xi' \geq \mathcal N(p,f(z))\}}= <u_{P\times V_{z,\xi'}},\zeta>.$$
using the constructible exponential function  $\mathcal C(P\times V_{z,\xi',y})^{\exp}$
$$\zeta : (p,z,\xi',y) \mapsto E( y\mid \xi') 1_{(\Lambda_\n A)_P}(p,z,\xi') \chi(p,z,y).$$
By pull-back relation of $u$ by $(\Phi_W \times \Id_{V_{\xi',_z}})$, we obtain relation (\ref{relation-psi}) with $C_A$ the definable map 
$(w,z) \mapsto \mathcal N(\Phi_W(w),f(z))$.\\

\textbf{Existence of $C_{A^c}$}.  \label{C_A^c}
We consider the definable morphism
$$m\colon
\begin{cases}
 \Lambda_\n \times V_{\xi} \times W \times V_z \longrightarrow V_{\xi'} \times W \times V_z\\
 (\lambda, \xi, w, z)  \longmapsto  (\xi':=\lambda\xi, w,z)
\end{cases}
$$
We consider the constructible exponential function
$$J:(\lambda, \xi, w, z) \mapsto I_{f,\alpha^{-},\alpha^{+}}(w,\lambda \xi,z).$$
If a couple $(x,\xi)$ belongs to $B(z,R_{x,\Phi_W}(w,z))\times \Bn$ then by definition of $R_{x,\Phi_W}(w,z)$,
 $f(x)$ belongs to the image $f(B(z,R_{x,\Phi_W}(w,z)))$ contained in the ball $B(f(z), R_{y}(\Phi_W(w),f(z)))$, and by definition of $R_y$, the couple $(f(x),\xi)$ belongs to $A$ if and only if the couple $(f(z),\xi)$ belongs to $A$. Thus,
for any $(w,z,\lambda\xi)$ in $(\Lambda_n A^{c})_W$, for any $x$ in $B(z,R_x(\Phi_W(w),z))$
the couple $(f(x),\xi)$ does not belong to $A$ and by assumption we have
$\abs{^{t}df(x)\xi}\geq \delta >0\:\:\mathrm{namely} \:\: \ord ^{t}df(x)\xi \leq N_{\delta}.$
By Proposition \ref{oscillante}, we conclude that
$$J \cdot 1_{\{\ord \lambda \geq C_{A^c}(w)\}} = J$$
with $C_{A^c}$ the definable map $-\max(N_{\delta}-N_{R}+1, \alpha^+)-N_{\delta}$.
By the direct image by $m$ and Fubini we deduce relation (\ref{relation-I}).\\

\textbf{Integrability of $\psi I_{f,\alpha^{-},\alpha^{+}}$}.
We consider the constructible exponential function
$$\theta :(w,z,y,x,\xi) \mapsto
\chi_{\Phi_W}(w,z,y) E(y-f(x)\mid \xi)t_{\alpha^-,\alpha^+}(w,z,x)
1_{f^{-1}(\pi(A))}(x) \mathcal I_A(w,z,\xi)$$
with
$$\mathcal I_A(w,z,\xi) =
1_{(\Lambda_\n A)_W}(w,z,\xi) 1_{B_{C_{A}(w)}}(\xi) +
1_{(\Lambda_\n A^c)_W}(w,z,\xi) 1_{B_{C_{A^c}(w)}}(\xi),$$
and $$t_{\alpha^-, \alpha^+}(w,z,x)=1_{B^{\alpha^+(w)}}(z-x)1_{B^{\alpha^-(w)}}(x).$$

As each variable $x$ and $\xi$ is bounded in the parameters $w$ and $z$ by the relations
\begin{equation}  \label{inegalite-x-xi}
\ord x \geq \alpha^{-}(w)\:\: \mathrm{and } \:\: \ord \xi \geq \min(C_{A}(w),C_{A^c}(w))
\end{equation}
the constructible exponential function $\theta$ is $\pi_{W\times V_z}$-integrable and $\pi_{W\times V_z \times V_\xi}$-integrable.

Using Proposition \ref{SB-exp}, we deduce that $\theta$ is a  $(\pi_{W\times V_{z,\xi}}\times \Id_{V_y})$-compatible and
$(\pi_{W\times V_z}\times \Id_{V_y})$-compatible Schwartz-Bruhat function in $S_{W\times V_{z,x,\xi}}(W\times V_{z,x,\xi,y})$
with
$$\alpha^{+}(\theta)(w,z,\xi,x) = \max(R_{y,\Phi_W}(w,f(z)) , -\min(C_{A}(w),C_{A^c}(w)))$$
and
$$
\alpha^{-}(\theta)(w,z,\xi,x) =  \min(\ord f(z), R_{y,\Phi_W}(w,f(z))).
$$
Then, by the push-forward condition of $u$,
we deduce that $<u_{\Phi_W \times V_{\xi,z,x}},\theta>$ is $\pi_{W \times V_z}$-integrable and also
$\pi_{W\times V_{z,\xi}}$-integrable with the equality
$$\psi I_{f,\alpha^-, \alpha^+} =
\pi_{W \times V_{z,\xi} ! } <u_{\Phi_W\times V_{\xi,z,x}},\theta>
$$
By Fubini we obtain that $\psi I_{f,\alpha^-, \alpha^+}$ is $\pi_{W\times V_z}$-integrable and we can consider
\begin{equation}
\label{eqtheta}
<(f^*u)_{\Phi_W \times V_z}, t_{\alpha^-,\alpha^+}>= \pi_{W\times V_z !} (\psi I_{f,\alpha^-, \alpha^+}) = 
<u_{\Phi_W \times V_z},(\pi_{W\times V_z}\times \Id_{V_y})_{!}\theta> \\
\end{equation}

\textbf{Step 2. Extension theorem.}
Let $\Phi_W$ be a definable morphism in $\Def_P$.\\

$\bullet$ Integrability condition of the extension theorem.\\
Let $\alpha^+$ and $\alpha^-$ be definable functions from $W$ to $\mathbb Z$ satisfying \ref{inegalites-alpha}.
Using conditions \ref{inegalite-x-xi} the constructible exponential function $\theta$ is $(\pi_{W}\times \Id_{V_y})$-integrable and also a $(\pi_W\times \Id_{V_y})$-compatible Schwartz-Bruhat function in $S_{W\times V_{z,x,\xi}}(W\times V_{z,x,\xi,y})$.
Then, by the push-forward property of the definable distribution $u$ and by Fubini, the constructible exponential function 
$<(f^*u)_{\Phi_W \times V_z},t_{\alpha^-,\alpha^+}>$ is $\pi_W$-integrable.\\

$\bullet$ Pull-back assumption of the extension theorem.\\
Let $\Phi_W$ and $\Phi_{W'}$ be two definable sets in $\Def_P$.
Let $g:W\ra W'$ be a definable morphism in $\Def_P$ such that $\Phi_W = g \circ \Phi_{W'}$.
We have by construction
$$(g\times \Id_{V_z})^{*}(<(f^*u)_{\Phi_{W'}\times V_z},t_{\alpha^-,\alpha^+}>) =
(g\times \Id_{V_z})^{*} \left(\pi_{W\times V_z !}<u_{\Phi_{W'} \times V_{z,x,\xi}},\theta>\right).$$
The equality
$$(g\times \Id_{V_z})^{*}(<(f^*u)_{\Phi_{W'}\times V_z},t_{\alpha^-,\alpha^+}>) =
<(f^*u)_{\Phi_W\times V_z},
(g\times \Id_{V_z \times V_x})^{*}(t_{\alpha^-, \alpha^+})>
$$
in $\mathcal C(W\times V_z)^{\exp}$ follows from Proposition \ref{compatibilite-pull-back-push-forward} and the pull-back compatibility of $u$.\\

$\bullet$ Push-forward axiom of the extension theorem.\\ 
Let $\Phi_W : W \ra P$ be a definable morphism. For any definable function $\beta^+$, $\beta^-$, $\alpha^+$, $\alpha^-$ from $W$ to $\mathbb Z$, satisfying the inequalities \ref{inegalites-alpha}, $\beta^+ \geq \alpha^+$ and $\alpha^- \geq \beta^-$ we have the relation
\begin{equation} 
	<(f^*u)_{\Phi_W \times V_z},t_{\alpha^-,\alpha^+}> = 
	(\pi_{W \times V_z })_{!} \left( \mathbb L^{\beta^+ m} t_{\alpha^-, \alpha^+} <(f^*u)_{(\Phi_W \times V_z) \times V_x},s\mapsto t_{\beta^-, \beta^+}(w,x,s)>\right).
\end{equation}
by application of the definition \ref{eqtheta},  the push-forward compatibility condition of $u$ and Fubini, see also example \ref{splitballs}.\\

\textbf{2. Independence in the data.}
 For any $(B,r_B, \check{r}_{B}, N_B)$ a $\Lambda$-definable microlocally smooth data of $u$ with same assumptions such that
 $\pi(B)\cap \pi(A)$ is non empty, the two constructions of
$f^{*}u$ on $S(f^{-1}(\pi(A) \cap \pi(B)))$ coincide.
Indeed, by the motivic average formula or Remark \ref{remark-extension}, we obtain that the construction of the pull-back $f^*u$ for the microlocally smooth data $(A,\tilde{r},\tilde{\check r}, \tilde{N})$ gives the same definable distribution when
$\tilde{r}\geq r$, $\tilde{\check r}\geq \check r$ and $N\geq \tilde{N}$. Then in our case, the result follows considering this fact and the definable functions $\max(r,r_B)$, $\max(\check{r},\check{r}_B)$ and $\min(N,N_B)$.\\

\begin{rem}[ Localization and Heifetz formula]  \label{localization}
Let $\Phi_W$ be a definable morphism in $Def_P$ and $\varphi$ be a Schwartz-Bruhat function in $S_{W}(W\times V)$. Let $x_0$ be a point in $V$. Let $r_{x_0}:W\to \mathbb Z$ be a definable function such that $r_{x_0} \geq R_{x,\Phi_W}(-,x_0)$. We denote by $B_{x_0}$ the ball of center $x_0$ and valuative radius $r_{x_0}$.
Considering $\alpha^{+}(\varphi) \geq R_{x,\varphi 1_{B_{x_0}}}$ by equality \ref{evalf*u} we have

$$<(f^*u)_{\Phi_W},\varphi 1_{B_{x_0}}> = w \mapsto \mathbb L^{-\alpha^{+}(\varphi)(w)m}
 \int_{z \in B_{\alpha^{-}(\varphi)(w)}}(\varphi1_{B_{x_0}})(w,z)$$
 $$
 \left(
 \int_{\xi \in V_\xi} \left<u_{\Phi_W \times V_{\xi,z}}, \chi_{\Phi_W}(w,z,\cdot)E(\cdot\mid \xi)\right>
 \left(
 \int_{x\in V_x}1_{B_{\alpha^{+}(\varphi)(w)}}(x-z)E(-f(x)\mid \xi)dx
 \right)
 d\xi
 \right)
 dz
 $$
 with $\chi_{\Phi_W} : (w,z,y)\mapsto 1_{B(f(z),R_{y,\Phi_W}(w,f(z)))}(y).$

By assumption on $r_{x_0}$ and construction of $R_{x,\Phi_W}(w,x_0)$, we have for any $w$ in $W$ and $z$ in $B_{x_0}$ the equality
$$ B(f(z),R_{y,\Phi_W}(w,f(z))) = B(f(x_0),R_{y,\Phi_W}(w,f(x_0))).$$
Applying Fubini and the convolution formula for $ \varphi 1_{B_{x_0}}$ and the equality
$$\{(x,z)\mid z\in B_{\alpha^{-}(\varphi)},\: x\in B(z,\alpha^{+}(\varphi))\} =
 \{(x,z)\mid x\in B_{\alpha^{-}(\varphi)},\: z\in B(x,\alpha^{+}(\varphi))\},$$
we obtain the motivic version of the Heifetz formula in \cite{Hei85a}
$$
<(f^*u)_{\Phi_W},\varphi 1_{B_{x_0}}> =  
w \mapsto
 \int_{\xi \in V_\xi} \left<u_{\Phi_W \times V_\xi }, \chi_{\Phi_W}(w,x_0,\cdot)E(\cdot\mid \xi)\right>
 \left(\int_{x \in V_x}(\varphi1_{B_{x_0}})(w,x) E(-f(x)\mid \xi) dx \right)d\xi.
 $$
\end{rem}

\textbf{3. Function case.}  \label{function-case}
Suppose that $u$ is a Schwartz-Bruhat function in $S_{P}(P\times V)$. For any definable set $\Phi_W$ in $\Def_P$,
we consider $u_{\Phi_W}$ equal to $(\Phi_W\times \Id_V)^{*}u$.
As in Example \ref{cas-SB}, we consider the definable data $(P\times V_y \times \Bn, r, \check r, N)$.\\
Let $\alpha^-$, $\alpha^+$ be definable functions from $W$ to $\mathbb Z$ satisfying relations \ref{inegalites-alpha}.
By definition of $R_x$ and $R_y$, for any $w$ in $W$, for any $z$ in the ball $B_{\alpha^+(w)}$ we have the inclusion
\begin{equation}  \label{inclusion-image}
f(B(z,\alpha^+)) \subset B(f(z),R_{y}(\Phi_W(w),f(z))).
\end{equation}
By definition we have
$$<(f^*u)_{\Phi_W \times V_z}, t_{\alpha^-,\alpha^+}> = \pi_{W\times V_z !}<u_{\Phi_W \times V_{z,\xi,x}},\theta>.$$
Then, applying Fubini on \ref{def-f*u-alpha}, the projection axiom and the inverse Fourier transform formula and example
\ref{inclusion-image}
we obtain the relation
$$<(f^*u)_{\Phi_W \times V_z}, t_{\alpha^-,\alpha^+}> = \mathbb L^{-m_x}
< (u \circ (\Id_P\times f))_{\Phi_W \times V_z},t_{\alpha^-,\alpha^+}>.$$
Then, by the extension theorem (or the average formula), we obtain that
$$ f^*u = \mathbb L^{m_x} u\circ (\Id_P\times f).$$

\textbf{4. Wave front set of $f^*u$.}\\
Let $(x_0,\eta_0)$ a point of $f^{-1}(\pi(A))\times V_\eta \setminus \{0\}$ which does not belong to $f^{*}(\Lambda_\n A^{c})$ with $\eta_{0}$ in $im(\:^{t}df(x_0))$. As $\eta_0$ is different from zero, one has that $\:^{t}df(x_0)$ is also different from zero.
By definition, for any $\xi_0$ in $V_\xi$ with $^{t}\!df(x_0)(\xi_0)=\eta_0$,
the point $(f(x_0),\xi_0)$ belongs to $\Lambda_\n A$.\\

 For any point $\eta$ in the ball $B(\eta_0,\ord \eta_0+1)$, the fiber $^{t}\!df(x_0)^{-1}(\eta)$ is contained in the definable set
 $\Lambda_\n A_{f(x_0)}$. Indeed, for such $\eta$ there is a $\lambda$ in $\Lambda_n$ such that $\eta$ is equal to $\lambda \eta_0$, then
 the fiber $^{t}\!df(x_0)^{-1}(\eta)$ is equal to $\lambda ^{t}\!df(x_0)^{-1}(\eta_0)$ which is contained in $\Lambda_\n A_{f(x_0)}$. \\

We prove that the point $(x_0,\eta_0)$ is a $\Lambda$-microlocally smooth point of $f^{*}u$.
We denote by $P_{x_0,\eta_0}$ the intersection $\:^{t}\!df(x_0)^{-1}(\eta_0)\cap A_{f(x_0)}$.
By definable compactness  (Proposition \ref{definablecompactness}), we define $\check{r}_{f(x_0)}$ as the definable function
$$ \check{r}_{f(x_0)} \colon
 \begin{cases}
 P \longrightarrow \mathbb Z \\
  p  \longmapsto  \min_{\xi,A_{f(x_0)}} \check{r}(p,f(x_0),\xi).
  \end{cases}
$$
For any $p$ in $P$, we consider the open definable subset
 $$\mathcal C_{x_0,\eta_0,p} = \bigcup_{\xi \in P_{x_0,\eta_0}}B(\xi, \check{r}_{f(x_0)}(p)).$$
By assumption, for any $\xi$ in $\Bn$, thanks to the inequality
$ \check{r}_{f(x_0)}(p) \geq \ord \xi$,
the definable set $\mathcal C_{x_0,\eta_0,p}$ is contained in $\Bn$.
We denote by $C_{x_0,\eta_0,p}^{c}$ the closed subset $\Bn \setminus \mathcal C_{x_0,\eta_0,p}$.
We consider the definable set
$$\mathcal E_{\eta_0} = \{ \xi \in \Lambda_\n C^c_{x_0,\eta_0,p} \mid
\ord \eta_0 - N_\delta \leq \ord \xi \leq \ord \eta_0 - \min_{x \in B(x_0,r_{x_0,\eta_0}(p))} \:\ord \: ^{t}df(x)
\}.$$
By the previous remark, the definable function $\xi \mapsto \ord (^{t}df(x_0)\xi -\eta)$ takes finite values on the set $\mathcal E_{\eta_0}$ and by definable compactness this function takes finitely many values. In particular it has a maximum denoted by $M_{\eta_0}$.
We define $r_{x_0,\eta_0}(p)$ as $R_{x,P}(x_0)$ and $\check{r}_{x_0,\eta_0}(p)$ as $\ord \eta_0 +1$.
We consider the continuous and definable function
 $$
 N_{x_0,\eta_0} \colon
\begin{cases}
 \mathcal B_{x_0,\eta_0}  \longrightarrow \mathbb Z \\
 (p,x',r')  \longmapsto  \min (\mathcal N_{p,f(x_0)},\mathcal C_{x_0,\eta_0,x',r',p})
 \end{cases}$$
 defined on $$\mathcal B_{x_0,\eta_0}=\{(p,x',r')\in P \times V_{x'} \times \mathbb Z \mid B(x',r')\subset B(x_0,r_{x_0,\eta_0}(p))\}.$$
 with for any $(p,x',r')$ in $\mathcal B_{x_0,\eta_0}$
 $$ \mathcal C_{x_0,\eta_0,x',r',p} = \min \left(
 -\ord \eta_0 -\max (\ord \eta_0 -N_{R}+1,r'),
 -M_{\eta_0}-\max (M_{\eta_0} -N_{R}+1,r')\right).$$

We obtain that the point $(x_0,\eta_0)$ does not belong to $\WF_{\Lambda_n}(f^*u)$ proving the equality
$$<(f^*u)_{\Lambda_\n \times \D_{x_0,\eta_0}},T>1_{\E_{x_0,\eta_0}}
 1_{B_{N_{x_0,\eta_0}}} =
 <(f^*u)_{\Lambda_\n \times \D_{x_0,\eta_0}},T>1_{\E_{x_0,\eta_0}},$$
 where $\D_{x_0,\eta_0}$ is the product $P \times V_{x'} \times \mathbb Z \times V_{\eta},$
$$\E_{x_0,\eta_0}=\left\{\left((p,x',r'),\xi \right) \in \mathcal B_{x_0,\eta_0} \times V_{\xi} \: \mid \:
 \xi \in B(\xi_0,\check{r}_{x_0,\eta_0}(p)) \right\},$$
$T_{x_0,\xi_0}$ is the Schwartz-Bruhat function in
$S_{\Lambda_\n \times \D_{x_0,\eta_0}}(\Lambda_\n \times \D_{x_0,\eta_0} \times V)$ defined by
 $$T_{x_0,\eta_0}(\lambda,\left(p,x',r',\xi \right),x)=
 1_{B(x',r')}(x)E(x\mid \lambda \xi),$$
 and
 $$B_{N_{x_0,\eta_0}}=\{(\lambda,\left(p,x',r' \right),\xi)\in \Lambda_\n \times \mathcal B_{x_0,\eta_0} \times V_\xi \mid
 \ord \lambda \geq N_{x_0,\eta_0}(p,x',r')\}.$$

By the equality $r_{x_0,\eta_0} = R_{x,P}(x_0)$ we can apply the localization point (Remark \ref{localization}) and obtain
$$<(f^*u)_{\Lambda_\n \times \D_{x_0,\eta_0}},T>(\lambda,p,x',r',\eta) = $$
$$\int_{\xi \in V_\xi}<u_{P\times V_\xi},\chi(p,x_0,\cdot) E(\cdot\mid \xi)>
\int_{s\in V_s}1_{B(x',r')}(s)E((s\mid \lambda\eta)-(f(s)\mid \xi))dsd\xi.
$$

By the change of variables formula we obtain
$$<(f^*u)_{\Lambda_\n \times \D_{x_0,\eta_0}},T>(\lambda,p,x',r',\eta) = $$
$$\mathbb L^{-\ord \lambda m}
\int_{\xi \in V_\xi}<u_{P\times V_\xi},\chi(p,x_0,\cdot) E(\cdot\mid \lambda \xi)>
\int_{s\in V_s}1_{B(x',r')}(s)E(\lambda \delta(s,\xi,\eta))dsd\xi.$$

where $\delta(s,\xi,\eta)=(s\mid \eta)-(f(s)\mid \xi)$.\\

By \ref{eqTaylor} we have the equality

\begin{equation}  \label{I}
\int_{s\in V_s}1_{B(x',r')}(s)E(\lambda \delta(s,\xi,\eta))ds
\end{equation}
$$= E(\lambda \delta(x_0,\xi,\eta))
\int_{s'\in V_s}1_{B(x',r')}(x_0+s')
E(\lambda \delta( ^{t}df(x_0)\xi -\eta \mid s'))E(R_f(x_0,s',\xi)s'\mid s')ds.
$$
By lemma \ref{definablecompactness}, for any $p$ in $P$, the definable and continuous function
$\ord ^{t}df$ admits a minimum on the ball $B(x_0,r_{x_0,\eta_0})$ \footnote{in the case where $df$ vanishes on this ball, it is enough to apply the definable compactness on the intersection between
$B(x_0,r_{x_0,\eta_0})$ and $\ord^{t}df \leq \alpha$, for any constant $\alpha$.}.

\begin{enumerate}
\item Let $\xi$ be in $V_{\xi}$ with
$\ord \xi > \ord \eta_0 - \min_{x \in B(x_0,r_{x_0,\eta_0}(p))} \ord ^{t}df(x)$
we have $$\ord ^{t}df(x)(\xi)-\eta = \ord \eta = \ord \eta_0 $$
then $\lambda \ra I_{x',r'}(\lambda,\eta,\xi)$ has a bounded support in the ball
$B_{C_{x_0,\eta_0,x',r',p}}$.
\item Let $\xi$ be in $\Lambda_\n C_{x_0,\eta,p}$ such that
$\ord \xi \leq \ord \eta_0 - \min_{x \in B(x_0,r_{x_0,\eta_0}(p))} \ord ^{t}df(x)$, then
$$\lambda \longmapsto <u_{P \times V_\xi},\chi(p,x_0,\cdot) E(\cdot\mid \lambda \xi)> 1_{\Lambda_\n \mathcal C_{x_0,\eta_0,p} }(\lambda\xi)$$
 has support in the ball
 $B_{\mathcal N(p,f(x_0))
 -\ord \eta_0 + \min_{x \in B(x_0,r_{x_0,\eta_0}(p))} \ord ^{t}df(x)}$.\\
 Indeed, there is $\mu$ in $\Lambda_n$ and $\tilde{\xi}$ in $C_{x_0,\eta,p}$ such that
 $\mu \tilde{\xi} = \xi$. Note that
 $$\ord \mu \leq \ord \xi \leq  \ord \eta_0 - \min_{x \in B(x_0,r_{x_0,\eta_0}(p))} \ord ^{t}df(x)$$
 then, for any $\lambda$ in $\Lambda_n$, if
 $$\ord \lambda \leq \mathcal N(p,f(x_0))
 -\ord \eta_0 + \min_{x \in B(x_0,r_{x_0,\eta_0}(p))} \ord ^{t}df(x)$$
 then $\ord \lambda \mu \leq \mathcal N(p,f(x_0))$ and
 $$<u_{P \times V_\xi},\chi(p,x_0,\cdot) E(\cdot\mid \lambda \xi)> =
 <u_{P \times V_\xi},\chi(p,x_0,\cdot) E(\cdot\mid (\lambda \mu) \tilde{\xi})> = 0.$$
 \item Let $\xi$ be in $\Lambda_\n C^c_{x_0,\eta_0,p}$ such that
 $\ord \xi \leq \ord \eta_0 - \min_{x \in B(x_0,r_{x_0,\eta_0}(p))} \ord ^{t}df(x)$.

 There is $\mu$ in $\Lambda_n$ and $\tilde{\xi}$ in $C_{x_0,\eta,p}^c$ such that
 $\mu \tilde{\xi} = \xi$.
 By construction of $C_{x_0,\eta_0,p}^{c}$,
 the point $(f(x_0),\tilde \xi)$ does not belong in $A$ then by assumption on $f$ and $A$, we have the relation $\ord \:^{t}df(x_0)\tilde{\xi} \leq N_{\delta}$
 \begin{enumerate}
\item If $\ord \xi < - N_{\delta} +  \ord \eta_0$, then $\ord \mu \leq -N_{\delta} + \ord \eta_0 -\n$ and we have
$$\ord (\:^{t}df(x_0)\xi -\eta) = \ord \:^{t}df(x_0)\xi \leq \ord \eta_0$$
\item If $\xi$ belongs to the bounded and closed set
$\mathcal E_{\eta_0}$ then
$$\ord (\:^{t}df(x_0)\xi -\eta) = \ord \:^{t}df(x_0)\xi \leq M_{\eta_0}$$
\end{enumerate}
In these cases, by Proposition \ref{oscillante},
$\lambda \ra I_{x',r'}(\lambda,\eta,\xi)$ has a bounded support in the ball
$B_{C_{x_0,\eta_0,x',r',p}}$.
\end{enumerate}
\end{proof}
\subsection{Tensor product and product of definable distributions}
Let $P$ be a definable set in $\Def_k$ and $d_x$ and $d_y$ be two positive integers.
Let $V_x$ and $V_y$ be the definable sets
$h[d_x,0,0]$ and $h[d_y,0,0]$. Let $(u_{\Phi_W})$ in $S'_{P}(V_x)$ and $(v_{\Phi_W})$
in $S'_{P}(V_y)$ be two definable distributions.
\begin{lem}
For any definable set $W$ in $\Def_P$, for any Schwartz-Bruhat function $\varphi$ in
 $S_{W}(W \times V_x \times V_y)$ the constructible exponential function $\psi$ equal to
 $<v_{\Phi_W \times V_x},\varphi>$ in $\mathcal C(W\times V_x)^{\exp}$
 is a Schwartz-Bruhat function in
 $S_{W}(W\times V_x)$ with
 $$\alpha^{+}(\psi) = \alpha^{+}(\varphi)\circ \pi_{W}\:\:\mathrm{and}\:\:
 \alpha^{-}(\psi) = \alpha^{-}(\varphi)\circ \pi_{W}.
 $$
 Furthermore, for any definable map $g$ from $W$ to $W'$, if $\varphi$ is
 $(g\times \Id_{V_{x,y}})$-compatible then, $<v_{\Phi_W \times V_x},\varphi>$ is
 $(g\times \Id_{V_x})$-compatible.
\end{lem}
\begin{proof}
As $\varphi$ is a Schwartz-Bruhat function in $S_{W}(W \times V_x \times V_y)$, using Fubini, we observe it is also a Schwartz-Bruhat function in
$S_{W\times V_x}(W \times V_x \times V_y)$ with same data. In particular, this Schwartz-Bruhat function is $(\pi_W \times \Id_{V_y})$-compatible. Then, for any definable function $\alpha$ from $W\times V_x$ to $\mathbb Z$ satisfying $\alpha \leq \alpha^{-}(\varphi)$, the equalities
$$\varphi 1_{B_{\alpha^{-}(\varphi)}} = \varphi\:\:\mathrm{and}\:\: 1_{B_{\alpha}}1_{B_{\alpha^{-}(\varphi)}} = 1_{B_{\alpha^{-}(\varphi)}}$$
imply the equality
$$<v_{\Phi_W \times V_x},\varphi> 1_{B_{\alpha}} = <v_{\Phi_W \times V_x}, \varphi>.$$
As $\varphi$ is $(\pi_W\times \Id_{V_y})$-integrable, by the push-forward assumption on $v$,
the constructible exponential function $<v_{\Phi_W \times V_x},\varphi>$ is $W$-integrable then using the pushforward relations on $v$ on the convolution product definition we obtain the equality
$$
<v_{\Phi_W \times V_x},\varphi> * 1_{B_{\alpha^+(\varphi)}} =
<v_{\Phi_W \times V_{x}},\varphi*1_{B_\alpha^{+}(\varphi)}> =
\mathbb L^{\alpha^{+}(\varphi)d_x} <v_{\Phi_W \times V_{x}},\varphi>
$$
and as in the proof of Lemma \ref{pull-backSB} we obtain the same relation for any definable function $\alpha$ from $W\times V_y$ to $Z$ which is bigger than $\alpha^{+}(\varphi)$. We conclude that $<v_{\Phi_W \times V_x},\varphi>$ is a Schwartz-Bruhat function in $S_{W}(W\times V_x)$.

 Let $g$ be a definable map from $W$ to $W'$ and assume $\varphi$ is
 $(g\times \Id_{V_{x,y}})$-compatible. Then by assumption $\varphi$ is
 $(g\times \Id_{V_{x,y}})$-integrable and there are definable functions
 $\beta^{+}$ and $\beta^{-}$ from $W'$ to $\mathbb Z$ such that
 $\alpha^{+}(\varphi)=\beta^{+}\circ g$ and $\alpha^{-}(\varphi)=\beta^{-}\circ g$.
 In particular $\varphi$ is $(g\times \Id_{V_x})\times \Id_{V_y}$ compatible
 then using the push-forward relation on $v$, $\psi$ equal to $<v_{\Phi_W \times V_x},\varphi>$ is  $(g\times \Id_{V_x})$-compatible.
 \end{proof}

\begin{defn}[Tensor product of definable distributions]  \label{produit-tensoriel}
Let $P$ be a definable set in $\Def_k$ and $d_x$ and $d_y$ be two positive integers.
Let $V_x$ and $V_y$ be the definable sets
$h[d_x,0,0]$ and $h[d_y,0,0]$. Let $(u_{\Phi_W})$ in $S'_{P}(V_x)$ and $(v_{\Phi_W})$
in $S'_{P}(V_y)$ be two definable distributions.
If for any definable function $\Phi_W$ from $W$ to $P$ and for any Schwartz-Bruhat function $\varphi$ in $S_{W}(W\times V_{x,y})$ we have the equality
$$<u_{\Phi_W},<v_{\Phi_W \times V_x},\varphi>> = <v_{\Phi_W},<u_{\Phi_W \times V_y},\varphi>>$$
then we define
$$<(u\otimes v)_{\Phi_W},\varphi> := <u_{\Phi_W},<v_{\Phi_W \times V_x},\varphi>>.$$
It is a definable distribution in $S'_{P}(V_x\times V_y)$.
\end{defn}
\begin{example}
If $P$ is a point and if $f$ and $g$ are locally integrable functions respectively on $V_x$
and $V_y$, then $f\otimes g$ is the usual tensor product on $V_x\times V_y$ and the previous equality follows from Fubini.
\end{example}
\begin{proof}
By the previous lemma, the above evalutions make sense.
Let $g:W \ra W'$ be a definable map and $\varphi$ be a Schwartz-Bruhat function in
$S_{W}(W\times V_{x,y})$. Using the pull-back relations on $u$ and $v$ we have

\begin{align*}
 g^{*}<(u\otimes v)_{\Phi_W},\varphi> &=
 g^{*}<u_{\Phi_{W'}},<v_{\Phi_{W'}\times V_x},\varphi>> \\
 &= <u_{\Phi_{W}},(g\times \Id_{V_x})^{*}(<v_{\Phi_{W'}\times V_x},\varphi>)> \\
 &= <u_{\Phi_{W}},<v_{\Phi_{W'}\times V_x},(g\times \Id_{V_x})^{*}\varphi>>.
\end{align*}
If $\varphi$ is $(g\times \Id_{V_{x,y}})$-compatible then by the previous lemma,
$<v_{\Phi_W \times V_x},\varphi>$ is also $(g\times \Id_{V_x})$-compatible and
using the push-forward relations on $u$ and $v$ we have
\begin{align*}
   g_{!}<(u\otimes v)_{\Phi_W},\varphi> &=
   g_{!}<u_{\Phi_W},<v_{\Phi_W \times V_x},\varphi>> \\
   &= <u_{\Phi_{W'}},(g\times \Id_{V_x})_{!}<v_{\Phi_W \times V_x},\varphi>> \\
   &= <u_{\Phi_{W'}},<v_{\Phi_{W'} \times V_x},(g\times \Id_{V_x})_{!}\varphi>>.
  \end{align*}
\end{proof}
As in the real setting and in the $p$-adic setting we have

\begin{prop}[Wave front set of $u \otimes v$]
Let $P$ be a definable set in $\Def_k$ and $d_x$ and $d_y$ be two positive integers.
Let $V_x$ and $V_y$ be the definable sets
$h[d_x,0,0]$ and $h[d_y,0,0]$. Let $(u_{\Phi_W})$ in $S'_{P}(V_x)$ and $(v_{\Phi_W})$
in $S'_{P}(V_y)$ be two definable distributions.
Let $\n$ be a positive integer and $\Lambda_\n$ the corresponding subgroup of $h[1,0,0]^{\times}$.
If the tensor product $u\otimes v$ exists then,
$$
\WF_{\Lambda_\n}(u\otimes v)  \subset 
\left(\begin{array}{cl}
                             & \{(x,\xi,y,\eta) \mid (x,\xi) \in \WF_{\Lambda_\n}(u)\:\mathrm{and}\: (y,\eta) \in \WF_{\Lambda_\n}(v) \} \\
                             \cup & \{(x,\xi,y,0)\mid (x,\xi) \in \WF_{\Lambda_n}(u)\:\mathrm{and}\: y\in \SS(v)\} \\
                             \cup & \{(x,0,y,\eta)\mid (y,\eta) \in \WF_{\Lambda_n}(v)\:\mathrm{and}\:x\in \SS(u)\}.
\end{array}
\right)
$$
\end{prop}

\begin{proof}
 We use in this proof notation of Definition \ref{microlocally smooth point}.
 If $(x_0,\xi_0)$ and $(y_0,\eta_0)$ does not belong to $\WF_{\Lambda_\n}u$ and
 $\WF_{\Lambda_\n}v$ with $\xi \neq 0$ and $\eta \neq 0$ then, there are definable functions
 $r_{x_0,\xi_0}$, $\check{r}_{x_0,\xi_0}$, $r_{y_0,\eta_0}$, $\check{r}_{y_0,\eta_0}$ from $P$ to $\mathbb Z$ and
 definable and continuous functions
 $$N_{x_0,\xi_0} : B_{x_0,\xi_0} \ra \mathbb Z,\:\:
 N_{y_0,\eta_0} : B_{y_0,\eta_0} \ra \mathbb Z
 $$
 satisfying
\begin{equation} \label{u}
(<u_{\Lambda_\n \times \D_{x_0,\xi_0}},T_{x_0,\xi_0}>1_{\E_{x_0,\xi_0}})1_{B_{N_{x_0,\xi_0}}} =
<u_{\Lambda_\n \times \D_{x_0,\xi_0}},T_{x_0,\xi_0}>1_{\E_{x_0,\xi_0}}
\end{equation}
and
\begin{equation} \label{v}
(<v_{\Lambda_\n \times \D_{y_0,\eta_0}},T_{y_0,\eta_0}>1_{\E_{y_0,\eta_0}})1_{B_{N_{y_0,\eta_0}}} =
<v_{\Lambda_\n \times \D_{y_0,\eta_0}},T_{y_0,\eta_0}>1_{\E_{y_0,\eta_0}}
\end{equation}
We consider
$$
r_{(x_0,y_0),(\xi_0,\eta_0)} = \max(r_{x_0,\xi_0}, r_{y_0,\eta_0}),\:
\check{r}_{(x_0,y_0),(\xi_0,\eta_0)} = \max(\check{r}_{x_0,\xi_0}, \check{r}_{y_0,\eta_0})
$$
and
$$N_{(x_0,y_0),(\xi_0,\eta_0)} = \max (N_{(x_0,\xi_0)},N_{(y_0,\eta_0)}).$$
By definition of the tensor, by previous identities \ref{u} and \ref{v}, and by the equivalencies
$$B((x',y'),r') \subset B((x_0,y_0),r_{(x_0,y_0),(\xi_0,\eta_0)}) $$
$$\Leftrightarrow
B(x',r') \subset B(x_0,r_{(x_0,y_0),(\xi_0,\eta_0)}) \:\mathrm{and} \:
B(y',r') \subset B(y_0,r_{(x_0,y_0),(\xi_0,\eta_0)})$$
and
$$
\ord \lambda \geq N((x_0,y_0),(\xi_0,\eta_0))(x',y',r') \Leftrightarrow
\ord \lambda \geq N(x_0,\xi_0)(x',r') \:\: \land \: \:
\ord \lambda \geq N(y_0,\eta_0)(y',r')
$$
namely
$$ 1_{B((x_0,y_0),(\eta_0,\xi_0))}(\lambda,p,x',y',r') =
1_{B(x_0,\xi_0)}(\lambda,p,x',r')1_{B(y_0,\eta_0)}(\lambda,p,y',r').$$
we
obtain
$$
(<(u\otimes v)_{\Lambda_\n \times \D_{(x_0,y_0),(\xi_0,\eta_0)}},
T_{(x_0,y_0),(\xi_0,\eta_0)}>1_{\E_{(x_0,y_0),(\xi_0,\eta_0)}})
1_{B_{N_{(x_0,y_0),(\xi_0,\eta_0)}}} $$
$$ =
<(u\otimes v)_{\Lambda_\n \times \D_{(x_0,y_0),(\xi_0,\eta_0)}},T_{(x_0,y_0),(\xi_0,\eta_0)}>
1_{\E_{(x_0,y_0),(\xi_0,\eta_0)}}
$$
which means that the point $((x_0,y_0),(\xi_0,\eta_0))$ does not belong to
$\WF_{\Lambda_\n}(u\otimes v)$.

If $(x_0,\xi_0)$ does not belong to $\WF_{\Lambda_\n}(u)$ and $y_0$ is a smooth point, then as before there are data $r_{x_0,\xi_0}$, $\check{r}_{x_0,\xi_0}$, $N_{x_0,\xi_0}$ satisfying
\ref{u}. By Definition \ref{singular-support}, there are also a definable function $r_{y_0}$ and a Schwartz-Bruhat function
$\psi$ in $S_{P}(P\times V_y)$ such that $1_{B(y_0,r_{y_0})}u$ and
$1_{B(y_0,r_{y_0})}\psi$ are equal. Then using Example \ref{cas-SB} and same ideas as before we obtain that $((x_0,y_0),\xi_0)$ does not belong to $\WF_{\Lambda_\n}(u \otimes v)$.
The last case is similar.
\end{proof}

As in the real case and in the $p$-adic case, using the construction of the tensor product
and the pull-back Theorem \ref{thmf^{*}u} we obtain locally the definition of the product of two definable distributions
\begin{defn}[Product of definable distributions]  \label{produit}
Let $\n\geq 1$. Let $u$ and $v$ be two definable distributions in $S_{P}'(V)$ and denote by $i$ the diagonal injection
$$ i : \Delta = \{(x,y)\in V^2 \mid x=y\} \ra V\times V.$$
If the tensor product $u\otimes v$ exists,
if $(P\times A,r,\check{r},N)$ is a $\Lambda$-microlocally smooth data of $u\otimes v$ such that
\begin{equation} \label{inclusion}
A \subset V^{2}\times \Bn^2 \setminus
\left(\WF_{\Lambda_\n}(u\otimes v) \cup \{(x,y),(\xi,\eta) \mid \xi=-\eta\}\right)
\end{equation}
then we define a product $u\cdot v$ in $S'_{P}(\pi(A) \cap \Delta)$ equal to $i^{-1}(u\otimes v)$.

In particular, we have the inclusions
$$\WF_{\Lambda_n}(u\cdot v) \subset i^{*}(\Lambda_{\n} A^{c})$$
and
$$\{(x,\zeta)\mid \exists \xi, \eta\,\; \zeta = \xi + \eta,\: (x,\xi)\in \WF_{\Lambda_\n}(u),
(x,\eta)\in \WF_{\Lambda_\n}(v)\} \subset i^{*}(\Lambda_{\n} A^{c}).$$
\end{defn}

\begin{proof}
The construction follows from Theorem \ref{thmf^{*}u} applied to $u\otimes v$.
Note that $^{t}di(x_0)(\xi,\eta)=\xi+\eta$, then, the inclusion \ref{inclusion} follows
from the assumption 2 of \ref{thmf^{*}u}.
\end{proof}

\subsection*{Acknowledgments}
We are very grateful to Fran\c{c}ois Loeser who proposed us a postdoctoral fellowship on that topic at the Institut Math\'ematique de Jussieu and for his guidance all along the project. We would like also thank Julia Gordon, Immanuel Halupczok and Yimu Yin for inspiring discussions and specially Raf Cluckers for his kindness and all his technical advices.
This work is partially supported by ANR-15-CE40-0008 (Defigeo) and  by  the
ERC Advanced Grant NMNAG.

\bibliographystyle{plain}
%\bibliography{/home/mraib/Seafile/Recherche/biblio_complete.bib}
%\bibliography{/home/michel/Seafile/Recherche/biblio_complete.bib}
\bibliography{biblio_complete}

\begin{thebibliography}{10}

\bibitem{Sato69}
{\em Hyperfunctions and partial differential equations. Proceedings of the
  {I}nternational {C}onference on {F}unctional {A}nalysis and {R}elated
  {T}opics ({T}okyo, {A}pril, 1969)}.
\newblock Organized by the Mathematical Society of Japan under the
  co-sponsorship of the International Mathematical Union and the Science
  Council of Japan. University of Tokyo Press, Tokyo, 1970.

\bibitem{AizDri}
Avraham Aizenbud and Vladimir Drinfeld.
\newblock The wave front set of the {F}ourier transform of algebraic measures.
\newblock {\em Israel J. Math.}, 207(2):527--580, 2015.

\bibitem{CR}
Jorge Cely and Michel Raibaut.
\newblock On the commutativity of pull-back and push-forward functors on
  motivic constructible functions.
\newblock arXiv:1811.06850.

\bibitem{CluComLoe12}
Raf Cluckers, Georges Comte, and Fran{\c{c}}ois Loeser.
\newblock Local metric properties and regular stratifications of {$p$}-adic
  definable sets.
\newblock {\em Comment. Math. Helv.}, 87(4):963--1009, 2012.

\bibitem{CluGorHal14a}
Raf Cluckers, Julia Gordon, and Immanuel Halupczok.
\newblock Integrability of oscillatory functions on local fields: transfer
  principles.
\newblock {\em Duke Math. J.}, 163(8):1549--1600, 2014.

\bibitem{CluHalLoe11}
Raf Cluckers, Thomas Hales, and Fran\c{c}ois Loeser.
\newblock {Transfer principle for the fundamental lemma}.
\newblock In {\em {On the stabilization of the trace formula}}, volume~1 of
  {\em {Stab. Trace Formula Shimura Var. Arith. Appl.}}, pages 309--347. Int.
  Press, Somerville, MA, 2011.

\bibitem{CH}
Raf Cluckers and Immanuel Halupczok.
\newblock Integration of functions of motivic exponential class, uniform in all
  non-archimedean local fields of characteristic zero.
\newblock {\em J. \'{E}c. polytech. Math.}, 5:45--78, 2018.

\bibitem{CHLR}
Raf Cluckers, Immanuel Halupczok, Fran\c{c}ois Loeser, and Michel Raibaut.
\newblock Distributions and wave front sets in the uniform non-archimedean
  setting.
\newblock {\em Trans. London Math. Soc.}, 5(1):97--131, 2018.

\bibitem{CluLoe04a}
Raf Cluckers and Fran\c{c}ois Loeser.
\newblock {Fonctions constructibles et int{\'e}gration motivique. {I}}.
\newblock {\em C. R. Math. Acad. Sci. Paris}, 339(6):411--416, 2004.

\bibitem{CluLoe04b}
Raf Cluckers and Fran\c{c}ois Loeser.
\newblock {Fonctions constructibles et int{\'e}gration motivique. {II}}.
\newblock {\em C. R. Math. Acad. Sci. Paris}, 339(7):487--492, 2004.

\bibitem{CluLoe05a}
Raf Cluckers and Fran\c{c}ois Loeser.
\newblock {Ax-{K}ochen-{E}r\v sov theorems for {$p$}-adic integrals and motivic
  integration}.
\newblock In {\em {Geometric methods in algebra and number theory}}, volume 235
  of {\em {Progr. Math.}}, pages 109--137. Birkh{\"a}user Boston, Boston, MA,
  2005.

\bibitem{CluLoe05b}
Raf Cluckers and Fran\c{c}ois Loeser.
\newblock {Fonctions constructibles exponentielles, transformation de {F}ourier
  motivique et principe de transfert}.
\newblock {\em C. R. Math. Acad. Sci. Paris}, 341(12):741--746, 2005.

\bibitem{CluLoe08a}
Raf Cluckers and Fran\c{c}ois Loeser.
\newblock {Constructible motivic functions and motivic integration}.
\newblock {\em Invent. Math.}, 173(1):23--121, 2008.

\bibitem{CluLoe10a}
Raf Cluckers and Fran\c{c}ois Loeser.
\newblock {Constructible exponential functions, motivic {F}ourier transform and
  transfer principle}.
\newblock {\em Ann. of Math. (2)}, 171(2):1011--1065, 2010.

\bibitem{Denef84}
J.~Denef.
\newblock The rationality of the {P}oincar\'e series associated to the
  {$p$}-adic points on a variety.
\newblock {\em Invent. Math.}, 77(1):1--23, 1984.

\bibitem{DenLoe99a}
Jan Denef and Fran\c{c}ois Loeser.
\newblock {Germs of arcs on singular algebraic varieties and motivic
  integration}.
\newblock {\em Invent. Math.}, 135(1):201--232, 1999.

\bibitem{DenLoe01a}
Jan Denef and Fran\c{c}ois Loeser.
\newblock {Definable sets, motives and {$p$}-adic integrals}.
\newblock {\em J. Amer. Math. Soc.}, 14(2):429--469 (electronic), 2001.

\bibitem{Duistermaat}
J.J. {Duistermaat}.
\newblock {Fourier integral operators.}
\newblock {New York: Courant Institute of Mathematical Science, New York
  University. III, 190 p. (1973).}, 1973.

\bibitem{Forey}
Arthur Forey.
\newblock Motivic local density.
\newblock {\em Math. Z.}, 287(1-2):361--403, 2017.

\bibitem{Gabor}
Akiva {Gabor}.
\newblock {Remarks on the wave front of a distribution.}
\newblock {\em {Trans. Am. Math. Soc.}}, 170:239--244, 1972.

\bibitem{GorYaf09}
Julia Gordon and Yoav Yaffe.
\newblock {An overview of arithmetic motivic integration}.
\newblock In {\em {Ottawa lectures on admissible representations of reductive
  {$p$}-adic groups}}, volume~26 of {\em {Fields Inst. Monogr.}}, pages
  113--149. Amer. Math. Soc., Providence, RI, 2009.

\bibitem{Hei85a}
D.~B. Heifetz.
\newblock {{$p$}-adic oscillatory integrals and wave front sets}.
\newblock {\em Pacific J. Math.}, 116(2):285--305, 1985.

\bibitem{Hormander71}
Lars H{\"o}rmander.
\newblock Fourier integral operators. {I}.
\newblock {\em Acta Math.}, 127(1-2):79--183, 1971.

\bibitem{Hormander83}
Lars {H\"ormander}.
\newblock {The analysis of linear partial differential operators. I:
  Distribution theory and Fourier analysis.}
\newblock {Grundlehren der Mathematischen Wissenschaften, 256. Berlin
  Heidelberg-New York - Tokyo: Springer-Verlag. IX, 391 p. DM 98.00; \$ 39.20
  (1983).}, 1983.

\bibitem{Howe79}
Roger Howe.
\newblock Wave front sets of representations of {L}ie groups.
\newblock In {\em Automorphic forms, representation theory and arithmetic
  ({B}ombay, 1979)}, volume~10 of {\em Tata Inst. Fund. Res. Studies in Math.},
  pages 117--140. Tata Inst. Fundamental Res., Bombay, 1981.

\bibitem{HruKaz06}
Ehud Hrushovski and David Kazhdan.
\newblock Integration in valued fields.
\newblock In {\em Algebraic geometry and number theory}, volume 253 of {\em
  Progr. Math.}, pages 261--405. Birkh\"auser Boston, Boston, MA, 2006.

\bibitem{HruKaz08}
Ehud Hrushovski and David Kazhdan.
\newblock The value ring of geometric motivic integration, and the {I}wahori
  {H}ecke algebra of {${\rm SL}_2$}.
\newblock {\em Geom. Funct. Anal.}, 17(6):1924--1967, 2008.
\newblock With an appendix by Nir Avni.

\bibitem{Igusabook}
Jun-ichi Igusa.
\newblock {\em An introduction to the theory of local zeta functions},
  volume~14 of {\em AMS/IP Studies in Advanced Mathematics}.
\newblock American Mathematical Society, Providence, RI; International Press,
  Cambridge, MA, 2000.

\bibitem{Kashiwara-Schapira}
Masaki Kashiwara and Pierre Schapira.
\newblock {\em Sheaves on manifolds}, volume 292 of {\em Grundlehren der
  Mathematischen Wissenschaften [Fundamental Principles of Mathematical
  Sciences]}.
\newblock Springer-Verlag, Berlin, 1994.
\newblock With a chapter in French by Christian Houzel, Corrected reprint of
  the 1990 original.

\bibitem{Kon95a}
Kontsevich.
\newblock {Lecture at Orsay}.
\newblock D{\'e}cembre 7, 1995.

\bibitem{Loe11a}
Fran\c{c}ois Loeser.
\newblock {Microlocal geometry and valued fields}.
\newblock {\em Publ. Res. Inst. Math. Sci.}, 47(2):613--627, 2011.

\bibitem{Pas89}
Johan Pas.
\newblock {Uniform {$p$}-adic cell decomposition and local zeta functions}.
\newblock {\em J. Reine Angew. Math.}, 399:137--172, 1989.

\bibitem{Presb29}
M.~Presburger.
\newblock {{\"U}ber die Vollst{\"a}ndigkeit eines gewissen Systems des
  Arithmetik ganzer Zahlen}.
\newblock In {\em { Comptes-rendus du I Congr{\`e}s des Math{\'e}maticiens des
  Pays Slaves, pp. 92--101. Warsaw (1929) }}.

\bibitem{SaKaKa73}
Mikio Sato, Takahiro Kawai, and Masaki Kashiwara.
\newblock Microfunctions and pseudo-differential equations.
\newblock In {\em Hyperfunctions and pseudo-differential equations ({P}roc.
  {C}onf., {K}atata, 1971; dedicated to the memory of {A}ndr\'e {M}artineau)},
  pages 265--529. Lecture Notes in Math., Vol. 287. Springer, Berlin, 1973.

\bibitem{VDDries89}
Lou van~den Dries.
\newblock {Dimension of definable sets, algebraic boundedness and {H}enselian
  fields}.
\newblock {\em Ann. Pure Appl. Logic}, 45(2):189--209, 1989.
\newblock Stability in model theory, II (Trento, 1987).

\bibitem{Yimu-selecta}
Yimu Yin.
\newblock Fourier transform of the additive group in algebraically closed
  valued field.
\newblock {\em Sel. Math. New Ser}, 20(4):1111--1157, 2014.

\end{thebibliography}
\end{document}